%% file: curvature_and_augmentation13__arxiv_submission_.tex
\documentclass{amsart} \usepackage{tikz-cd, calc}
\input{definitions}
\input{boxdiagrams}

\numberwithin{equation}{thm}

\makeatletter
  \tikzset{
  edge node/.code={%
      \expandafter\def\expandafter\tikz@tonodes\expandafter{\tikz@tonodes #1}}}
\makeatother
\tikzset{
  subseteq/.style={
    draw=none,
    edge node={node [sloped, allow upside down, auto=false]{$\subseteq$}}},
  Subseteq/.style={
    draw=none,
    every to/.append style={
      edge node={node [sloped, allow upside down, auto=false]{$\subseteq$}}}
  }
}

\begin{document}
\title{Strictly Unital \Ai-algebras} \author{Jesse Burke}
\address{Mathematical Sciences Institute\\ Australian National University\\ Canberra ACT}
\email{jesse.burke@anu.edu.au}
\begin{abstract}
Given a graded module over a commutative ring, we define a dg-Lie
algebra whose Maurer-Cartan elements are the strictly unital
\Ai-algebra structures on that module. We use this to generalize
Positselski's result that a curvature term on the bar
construction compensates for a lack of augmentation, from a field to arbitrary commutative base ring. We also use this
 to show that the reduced Hochschild cochains control the strictly
 unital deformation functor. We motivate these results by giving a full development of the
deformation theory of a nonunital \Ai-algebra. 
\end{abstract}

\maketitle
\setcounter{section}{-1}

\tikzset{every picture/.append style={scale=0.75}, every node/.style={transform shape}}

\tikzset{baseline={(current bounding box.center)}}
\pgfkeys{unit/.style = {color = black!20!red, arm style = densely
    dashed,}}
\pgfkeys{unit output/.style = {output color = black!20!red, output style =
    dashed},}
\pgfkeys{decorate/.style = {color = gray}, }
\pgfkeys{fake decorate/.style = {color = white}, }
\pgfkeys{module/.style = {right arm color = black!20!orange, right arm
    decoration = snake,output color = black!20!orange, output decoration =
    snake},} 
\pgfkeys{secmodule/.style = {right arm color = black!60!green,
    right arm decoration = snake,output color = black!60!green, output
    decoration = {snake}, },} 
\pgfkeys{thirdmodule/.style = {right arm color = purple,
    right arm decoration = snake,output color = purple, output
    decoration = snake},} 
\pgfkeys{modulemap/.style = {right arm color
    = black!20!orange, right arm decoration = snake,output color = black!60!green,
    output decoration = snake},}
\pgfkeys{secmodulemap/.style = {right arm color
    = green, right arm decoration = snake,output color = purple,
    output decoration = snake},}
\pgfkeys{end output/.style = {output color = black!20!red,
    output style = densely dash dot},} 
\pgfkeys{hom output/.style = {output color = black!20!blue,
    output decoration = bumps, arm dec amp = 2},} 
\pgfkeys{parameter/.style = {output color = black!20!red, color = black!20!red, arm decoration = coil, arm dec
    amp = 2, arm dec seg length = .1cm}}
\pgfkeys{secparameter/.style = {output color = black!20!green, arm decoration = coil, arm dec
    amp = 2, arm dec seg length = .1cm},}
\pgfkeys{bmod/.style = {output color = blue, output style = dotted, output width = .5mm}}

\tikzset{unit/.style = {color = black!20!red, densely dashed, decorate, very
  thick}}
\tikzset{module/.style = {color = black!20!orange, decorate,
    decoration = snake, very thick}} 
\tikzset{secmodule/.style =
  {color =black!60!green, decorate, decoration = snake, very thick}}
\tikzset{thirdmodule/.style =
  {color = purple, decorate, decoration = snake, very thick}}
\tikzset{end output/.style = {color = black!20!red, decorate,
    densely dash dot, very thick}} 
\tikzset{hom output/.style = {color = black!20!blue, decorate, decoration = bumps, very thick}} 
\tikzset{parameter/.style = {color = black!20!red, thick, decorate, decoration = {coil,
      amplitude = .07cm,  segment length = .1cm}}}
\tikzset{secparameter/.style = {color = black!20!green, thick, decorate,
    decoration = {coil, dashed,
      amplitude = .07cm,  segment length = .1cm}}}
\tikzset{bmod/.style = {color = blue, dotted, line width = 2mm, very thick}}


\section{Introduction}
The bar and cobar constructions are an adjoint pair of functors
between the categories of augmented differential graded (dg) algebras and 
coaugmented dg coalgebras, both defined over a fixed commutative
ring $k$,
\begin{align*}
&\xymatrix{\augdgAlgk \ar@<-1.2ex>[rr]_{\Bar {}} &&
 \coaugdgCoalgk. \ar@<-1.2ex>[ll]_{\cobar{}}}\\
\intertext{This adjoint pair is the algebraic analogue of the
classifying space and Moore loop space adjoint pair between topological monoids
and based topological spaces.
Analogous to the situation in topology, the bar and cobar functors
                                                              are
                                                              decidedly
                                                              nontrivial:
                                                              the unit
                                                              of the
                                                              adjunction
                                                              is a
                                                              homotopy
                                                              equivalence. This
                                                              non-triviality
                                                              is
                                                              at the
                                                              root of their
                                                                                                                            usefulness
                                                                                                                           in
                                                              algebra. 
                                                              In
particular, the bar construction gives
canonical resolutions, of both modules and bimodules, and plays a
large role in infinitesimal deformation theory.
\endgraf
Stasheff, in \cite{MR0158400}, relaxed the assumption of
associativity in a topological monoid to define an \Ai-space, using a generalized notion of
classifying space. He then showed that a connected topological space has the homotopy type of a
loop space exactly when it is an \Ai-space (see also \cite[Chapter
2]{MR505692}). The algebraic analogue of a connected
\Ai-space is an augmented \Ai-algebra \cite{MR0158400part2},
                                                            generalizing
                                                            an augmented dg-algebra.
An augmented \Ai-algebra is an augmented complex $(A, m^{1})$ and a sequence
                                                            of
                                                            augmented
                                                            maps
                                                            $m^{n}:
                                                            A^{\otimes
                                                            n} \to A,$ 
where $m^{2}$ satisfies the Leibniz rule with respect to $m^{1},$ $m^{3}$ is a nullhomotopy for the associator of $m^{2}$, and
more generally, the $m^{n}$ satisfy the
                                                            quadratic equations necessary for a bar construction,
                                                            giving the following diagram:}
   &\xymatrix{\augaiAlgk \ar@<-1.2ex>[rr]_{\Bar {}} &&
 \coaugdgCoalgk. \ar@<-1.2ex>[ll]_{\cobar{}}}
\end{align*}
Every \Ai-algebra is homotopy equivalent to a dg-algebra, so nothing is gained at
the homotopy level by enlarging the category of dg-algebras, but often
there are dramatically smaller \Ai-versions 
smaller than the dg-models, e.g., the cochains of the classifying space of a finite
group; see \cite[\S 6]{MR2844537}.

The augmentation assumption plays a vital but subtle role in the
nontriviality of the above functors. Indeed, by bar construction of an augmented
dg-algebra
$\epsilon: A \to k,$ we mean the bar construction
applied to the nonunital algebra $\ker \epsilon$. The bar
construction of
a unital dg-algebra is homotopy equivalent to the trivial coalgebra, destroying the ``homotopy type'' of the unital
dg-algebra. In particular, the resolutions traditionally constructed
using the bar construction, will not necessarily be resolutions if one
doesn't kill the unit.  Augmented
algebras are exactly those we can do this to, without
losing information. All of this remains true for augmented versus strictly unital
\Ai-algebras, summarized in the following diagram:
\begin{displaymath}
  \begin{tikzpicture}
  \matrix (m) [matrix of math nodes,row sep=3em,column sep=4em,minimum width=2em]{
    \suaiAlgk & {} \\
    \augaiAlgk & \augdgCoalgk. \\
  };
 \path[-stealth, auto, scale = 2] (m-2-1) edge [subseteq]                  (m-1-1);
  \path[-stealth, auto]     (m-2-1) edge node[swap] {$\bar$} (m-2-2)
                        (m-1-1) edge node[auto = false, draw, minimum size = 20pt,
                        dashed, cross out, red, thick]{} (m-2-2);
\end{tikzpicture}
\end{displaymath}
Positselski had the insight 
that the right side of the diagram can be extended to \emph{curved
dg-coalgebras}. He showed how to construct, for a strictly unital, but not
necessarily augmented, \Ai-algebra, a curved bar construction, killing
the unit and transferring the potentially lost information
to a curvature term \cite{MR1250981,
  MR2830562}, giving the following diagram:
\begin{displaymath}
  \begin{tikzpicture}
  \matrix (m) [matrix of math nodes,row sep=3em,column sep=4em,minimum width=2em]{
    \suaiAlgk & \curvdgCoalgk \\
    \augaiAlgk & \augdgCoalgk. \\
  };
  \path[-stealth, auto, scale = 2] (m-2-1) edge [subseteq]                  (m-1-1)
  (m-2-2) edge [subseteq]                  (m-1-2);
    \path[-stealth, auto]          (m-2-1) edge node[swap] {$\bar$} (m-2-2)
                        (m-1-1) edge node {$\bar$} (m-1-2);
\end{tikzpicture}
\end{displaymath}
He proved analogous results for \Ai-morphisms and \Ai-modules, and also stated a strong converse: the curved
bar construction characterizes strictly unital \Ai-algebras
(and morphisms, and representations).

The fundamental idea that a curvature term compensates for lack of
augmentation is not particularly emphasized in the long paper
\cite{MR2830562} (a paper that contains many new and powerful ideas),
full details of
the proofs are not given, and, most importantly for us, the ground ring is assumed to be a
field. In this paper, we give careful proofs of Positselski's results,
valid for an arbitrary commutative ground ring (in fact with a few small
adjustments, noted in remarks, the results hold when
replacing modules over a commutative ring with any symmetric monoidal
category with countable coproducts, where finite coproducts are also
finite products). Positselski also showed that the bar construction is
homotopically non-trivial, and this opens the door to using it for the
construction of resolutions. We do not pursue the generalization from
a field to arbitrary base ring here, but hope to return to it in the future.

The proofs in this paper use a characterization of
strictly unital \Ai-algebra structures as Maurer-Cartan elements of a
certain dg-Lie algebra (the coassociative analogue of a construction used by Schlessinger and Stasheff for Lie coalgebras
\cite{SchlStash}). We also use this characterization to show the dg-Lie algebra of reduced Hochschild
cochains controls the
strictly unital infinitesimal deformations of the corresponding
\Ai-algebra. As motivation and context for using Maurer-Cartan
elements of a dg-Lie algebra, we include a detailed discussion of the
deformation theory of nonunital \Ai-algebras via the dg-Lie algebra
of Hochschild cochains. We also prove linear analogues of all
of the above results. In particular, we recover, and generalize to arbitrary commutative base
ring, Positselski's result that strictly unital modules correspond
functorially to cofree curved dg-comodules over the curved bar
construction.

There has been considerable further work developing Positselski's
ideas, especially for operads
\cite{MR2993002,DerKosDuality,1403.3644,1612.02254}, see also \cite{MR3597150}. There has also been much work
on homotopy, or weak, units in an \Ai-algebra; see
\cite{MR2596638,MR2769322,MR3176644} and the references contained
there. An \Ai-automorphism does not necessarily preserve a strict
unit, but it does preserve a homotopy unit (one can take for the
definition of homotopy unit that there is an automorphism that takes it to
a strict unit). Positselski's idea on curvature gives a way of
maintaining a strict unit through certain processes, e.g., transfer of
\Ai-structure, rather than working in the larger category of homotopy unital
\Ai-algebras.

Finally, let us mention one motivation for this paper. In
\cite{1508.03782} we study projective resolutions of modules over a
commutative ring $R = Q/I$ by putting $Q$-linear strictly unital \Ai-structures
on $Q$-projective resolutions of $R$ and its modules. (This example emphasizes the importance of working with
an arbitrary commutative base ring.) In particular, we show that minimality of \Ai-structures
characterizes Golod singuliarities, and the bar construction can then
be used to construct the minimal free resolution of every module over
a Golod ring. To work effectively with different classes of singularities, e.g., complete intersections, a
relative Koszul duality (relative to $Q$) is needed. We hope to develop this in future work. Throughout this paper we give a sequence of running
examples illustrating the elementary, but
interesting, example of the Koszul complex on a single element $f$ of
the ground ring $k,$ where e.g., if
$k = \bC[x_{1}, \ldots, x_n]$, then we are studying the zero set of
$f$ relative to $\bC^{n}.$

I would like to thank the referee for his or her careful reading and very
helpful comments that improved the exposition of the paper.

\section{Notation and conventions}\label{sect:notation}
\begin{enumerate}
\item Throughout, $k$ is a fixed commutative ring. By module, complex, map,
  etc.\ we mean $k$-module, complex of $k$-modules, $k$-linear map, etc.  We place no boundedness or
connectedness assumptions on complexes. For graded modules
  $M,N,$ define graded modules $\Hom {} M N$ and $M \otimes N$ by
  \begin{displaymath}
    \Hom {} M N_n = \prod_{i \in \bZ} \Hom {} {M_i} {N_{i + n}} \quad (M
    \otimes N)_n = {\bigoplus_{i \in \bZ} M_i \otimes N_{n - i}}.
  \end{displaymath}
  If $(M, \delta_{M})$ and $(N, \delta_{N})$ are complexes, then
  $\Hom {} M N$ and $M \otimes N$ are complexes with differentials
  \begin{displaymath}
    \dhom(f) = \delta_{N}f - (-1)^{|f|}f \delta_{M} \quad \delta_{\otimes} = \delta_M \otimes 1 + 1 \otimes \delta_N.
  \end{displaymath}
  A morphism of complexes is a degree 0 cycle of the complex
  $(\Hom {} M N, \dhom).$

\item All elements of graded objects are assumed to be homogeneous. We
  write $|x|$ for the degree of an element $x$. If $M$ is a graded
  module, $\Pi M$ is the graded module with $(\Pi M)_n = M_{n -1}$. Set
  $s \in \Hom {} M {\Pi M}_{1}$ to be the identity map. For $x\in M,$
  we set $[x] = s(x) \in \Pi M$ and more generally
  $[x_1|\ldots|x_n] = sx_{1}\otimes \ldots\otimes s x_n.$ If
  $(M, \delta_{M})$ is a complex, set
  $\delta_{\Pi M} = -s\delta_M s^{-1}$. Then
  $s: (M, \delta_{M}) \to (\Pi M, \delta_{\Pi M})$ is a cycle in $(\Hom {} {\Pi M} M, \dhom).$

\item We use the sign conventions that when $x,y$ are permuted, a
  factor of
  $(-1)^{|x||y|}$ is introduced, and when applying a tensor product of
  morphisms, we have
  $(f \otimes g)(x \otimes y) = (-1)^{|g||x|}f(x) \otimes g(y)$.

\item For complexes $M$ and $N$, the maps
  $(\Pi M) \otimes N \to \Pi (M \otimes N), [m] \otimes n \mapsto [m
  \otimes n]$ and
  $M \otimes (\Pi N) \to \Pi (M \otimes N),m \otimes [n] \mapsto
  (-1)^{|m|}[m \otimes n]$
  are isomorphisms of complexes, as are the maps
  $\Pi \Hom {} M N \to \Hom {} {\Pi^{-1} M} N, [f] \mapsto
  (-1)^{|f|}fs$ and
  $\Pi \Hom {} M N \to \Hom {} M {\Pi N}, [f] \mapsto sf$.

\item String diagrams are used to represent morphisms between
  tensor products of graded modules. A line represents a graded
  module, parallel lines represent a tensor product of graded modules, and a rectangle
  represents a morphism. Lines may be decorated to
  distinguish graded modules, for instance, let
  $\tikz[baseline={([yshift = -.1cm]current bounding
    box.center)}]{\draw[thick] (0,0) -- (0,.5)}$ represent a graded module $A$ and
  $\begin{tikzpicture}[baseline={([yshift = -.1cm]current bounding
      box.center)}]\draw[bmod, line width = .5mm] (0,0) --
    (0,.5);\end{tikzpicture}$ represent a graded module $B$. Then, e.g., $
  \begin{tikzpicture}
    \newBoxDiagram [arm height = .5cm, width = .25cm, decoration
    yshift label = .05cm, box height = .5cm, label = {\tiny{f}}]{f};
    \print[/bmod]{f}; \printArm{f}{0};
  \end{tikzpicture}$ represents a morphism $f: A^{\otimes 3} \to B.$
  The utility of the diagrams becomes apparent when
  composing such morphisms.
\item For unexplained conventions or definitions related to
  differential graded Lie algebras, see \cite{MR0195995}, and for
  graded coalgebras, see Chapters 1 and 5 of \cite{MR1243637}.
\end{enumerate}

\section{Nonunital \Ai-algebras}
In this section we recall the definitions of nonunital \Ai-algebra, \Ai-morphism, bar construction, and dg-Lie algebra of
Hochschild cochains. We use the approach of Proute \cite{MR2844537}
and Getzler
\cite{MR1261901}, and the string diagram notation of Hinich
\cite{MR1978336}. Fix graded modules $A, B$ throughout the section. In string diagrams $\tikz[baseline={([yshift = -.1cm]current bounding
  box.center)}]{\draw[thick] (0,0) -- (0,.5)}$ will denote $\Pi A$ and $\tikz[baseline={([yshift = -.1cm]current bounding
  box.center)}]{\draw[bmod, line width = .5mm] (0,0) -- (0,.5)}$ will denote $\Pi B.$

\begin{defn}\label{defn:bar_and_cc}
  Set
  $\hhn n A B = \Hom {} {(\Pi A)^{\otimes n}} {\Pi B},$ and
  \begin{align*}
    \hhc A B & = \prod_{n \geq 1} \Hom {} {(\Pi A)^{\otimes n}}
               {\Pi B}.
  \end{align*}
  We write $f = (f^n) \in \hhc A B$ with $f^n \in \hhn n A B$ and call
  $f^n$ the $n$th tensor homogeneous component.

Since $A$ and $B$ are graded,
so is $\hhc A B$, using \ref{sect:notation}.(1). The
$i$th homogeneous component of this grading is denoted $\hhc A B_{i}$.
\end{defn}

The module $\hhc A A$ has a very intricate algebraic structure. In
particular it is a graded Lie algebra under the commutator of the
following.
\begin{defn}
 The \emph{Gerstenhaber product} of $g = (g^n) \in \hhc A B_{i}$ and $f = (f^{n}) \in \hhc A
  A_{j}$, denoted $g \circ f \in \hhc A B_{i+j},$ has $n$th tensor homogeneous component given by the following:
  \begin{align*}\label{def:gerst_prod}
    (g \circ f)^n
                  &= \sum_{i=1}^{n} \sum_{j=0}^{i-1} \hspace{.2cm}
                    \begin{tikzpicture}[baseline={(current bounding box.center)}]
                      \def\w{1.5cm}
                      \def\hgt{.55cm}
                      \def\highPos{-.1}
                      \def\relw{.5}
                      \def\sepw{.1}
                      \def\smallsep{.08}
                      \def\modsep{.35}
                      \pgfkeys{lowBoxStyle/.style = {box height =
                          \hgt, output height = \hgt, width = \w,
                          coord =
                          {($(0,0)$)}, decoration yshift label = .2cm,
                          decoration yshift label = .05cm, arm height
                          = 3*\hgt, label = g^{i}}}
                      \newBoxDiagram [/lowBoxStyle]{lowBox};
                      {\print[/bmod]{lowBox};}
%
                      \lowBox.coord at top(highcoor, \highPos);
                      \pgfkeys{highBoxStyle/.style = {width =
                          \w*\relw, box height = \hgt, arm height =
                          \hgt, output height = \hgt, coord =
                          {(highcoor)}, output height = \hgt, label =
                          {f^{n-i+1}}}}
                      \newBoxDiagram[/highBoxStyle]{highBox};
                      \print{highBox}; \printArmSep{highBox};
%
                      \pgfmathsetmacro{\var}{\highPos - \relw - \sepw}
                      \printArm{lowBox}{\var}; \printDec[/decorate,
                      decorate left = -1, decorate right =
                      \var]{lowBox}{j};
                      \pgfmathsetmacro{\var}{-1+\smallsep}
                      \pgfmathsetmacro{\vara}{\highPos - \relw - \sepw
                        - \smallsep} \printArmSep[arm sep left = \var,
                      arm sep right = \vara]{lowBox}
                      \pgfmathsetmacro{\var}{\highPos + \relw + \sepw}
                      \printArm{lowBox}{\var} \printDec[/decorate,
                      decorate left = \var, decorate right =
                      1]{lowBox}{ i - j-1}
                      \pgfmathsetmacro{\varb}{1-\smallsep}
                      \pgfmathsetmacro{\vara}{\var+\smallsep}
                      \printArmSep[arm sep left = \vara, arm sep right
                      = \varb]{lowBox}
                                          \end{tikzpicture}
                    = \scalebox{.9}{$\displaystyle \sum_{i=1}^{n} \sum_{j=0}^{i-1} g^{i}(1^{\otimes j} \otimes f^{n - i + 1} \otimes 1^{\otimes i - j - 1}).$}
  \end{align*}
\end{defn}

This was defined in \cite{MR0161898} where it was shown to be a
pre-Lie algebra structure (Corollary to Theorem 2,
applied to Example 5.5)\footnote{A \emph{pre-Lie algebra structure}
  on a graded module $G$ is a degree zero morphism
  $\circ : G \otimes G \to G$ such that for all $f, g, h \in G$, the
  following equation holds,
  \[ (f \circ g) \circ h - f \circ (g \circ h) = (-1)^{|g||h|} \left (
      (f \circ h) \circ g - f \circ (h \circ g) \right ),\] i.e., the
  associator is symmetric in the last two values. In particular, every
  associative algebra is a pre-Lie algebra, since the
  associator is zero. See e.g., \cite[\S 1.4]{MR2954392} for more
  information.}, and by \cite[Theorem 1]{MR0161898},
the commutator of any pre-Lie algebra, defined to be
$[x,y] = x \circ y - (-1)^{|x||y|} y \circ x$, is a graded Lie algebra.

 Using string diagrams, it
  is a relatively easy exercise to show $\hhc A A$ is a pre-Lie algebra
  (see \cite[p. 20, Figure 1]{KellerDefTheory} for
  details). Performing this exercise, one will see that
  care must be taken with signs and string diagrams. The conventions
  we use
  for signs and string diagrams are formalized below
  (we encourage the reader to skip ahead and return when a sign issue
  occurs).
\begin{rem}\label{rem:signs_for_diags} 
  \mbox{}\\
  \begin{enumerate}
  \item Morphisms will always be grouped into horizontal lines,
    i.e., the projections of any two boxes onto the left side of the page are either
    disjoint or equal.
    
\item If all morphisms are on the same line,
    we visualize inputs feeding into the diagram from the
    right, along the front, and use sign convention \ref{sect:notation}.(3), e.g.,
    \begin{displaymath}
\begin{tikzpicture}[baseline={([yshift = -.3cm]current bounding
          box.center)}]
        \def\sepw{.2cm} \def\w{.5cm}
        \def\rightw{.4cm}  \def\hgt{.5cm}
        \newBoxDiagram [arm height = 2*\hgt, box
        height = \hgt, output height = \hgt, width = \w, coord =
        {($(\sepw + \w,
          0)$)}, decoration yshift label = .2cm, label = f^{n}]{f};
        \newBoxDiagram
        [box height = 0cm, output height = 0cm, arm height =4*\hgt,
        width = \rightw, coord = {( 2*\sepw+2*\w + \rightw,
          0)}, decoration yshift label = .2cm]{right};
\print{f}; \printArmSep{f};
        \print{right}; \printArmSep{right};
        \printDec[/decorate]{right}{l}; \newBoxDiagram [arm height =
        2*\hgt, box height = \hgt, output height = \hgt, width = 1.3*\w,
        coord = {($(3*\sepw + 3.3*\w + 2*\rightw,
          0)$)}, decoration yshift label = .2cm, label = g^{m}]{g};
        \print{g}; \printArmSep{g};
        \g.coord at top(coor, 1);
        \begin{scope}[thick, color = black!40!white, decoration={
    markings,
    mark=at position 0.55 with {\arrow{>}}}]
  \draw[postaction = {decorate}] ($(coor) + (\sepw, 1.8*\hgt)$) arc (10:150:2.1 and -0.2);
\end{scope}
\end{tikzpicture} \, (\bs x \otimes \bs y \otimes \bs z) = (-1)^{|\bs
  x||g| + |\bs y||g|} f^{n}(\bs x) \otimes \bs y \otimes g^{m}(\bs z),
\end{displaymath}
where $\bs x = x_{1} \otimes \ldots \otimes x_{n}, \bs y = y_{1}
\otimes \ldots \otimes y_{l}, \bs z = z_{1} \otimes \ldots \otimes z_{m}.$
    
\item If there are multiple lines of morphisms, we visualize the
  output from each line
  coming out behind the diagram,
  needing to be twisted around to the front, where the next line of
  morphisms is applied as in Step 2; e.g.,
\begin{displaymath}
\begin{tikzpicture}[baseline={([yshift = -.3cm]current bounding
          box.center)}]
        \def\sepw{.2cm} \def\w{.5cm}
        \def\rightw{.4cm}  \def\hgt{.55cm}
        \newBoxDiagram [arm height = 2*\hgt, box
        height = \hgt, output height = 4*\hgt, width = \w, coord =
        {($(\sepw + \w,
          0)$)}, decoration yshift label = .2cm, label = f^{n}]{f};
        \newBoxDiagram
        [box height = \hgt, output height = \hgt, arm height =5*\hgt,
        width = \rightw, coord = {( 2*\sepw+2*\w + \rightw,
          0)}, decoration yshift label = .2cm, label = h^{l}]{right};
\print{f}; \printArmSep{f};
        \print{right}; \printArmSep{right};
        \newBoxDiagram [arm height =
        2*\hgt, box height = \hgt, output height = 4*\hgt, width = 1.3*\w,
        coord = {($(3*\sepw + 3.3*\w + 2*\rightw,
          0)$)}, decoration yshift label = .2cm, label = g^{m}]{g};
        \print{g}; \printArmSep{g};
        \g.coord at top(coor, 1);
                \begin{scope}[thick, color = black!40!white, decoration={
    markings,
    mark=at position 0.55 with {\arrow{>}}}]
  \draw[postaction = {decorate}] ($(coor) + (\sepw, 1.8*\hgt)$) arc (10:150:2.1 and -0.2);
  \end{scope}
\begin{scope}[thick, color = black!40!white, decoration={
    markings,
    mark=at position 0.65 with {\arrow{<}}}]
    \draw [postaction = {decorate}, dashed] ($(coor) + (\sepw, -2*\hgt)$) arc (0:140:2.1 and 0.25);
  \end{scope}
        \begin{scope}[thick, color = black!40!white, decoration={
    markings,
    mark=at position 0.565 with {\arrow{>}}}]
  \draw[postaction = {decorate}] ($(coor) + (\sepw, -2*\hgt)$) arc (0:140:2.1 and -0.25);
  \end{scope}
      \end{tikzpicture}  \, (\bs x \otimes \bs y \otimes \bs z) = (-1)^{|\bs
  x||g| + |\bs y||g|+(|\bs x| + |f|)|h|} f^{n}(\bs x) \otimes h^{l}(\bs y) \otimes g^{m}(\bs z).
\end{displaymath}
Moving the line a morphism is on only changes the diagram by a
sign. Sign rules for vertical moves of a morphism are:
  \item up one line: multiply by $-1$ to the
    degree of the morphism times the sum of the degrees of the morphims
    \emph{to the left on the new line.}
  \item down one line: multiply by $-1$ to
    the degree of the morphism times the sum of the degrees of the morphisms
    \emph{to the right on the new line.}
  \end{enumerate}
  In particular, we have
  \begin{displaymath}
  \begin{tikzpicture}
      \def\sep{.2cm} \newBoxDiagram [output height = 1.5cm] {f};
      \newBoxDiagram [coord = {(1cm + \sep, 0)}, arm height = 1.75cm]
      {g}; \print{f}; \printArmSep{f}; \print{g};
      \printArmSep{g};
    \end{tikzpicture} = (-1)^{|f||g|} \hspace{.1cm}
    \begin{tikzpicture}
      \def\sep{.2cm} \newBoxDiagram {f}; \newBoxDiagram [coord = {(1cm
        + \sep, 0)}] {g}; \print{f}; \printArmSep{f}; \print{g};
      \printArmSep{g};
    \end{tikzpicture}
    = (-1)^{|f||g|} \hspace{.1cm}
    \begin{tikzpicture}
      \def\sep{.2cm} \newBoxDiagram [arm height = 1.75cm] {f};
      \newBoxDiagram [coord = {(1cm + \sep, 0)}, output height = 1.5cm]
      {g}; \print{f}; \printArmSep{f}; \print{g};
      \printArmSep{g};
      \printPeriod{g};
    \end{tikzpicture}
  \end{displaymath}
\end{rem}

\begin{defn}
 The \emph{nonunital tensor coalgebra
    on a graded module $V$} is $\Tco V = \bigoplus_{n \geq 1} V^{\otimes
    n}$ with
  comultiplication the linear extension of
  \[ \Delta( v_1 \otimes \ldots \otimes v_n ) = \sum_{i = 1}^{n-1} (v_1
    \otimes \ldots \otimes v_i) \otimes (v_{i+1} \otimes \ldots
    \otimes v_n).\] Note that
  $\hhc A B = \Hom {} {\Tco {\Pi A}} {\Pi B}.$

A \emph{graded coderivation} of a graded coalgebra $C$ with comultiplication
  $\Delta$ is a homogeneous
  endomorphism $d$ of $C$ such that
  $(d \otimes 1 + 1 \otimes d) \Delta = \Delta d$. We write
  $\Coder(C,C)$ for the set of coderivations. This is a graded Lie
  subalgebra of the commutator bracket on $\Hom {} C C$.
\end{defn}

\begin{lem}\label{lem:isom_hhc_coder}
The tensor coalgebra satisfies the following universal properties.
  \begin{enumerate}
  \item The canonical projection $\pi_1: \Tco {\Pi A} \to \Pi A$
    induces an isomorphism,
    \[ \Phi = (\pi_1)_*: {\Coder}( \Tco {\Pi A},
      \Tco {\Pi A}) \xra{\cong} \Hom {} {\Tco {\Pi A}} {\Pi A} = \hhc
      A A.\] This is an isomorphism of graded Lie algebras, where the
    bracket on the source is the commutator, and the bracket on the
    target is the Gerstenhaber bracket. The inverse applied to
    $f =(f^n) \in \hhc A A$ is given by
    \begin{align*}
      \pi_{n-i+1} \Phi^{-1}(f)|_{(\Pi A)^{\otimes n}}
                                           &\begin{tikzpicture}[inner sep =0mm]
                                                      \def\sepw{.2cm}
                                                      \def\leftw{.4cm}
                                                      \def\w{.5cm}
                                                      \def\rightw{.6cm}
                                                      \node at
                                                      (-1.3,.8)
                                                      {$=
                                                       \displaystyle
                                                        \sum_{j
                                                          =0}^{n-i}$};
                                                      \newBoxDiagram
                                                      [box height =
                                                      0cm, output
                                                      height = 0cm,
                                                      arm height =
                                                      1.5cm, width =
                                                      \leftw,
                                                      decoration
                                                      yshift label =
                                                      .2cm]{left};
                                                      \newBoxDiagram
                                                      [arm height =
                                                      .5cm, width =
                                                      \w, coord =
                                                      {($(\sepw +
                                                        \leftw + \w,
                                                        0)$)},
                                                      decoration
                                                      yshift label =
                                                      .2cm, label =
                                                      f^i]{f};
                                                      \newBoxDiagram
                                                      [box height =
                                                      0cm, output
                                                      height = 0cm,
                                                      arm height =
                                                      1.5cm, width =
                                                      \rightw, coord =
                                                      {(\leftw +
                                                        2*\sepw+2*\w +
                                                        \rightw, 0)},
                                                      decoration
                                                      yshift label =
                                                      .2cm]{right};
                                                      \print{left};
                                                      \printArmSep{left};
                                                      \printDec[/decorate]{left}{j};
                                                      \print{f};
                                                      \printArmSep{f};
                                                      \print{right};
                                                      \printArmSep{right};
                                                \printDec[/decorate]{right}{n-i-j};
                                                      \printPeriod{right};
                                                    \end{tikzpicture}
    \end{align*}

  \item The canonical projection $\pi_1: \Tco {\Pi B} \to \Pi B$
    induces an isomorphism,
    \[ \Psi = (\pi_1)_*: \Hom {\Coalgk} {\Tco {\Pi A}} {\Tco {\Pi B}}
      \xra{\cong}\hhc A B_{0}.\]
    The inverse applied to $g =(g^n) \in \hhc A B_{0}$ is
    given by:
    \[ \pi_k \Psi^{-1}(g)|_{(\Pi A)^{\otimes n}} = \sum_{i_1 + \ldots +
        i_k = n} \hspace{.1cm}
      \begin{tikzpicture}
        \def\wida{.4cm} \def\widb{.35cm} \def\widk{.45cm}
        \def\dotsw{.45cm} \def\sepw{.1cm} \def\hgt{.43cm}
        \def\boxht{.54cm} \newBoxDiagram[width = \wida, box height =
        \boxht, arm height = \hgt, output height = \hgt, label =
        g^{i_1}]{g1} \print[/bmod]{g1} \printArmSep{g1}
        \newBoxDiagram[width = \widb, box height = \boxht, arm height
        = \hgt, output height = \hgt, coord =
        {($(\wida+\sepw+\widb,
          0)$)}, label = g^{i_2}]{g2} \print[/bmod]{g2} \printArmSep{g2}
        \node at
        ($(\wida + \sepw + 2*\widb + \dotsw, \hgt +
        \boxht/2)$) {\ldots}; \newBoxDiagram[width = \widk, box height
        = \boxht, arm height = \hgt, output height = \hgt, coord =
        {($(\wida+\sepw+2*\widb+2*\dotsw+\widk,
          0)$)}, label = g^{i_k}]{gk};\print[/bmod]{gk}; \printArmSep{gk};
        \printPeriod{gk};
      \end{tikzpicture}\]
  \end{enumerate}
\end{lem}

For  $g \in \hhc A B$ and $f \in \hhc A A$, it follows from \ref{lem:isom_hhc_coder}.(1) that $g \circ f = g \Phi^{-1}(f).$ We define an analogous product using $\Psi^{-1}.$
\begin{defn}\label{rem:defining-start-and-gerst-produc}
  For $g \in \hhc A B_{0}$ and $h\in \hhc B B_{i}$, set
  \[h * g = h \Psi^{-1}(g) \in \hhc A B_{i}.\]
\end{defn}

\begin{rem}\label{rem:summation-convention}
  If there are no superscripts on the morphisms of a string diagram,
  the diagram represents an element $\xi$ of $\hhc A B$ with $\xi^{n}$
  given by the summing over all diagrams of the given shape that have $n$
  inputs. For example, if $h \in \hhc B B$ and $g \in \hhc A B_{0}$,
we write
{$h * g = \begin{tikzpicture}[scale=0.75, every node/.style={transform shape}] \def\wida{.33cm} 
    \def\widb{.30cm}
 \def\widk{.33cm}
 \def\dotsw{.32cm}
 \def\sepw{.1cm}
\def\hgt{.5cm}
 \def\wdth{.55cm}
        \newBoxDiagram[width = \wida, box height = \hgt, arm height
        = \hgt, coord = {(0,0)}, output height = \hgt, label =
        {g}]{g1};
        \print[/bmod]{g1};
 \printArmSep{g1};
        \printDec[color = white]{g1}{\phantom{test}}

        \node at
        ($(\wida + \sepw + \dotsw, \hgt +
        \hgt/2)$) {\ldots};
 \newBoxDiagram[width = \widk, box height
        = \hgt, arm height = \hgt, output height = \hgt, coord =
        {($(\wida+\sepw+2*\dotsw+\widk,
          0)$)}, label = {g}]{gk};
 \print[/bmod]{gk}; 
\printArmSep{gk};
\newlength{\var}
\newlength{\vara}
        \setlength{\var}{\wida/2+\sepw/2+\dotsw+\widk/2}
        \setlength{\vara}{-\hgt - \hgt}
 \newBoxDiagram[width = \var,
        box height = \hgt, arm height = \hgt, output height = \hgt,
        coord = {(\var, \vara)}, label = {h}]{mut}
        \printNoArms[/bmod]{mut}
      \end{tikzpicture}$}
to mean
\begin{displaymath} 
\def\wida{.4cm} 
\def\widb{.35cm}
 \def\widk{.45cm}
 \def\dotsw{.45cm}
 \def\sepw{.1cm}
\def\hgt{.4cm}
\def\boxht{.55cm}
(h*g)^{n} = \sum_{i_1 + \ldots + i_k = n} \hspace{.1cm}
\begin{tikzpicture}[baseline={(current bounding
        box.center)}]
        \newBoxDiagram[width = \wida, box height = \boxht, arm height
        = \hgt, output height = \hgt, label = g^{i_1}]{g1};
        \print[/bmod]{g1};
 \printArmSep{g1};

 \newBoxDiagram[width = \widb,
        box height = \boxht, arm height = \hgt, output height = \hgt,
        coord = {($(\wida+\sepw+\widb,
          0)$)}, label = g^{i_2}]{g2} \print[/bmod]{g2} \printArmSep{g2};

        \node at
        ($(\wida + \sepw + 2*\widb + \dotsw, \hgt +
        \boxht/2)$) {\ldots};
 \newBoxDiagram[width = \widk, box height
        = \boxht, arm height = \hgt, output height = \hgt, coord =
        {($(\wida+\sepw+2*\widb+2*\dotsw+\widk,
          0)$)}, label = g^{i_k}]{gk};
 \print[/bmod]{gk}; 
\printArmSep{gk};

        \setlength{\var}{\wida/2+\sepw/2+\widb+\dotsw+\widk/2}
        \setlength{\vara}{-\boxht - \hgt} \newBoxDiagram[width = \var,
        box height = \boxht, arm height = \hgt, output height = \hgt,
        coord = {(\var, \vara)}, label = {h^{k}}]{mut}
        \printNoArms[/bmod]{mut}
      \end{tikzpicture}
    \end{displaymath}
    and analogously,
    $\displaystyle g \circ f =
    \sum_{j}\hspace{.2cm}\begin{tikzpicture}[baseline={(current
        bounding box.center)}, scale=0.75, every node/.style={transform shape}]
      \def\w{1cm} \def\gstart{-.1} \def\relw{.4} \def\sepw{.1}
      \newBoxDiagram[width/.expand once = \w, arm height = 1.5cm, box
      height = .5cm,
      label = g]{f}; \print[/bmod]{f}; \f.coord at top(gcoor, \gstart);
      \newBoxDiagram[width = \w*\relw, arm height = .5cm, coord =
      {(gcoor)}, output height = .5cm, label = f]{g};
      \print{g}; \printArmSep{g};
      \def\smallsep{.08}
      \pgfmathsetmacro{\var}{\gstart - \relw - \sepw}
      \printArm{f}{\var}; \printDec[/decorate, decorate left = -1,
      decorate right = \var]{f}{j};
      \pgfmathsetmacro{\var}{-1+\smallsep}
      \pgfmathsetmacro{\vara}{\gstart - \relw - \sepw - \smallsep}
      \printArmSep[arm sep left = \var, arm sep right = \vara]{f}
      \pgfmathsetmacro{\var}{\gstart + \relw + \sepw}
      \printArm{f}{\var} \pgfmathsetmacro{\var}{1-\smallsep}
      \pgfmathsetmacro{\vara}{\gstart + \relw + \sepw + \smallsep}
      \printArmSep[arm sep left = \vara, arm sep right = \var]{f}
      \printPeriod{f}.
    \end{tikzpicture}$ (We also extend this notation to tensor products of
    elements of $\hhc A A$ in the proof below.)
\end{rem}

\begin{proof}[Proof of Lemma \ref{lem:isom_hhc_coder}]
For proofs that $\Phi$ and $\Psi$ are isomorphisms of modules see
e.g., \cite[2.16, 2.19]{MR2844537}. We
will show that $\Phi^{-1}$ is a morphism of graded Lie algebras (this
is also presumably well known, but string diagrams give an easy proof). Let $f, g \in \hhc A A$, and set
  $d = \Phi^{-1}(f), e = \Phi^{-1}(g)$. We then have
  \[ d e = \Phi^{-1}(f) \Phi^{-1}(g) = \hspace{.1cm} \def\sepw{.2cm}
    \def\leftw{.2cm} \def\w{.25cm} \def\rightw{.3cm} \left
      ( \begin{tikzpicture}[baseline={([yshift = -.3cm]current
          bounding box.center)}] \node at (-.75,.75)
        {$\displaystyle
          \sum_{j}$}; \newBoxDiagram [box height = 0cm, output height
        = 0cm, arm height = 1.5cm, width = \leftw, decoration yshift
        label = .2cm]{left};
        \newBoxDiagram [arm height = .5cm, width = \w, coord =
        {($(\sepw + \leftw + \w,
          0)$)}, decoration yshift label = .2cm]{f};
        \newBoxDiagram [box height = 0cm, output height = 0cm, arm
        height = 1.5cm, width = \rightw, coord = {(\leftw +
          2*\sepw+2*\w + \rightw, 0)}, decoration yshift label =
        .2cm]{right};
        \print{left}; \printArmSep{left};
        \printDec[/decorate]{left}{j}; \print{f}; \printArmSep{f};
        \print{right}; \printArmSep{right};
      \end{tikzpicture} \right ) \def\sepw{.2cm} \def\leftw{.3cm}
    \def\w{.25cm} \def\rightw{.2cm} \left
      ( \begin{tikzpicture}[baseline={([yshift = -.3cm]current
          bounding box.center)}] \node at (-.75,.75)
        {$\displaystyle
          \sum_{k}$}; \newBoxDiagram [box height = 0cm, output height
        = 0cm, arm height = 1.5cm, width = \leftw, decoration yshift
        label = .2cm]{left};
        \newBoxDiagram [arm height = .5cm, width = \w, coord =
        {($(\sepw + \leftw + \w,
          0)$)}, decoration yshift label = .2cm, label = g]{f};
        \newBoxDiagram [box height = 0cm, output height = 0cm, arm
        height = 1.5cm, width = \rightw, coord = {(\leftw +
          2*\sepw+2*\w + \rightw, 0)}, decoration yshift label =
        .2cm]{right};
        \print{left}; \printArmSep{left};
        \printDec[/decorate]{left}{k}; \print{f}; \printArmSep{f};
        \print{right}; \printArmSep{right};
      \end{tikzpicture} \right ).
  \]
  When composing terms, $g$ is inserted to the left of, into, or to
  the right of, $f$, so
  \[ d e = \sum_{j,k} \left (
      \begin{tikzpicture}[baseline={([yshift = -.3cm]current bounding
          box.center)}]
        \def\sepw{.2cm} \def\leftw{.25cm} \def\w{.25cm}
        \def\rightw{.2cm} \def\rightrightw{.25cm} \def\hgt{.5cm}
        \hspace{.1cm} \newBoxDiagram [box height = 0cm, output height
        = 0cm, arm height = 5*\hgt, width = \leftw, decoration yshift
        label = .2cm]{left}; \newBoxDiagram [arm height = \hgt, box
        height = \hgt, output height = 3*\hgt, width = \w, coord =
        {($(\sepw + \leftw + \w,
          0)$)}, decoration yshift label = .2cm]{g}; \newBoxDiagram
        [box height = 0cm, output height = 0cm, arm height =5*\hgt,
        width = \rightw, coord = {(\leftw + 2*\sepw+2*\w + \rightw,
          0)}, decoration yshift label = .2cm]{right};
        \print{left}; \printArmSep{left};
        \printDec[/decorate]{left}{k}; \print{g}; \printArmSep{g};
        \print{right}; \printArmSep{right};
        \printDec[/decorate]{right}{j}; \newBoxDiagram [arm height =
        3*\hgt, box height = \hgt, output height = \hgt, width = \w,
        coord = {($(3*\sepw + \leftw + 3*\w + 2*\rightw,
          0)$)}, decoration yshift label = .2cm]{f}; \newBoxDiagram
        [box height = 0cm, output height = 0cm, arm height =5*\hgt,
        width = \rightrightw, coord = {(\leftw + 4*\sepw + 4*\w +
          2*\rightw+\rightrightw, 0)}, decoration yshift label=
        .2cm]{rightright}; \print{f}; \printArmSep{f};
        \print{rightright}; \printArmSep{rightright};
      \end{tikzpicture}
      \quad + \quad
      \begin{tikzpicture}[baseline={([yshift = -.3cm]current bounding
          box.center)}]
        \def\w{.75cm} \def\gstart{.1} \def\relw{.4} \def\sepw{.1}
        \def\sepwidth{.1cm} \def\leftw{.3cm} \def\rightw{.3cm}
        \def\hgt{.5cm} \newBoxDiagram [box height = 0cm, output height
        = 0cm, arm height = 2*\hgt, width = \leftw, decoration yshift
        label = .05cm, arm height = 5*\hgt]{left};
        \newBoxDiagram [box height = \hgt, output height = \hgt, width
        = \w, coord = {($(\sepwidth + \leftw + \w,
          0)$)}, decoration yshift label = .2cm, decoration yshift
        label = .05cm, arm height = 3*\hgt]{f};
        \newBoxDiagram [box height = 0cm, output height = 0cm, arm
        height = 3*\hgt, width = \rightw, coord = {(\leftw +
          2*\sepwidth+2*\w + \rightw, 0)}, decoration yshift label =
        .2cm, arm height = 5*\hgt]{right}; \print{f}; \f.coord at
        top(gcoor, \gstart); \newBoxDiagram[width = \w*\relw, box
        height = \hgt, arm height = \hgt, output height = \hgt, coord
        = {(gcoor)}, output height = \hgt]{g}; \print{g};
        \printArmSep{g};

        \def\smallsep{.08}
        \pgfmathsetmacro{\var}{\gstart - \relw - \sepw}
        \printArm{f}{\var}; \printDec[/decorate, decorate left = -1,
        decorate right = \var]{f}{j};
        \pgfmathsetmacro{\var}{-1+\smallsep}
        \pgfmathsetmacro{\vara}{\gstart - \relw - \sepw - \smallsep}
        \printArmSep[arm sep left = \var, arm sep right = \vara]{f}
        \pgfmathsetmacro{\var}{\gstart + \relw + \sepw}
        \printArm{f}{\var}
        \pgfmathsetmacro{\var}{1-\smallsep}
        \pgfmathsetmacro{\vara}{\gstart + \relw + \sepw + \smallsep}
        \printArmSep[arm sep left = \vara, arm sep right = \var]{f}
        \print{left}; \printArmSep{left};
        \printDec[/decorate]{left}{k}; \print{right};
        \printArmSep{right};
      \end{tikzpicture}
      \quad+\quad
      \begin{tikzpicture}[baseline={([yshift = -.3cm]current bounding
          box.center)}]
        \def\sepw{.2cm} \def\leftw{.2cm} \def\w{.25cm}
        \def\rightw{.3cm} \def\rightrightw{.25cm} \def\hgt{.5cm}
        \newBoxDiagram [box height = 0cm, output height = 0cm, arm
        height = 5*\hgt, width = \leftw, decoration yshift label =
        .2cm]{left}; \newBoxDiagram [arm height = 3*\hgt, box height =
        \hgt, output height = \hgt, width = \w, coord =
        {($(\sepw + \leftw + \w,
          0)$)}, decoration yshift label = .2cm, label = f]{g};
        \newBoxDiagram [box height = 0cm, output height = 0cm, arm
        height =5*\hgt, width = \rightw, coord = {(\leftw +
          2*\sepw+2*\w + \rightw, 0)}, decoration yshift label =
        .2cm]{right};
        \print{left}; \printArmSep{left};
        \printDec[/decorate]{left}{j}; \print{g}; \printArmSep{g};
        \print{right}; \printArmSep{right};
        \printDec[/decorate]{right}{k}; \newBoxDiagram [arm height =
        \hgt, box height = \hgt, output height = 3*\hgt, width = \w,
        coord = {($(3*\sepw + \leftw + 3*\w + 2*\rightw,
          0)$)}, decoration yshift label = .2cm, label = g]{f};
        \newBoxDiagram [box height = 0cm, output height = 0cm, arm
        height =5*\hgt, width = \rightrightw, coord = {(\leftw +
          4*\sepw + 4*\w + 2*\rightw+\rightrightw, 0)}, decoration
        yshift label = .2cm]{rightright};
 
        \print{f}; \printArmSep{f}; \print{rightright};
        \printArmSep{rightright}
      \end{tikzpicture}
      \hspace{.2cm}\right ) \] We then have, see
  \ref{rem:signs_for_diags} for signs,
  \[ [d,e] = de - (-1)^{|d||e|}ed = \sum_{j,k} \left (
      \begin{tikzpicture}[baseline={([yshift = -.3cm]current bounding
          box.center)}]
        \def\w{.75cm} \def\gstart{.1} \def\relw{.4} \def\sepw{.1}
        \def\sepwidth{.1cm} \def\leftw{.3cm} \def\rightw{.3cm}
        \def\hgt{.5cm} \newBoxDiagram [box height = 0cm, output height
        = 0cm, arm height = 2*\hgt, width = \leftw, decoration yshift
        label = .05cm, arm height = 5*\hgt]{left};
        \newBoxDiagram [box height = \hgt, output height = \hgt, width
        = \w, coord = {($(\sepwidth + \leftw + \w,
          0)$)}, decoration yshift label = .2cm, decoration yshift
        label = .05cm, arm height = 3*\hgt]{f};
        \newBoxDiagram [box height = 0cm, output height = 0cm, arm
        height = 3*\hgt, width = \rightw, coord = {(\leftw +
          2*\sepwidth+2*\w + \rightw, 0)}, decoration yshift label =
        .2cm, arm height = 5*\hgt]{right}; \print{f}; \f.coord at
        top(gcoor, \gstart); \newBoxDiagram[width = \w*\relw, box
        height = \hgt, arm height = \hgt, output height = \hgt, coord
        = {(gcoor)}, output height = \hgt]{g}; \print{g};
        \printArmSep{g};

        \def\smallsep{.08}
        \pgfmathsetmacro{\var}{\gstart - \relw - \sepw}
        \printArm{f}{\var}; \printDec[/decorate, decorate left = -1,
        decorate right = \var]{f}{j};
        \pgfmathsetmacro{\var}{-1+\smallsep}
        \pgfmathsetmacro{\vara}{\gstart - \relw - \sepw - \smallsep}
        \printArmSep[arm sep left = \var, arm sep right = \vara]{f}
        \pgfmathsetmacro{\var}{\gstart + \relw + \sepw}
        \printArm{f}{\var}
        \pgfmathsetmacro{\var}{1-\smallsep}
        \pgfmathsetmacro{\vara}{\gstart + \relw + \sepw + \smallsep}
        \printArmSep[arm sep left = \vara, arm sep right = \var]{f}
        \print{left}; \printArmSep{left};
        \printDec[/decorate]{left}{k}; \print{right};
        \printArmSep{right};
      \end{tikzpicture}
      -(-1)^{|d||e|} \begin{tikzpicture}[baseline={([yshift =
          -.3cm]current bounding box.center)}] \def\w{.75cm}
        \def\gstart{-.1} \def\relw{.4} \def\sepw{.1}
        \def\sepwidth{.1cm} \def\leftw{.35cm} \def\rightw{.25cm}
        \def\hgt{.5cm} \newBoxDiagram [box height = 0cm, output height
        = 0cm, arm height = 2*\hgt, width = \leftw, decoration yshift
        label = .05cm, arm height = 5*\hgt]{left};
        \newBoxDiagram [box height = \hgt, output height = \hgt, width
        = \w, coord = {($(\sepwidth + \leftw + \w,
          0)$)}, decoration yshift label = .2cm, decoration yshift
        label = .05cm, arm height = 3*\hgt, label = g]{f};
        \newBoxDiagram [box height = 0cm, output height = 0cm, arm
        height = 3*\hgt, width = \rightw, coord = {(\leftw +
          2*\sepwidth+2*\w + \rightw, 0)}, decoration yshift label =
        .2cm, arm height = 5*\hgt]{right}; \print{f}; \f.coord at
        top(gcoor, \gstart); \newBoxDiagram[width = \w*\relw, box
        height = \hgt, arm height = \hgt, output height = \hgt, coord
        = {(gcoor)}, output height = \hgt, label = f]{g};
        \print{g}; \printArmSep{g};

        \def\smallsep{.08}
        \pgfmathsetmacro{\var}{\gstart - \relw - \sepw}
        \printArm{f}{\var}; \printDec[/decorate, decorate left = -1,
        decorate right = \var]{f}{k};
        \pgfmathsetmacro{\var}{-1+\smallsep}
        \pgfmathsetmacro{\vara}{\gstart - \relw - \sepw - \smallsep}
        \printArmSep[arm sep left = \var, arm sep right = \vara]{f}
        \pgfmathsetmacro{\var}{\gstart + \relw + \sepw}
        \printArm{f}{\var}
        \pgfmathsetmacro{\var}{1-\smallsep}
        \pgfmathsetmacro{\vara}{\gstart + \relw + \sepw + \smallsep}
        \printArmSep[arm sep left = \vara, arm sep right = \var]{f}
        \print{left}; \printArmSep{left};
        \printDec[/decorate]{left}{j};

        \print{right}; \printArmSep{right};
      \end{tikzpicture} \right )
  \]
  \[ = \Phi^{-1}(f \circ g - (-1)^{|f||g|} f \circ g) =
    \Phi^{-1}[f,g].\qedhere\]
\end{proof}

We will need the following in a later section. It follows from the
explicit formulas for $\Phi^{-1}$ and $\Psi^{-1}$ given in Lemma \ref{lem:isom_hhc_coder}.
\begin{cor}\label{cor:morphisms-commuting-with-coderivations}
  A graded coalgebra morphism $\gamma: \Tco {\Pi A} \to \Tco {\Pi B}$
  commutes with coderivations $d_{A}$ and $d_{B}$, of $\Tco {\Pi A}$ and
  $\Tco {\Pi B}$, respectively, if and only if
  $\pi_{1} \gamma d_{A} = \pi_{1} d_B \gamma.$
\end{cor}

\begin{defn}\label{defn-A-infinity-algebra-morphism} Let $A, B$ be
  graded modules.
  \begin{enumerate}
  \item A \emph{nonunital \Ai-algebra structure} on $A$ is an element
    $\nu \in \hhc A A_{-1}$ such that $\nu \circ \nu = 0.$ For
    $\nu = (\nu^n)$, this is equivalent to
    \[
      \sum_{\substack{1 \leq i \leq n\\0 \leq j \leq i - 1}} \hspace{.1cm} \begin{tikzpicture}[baseline={(current
          bounding box.center)}] \def\w{1.5cm} \def\gstart{-.1}
        \def\relw{.5} \def\sepw{.1} \newBoxDiagram[width/.expand once
        = \w, arm height = 1.5cm, box height = .55cm, label =
        \nu^i]{f}; \print{f}; \f.coord at top(gcoor, \gstart);
        \newBoxDiagram[width = \w*\relw, arm height = .5cm, coord =
        {(gcoor)}, output height = .5cm, label = \nu^{n-i+1}]{g};
        \print{g}; \printArmSep{g};

        \def\smallsep{.08}
        \pgfmathsetmacro{\var}{\gstart - \relw - \sepw}
        \printArm{f}{\var}; \printDec[/decorate, decorate left = -1,
        decorate right = \var]{f}{j};
        \pgfmathsetmacro{\var}{-1+\smallsep}
        \pgfmathsetmacro{\vara}{\gstart - \relw - \sepw - \smallsep}
        \printArmSep[arm sep left = \var, arm sep right = \vara]{f}

        \pgfmathsetmacro{\var}{\gstart + \relw + \sepw}
        \printArm{f}{\var} \printDec[/decorate, decorate left = \var,
        decorate right = 1]{f}{ i - j-1}
        \pgfmathsetmacro{\var}{1-\smallsep}
        \pgfmathsetmacro{\vara}{\gstart + \relw + \sepw + \smallsep}
        \printArmSep[arm sep left = \vara, arm sep right = \var]{f}
      \end{tikzpicture} = 0 \text{ \quad for all } n \geq 1.
    \]
  \item An \emph{\Ai-morphism} $(A, \nu_A) \to (B, \nu_B)$ between
    nonunital \Ai-algebras is an element $g \in \hhc A B_{0}$ such
    that $ \nu_B * g = g \circ \nu_A,$ where $*$ is defined in
    \ref{rem:defining-start-and-gerst-produc}. In diagrams this means,
    \[
      \def\wida{.4cm} \def\widb{.35cm} \def\widk{.45cm}
      \def\dotsw{.45cm} \def\sepw{.1cm} \def\hgt{.55cm}
      \def\boxht{.55cm} \sum_{i_1 + \ldots + i_k = n} \hspace{.1cm}
      \begin{tikzpicture}
        \newBoxDiagram[width = \wida, box height = \boxht, arm height
        = \hgt, output height = \hgt, label = g^{i_1}]{g1}
        \print[/bmod]{g1} \printArmSep{g1} 
        \printDec[color = white]{g1}{test}

\newBoxDiagram[width = \widb,
        box height = \boxht, arm height = \hgt, output height = \hgt,
        coord = {($(\wida+\sepw+\widb,
          0)$)}, label = g^{i_2}]{g2} \print[/bmod]{g2} \printArmSep{g2}
        \node at
        ($(\wida + \sepw + 2*\widb + \dotsw, \hgt +
        \boxht/2)$) {\ldots}; \newBoxDiagram[width = \widk, box height
        = \boxht, arm height = \hgt, output height = \hgt, coord =
        {($(\wida+\sepw+2*\widb+2*\dotsw+\widk,
          0)$)}, label = g^{i_k}]{gk} \print[/bmod]{gk} \printArmSep{gk}
        \setlength{\var}{\wida/2+\sepw/2+\widb+\dotsw+\widk/2}
        \setlength{\vara}{-\boxht - \hgt} \newBoxDiagram[width = \var,
        box height = \boxht, arm height = \hgt, output height = \hgt,
        coord = {(\var, \vara)}, label = {\nu^{k}_{B}}]{mut}
        \printNoArms[/bmod]{mut}
      \end{tikzpicture}
      =\sum_{\substack{1 \leq i \leq n\\0 \leq j \leq i - 1}} \hspace{.1cm}
      \begin{tikzpicture}
        \def\w{1.5cm} \def\gstart{-.1} \def\relw{.5} \def\sepw{.1}
        \newBoxDiagram[width/.expand once = \w, arm height = 2*\hgt +
        \boxht, output height = \hgt, box height = \boxht, label =
        g^i]{f}; \print[/bmod]{f}; \f.coord at top(gcoor, \gstart);
        \newBoxDiagram[box height = \boxht, width = \w*\relw, arm
        height = \hgt, coord = {(gcoor)}, output height = \hgt, label
        = \nu^{n-i+1}_A]{g}; \print{g}; \printArmSep{g};
        \def\smallsep{.08}
        \pgfmathsetmacro{\var}{\gstart - \relw - \sepw}
        \printArm{f}{\var}; \printDec[/decorate, decorate left = -1,
        decorate right = \var]{f}{j};
        \pgfmathsetmacro{\var}{-1+\smallsep}
        \pgfmathsetmacro{\vara}{\gstart - \relw - \sepw - \smallsep}
        \printArmSep[arm sep left = \var, arm sep right = \vara]{f}

        \pgfmathsetmacro{\var}{\gstart + \relw + \sepw}
        \printArm{f}{\var}
        \pgfmathsetmacro{\var}{1-\smallsep}
        \pgfmathsetmacro{\vara}{\gstart + \relw + \sepw + \smallsep}
        \printArmSep[arm sep left = \vara, arm sep right = \var]{f}
      \end{tikzpicture} \text{ for all } n \geq 1.
    \]
  \item The \emph{bar construction} of a nonunital \Ai-algebra
    $(A, \nu)$ is the dg-coalgebra $\Bar A = (\Tco {\Pi A}, \Phi^{-1}(\nu)).$ Since
    $[\nu, \nu] = 0$, and $\Phi^{-1}$ is a morphism of Lie algebras, it follows that $[\Phi^{-1}(\nu),\Phi^{-1}(\nu)] = 0$ and so
    $\Phi^{-1}(\nu)^2 = 0$ (assuming $1/2 \in k$; or one can modify the proof of
    \ref{lem:isom_hhc_coder}). This is functorial with respect to
    \Ai-morphisms, using \ref{lem:isom_hhc_coder}.(2).
    
  \item The \emph{Hochschild cochains} of a nonunital \Ai-algebra
    $(A, \nu)$ is the dg-Lie algebra
    $(\hhc A A, [\nu, -])$.\footnote{This is a slightly non-standard
      version of the Hochschild
      cochains; the standard definition is $\Sigma
      \hhc A A \oplus A.$ To see the Lie algebra structure and
      differential agree in the classical case when $A$ is a
      $k$-algebra, see equation 23 on
  page 280 of \cite{MR0161898} and \cite{MR1239562}.}
  \end{enumerate}
\end{defn}

\begin{rem}\label{ex:dg-alg-is-ainf-alg}
It is often convenient to pass from a family of degree -1 maps
$\nu^{n}: (\Pi A)^{\otimes n} \to \Pi A$ to a family of degree $n -2$
maps $m^{n}: A^{\otimes n} \to A,$ and vice versa. We use the convention that
\begin{align*}
  &m^{n} = s^{-1} \nu^{n}s^{\otimes n},\\
  \intertext{and since $(s^{\otimes n})^{-1} = (-1)^{\frac{n(n-1)}{2}}(s^{-1})^{\otimes
    n}$, it follows that}
  &\nu^{n} = (-1)^{\frac{n(n-1)}{2}} s m^{n} (s^{-1})^{\otimes n}.
\end{align*}
(Proute \cite{MR2844537} uses the convention that $\nu^{n} =
-sm^{n}(s^{-1})^{\otimes n}.$) If $(A, (\nu^{n}))$ is an \Ai-algebra, then,
in low tensor degrees, the corresponding maps $m^{n}$ satisfy:
\begin{equation*}
  \label{eq:1}
  \begin{array}{lccr}
    n = 1 && m^{1} m^{1} = 0\\
    n = 2 && m^{1} m^{2} = m^{2}(m^{1} \otimes 1 + 1 \otimes m^{1})\\
n = 3 && m^{2}(1 \otimes m^{2} - m^{2} \otimes 1) = m^{1}m^{3} + m^{3}
         \circ m^{1}.
  \end{array}
\end{equation*}
Thus, $(A, m^{1})$ is a complex, $m^{2}$ satisfies the Leibniz rule
with respect to $m^{1},$ and the associator of $m^{2}$ is a boundary
in the Hom-complex $(\Hom {} {A^{\otimes 3}} A, \dhom)$ between the
complexes $(A^{\otimes 3}, \delta_{\otimes})$ and
$(A, m^{1}).$

It follows easily from the above that a dg-algebra, i.e., a complex $(A, m^{1})$ with a
compatible associative multiplication $m^{2},$ uniquely determines an
\Ai-algebra $(A, (\nu^{n}))$ with $\nu^{n} = 0$ for all $n \geq 3.$
Conversely, such an \Ai-algebra uniquely determines a dg-algebra.
\end{rem}

\begin{rems}
In this section we could replace the category of graded $k$-modules
with the category of graded objects in an arbitrary symmetric monoidal
category with coproducts, such that a finite coproduct is also a
product, and such that the coproduct behaves as expected with respect
to the tensor product. Indeed, given an object $V$ in such a category, set $\Tco V =
\bigoplus_{n \geq 1} V^{\otimes n}$, and define a comultiplication $\Tco
V \to \Tco V \otimes \Tco V$ on the component $V^{\otimes n}$ to be
the map $V^{\otimes n} \to \bigoplus_{i = 1}^{n-1}(V^{\otimes i} \otimes
V^{\otimes n-i}) \to \Tco V \otimes \Tco V,$ which has components $V^{\otimes n}
\xra{\cong} V^{\otimes i} \otimes V^{\otimes n - i}.$ Then $\Tco V$ is
a coalgebra object in the category, and satisfies the formal
properties of \ref{lem:isom_hhc_coder}, and so the definitions of
\ref{defn-A-infinity-algebra-morphism} make sense in this context.
\end{rems}

\section{Deformation theory of \Ai-algebras}
In this section we recall how the Hochschild
cochains control the infinitesimal deformation theory of an
\Ai-algebra. A goal is to give context and motivation for the
definition and use of Maurer-Cartan elements of dg-Lie algebras. The
reader uninterested in deformation theory only needs Definition
\ref{defn:MC-things}. There are no new results, and the approach follows 
\cite{MR981619, KSDefTheory, KellerDefTheory}; see also
\cite{MR1364455,MR1935035, MR1239562}. We assume that $\frac{1}{2} \in
k$.\footnote{This assumption can be removed by treating $\circ$ as a quadratic squaring map, as
  in \cite[\S 2]{MR0195995}.}

  \begin{defn} Let $l$ be a commutative $k$-algebra.
    \begin{enumerate}
    \item \emph{An $l$-family of
        \Ai-algebra structures on $A$} is an $l$-linear \Ai-algebra structure on
      $A \otimes l$. We set
\[\hhca l {A \otimes l} {A \otimes l} = \prod_{n \geq 1} \Hom {l}
  {(\Pi (A \otimes l))^{\otimes n}} {\Pi (A \otimes l)}.\] Using the isomorphism $\hhca l {A \otimes l} {A \otimes l}
  \cong \hhc A {A \otimes l},$  and denoting $l$ in string
  diagrams as $\tikz[baseline={([yshift = -.1cm]current bounding
    box.center)}]{\draw[parameter] (0,0) -- (0,.5)}$, 
we can write an $l$-family as
$\begin{tikzpicture}[inner sep =0mm]
                                                      \def\sepw{.2cm}
                                                      \def\w{.4cm}

                                                      \newBoxDiagram
                                                      [arm height =
                                                      .5cm, width =
                                                      \w, coord =
                                                      {($(\w,
                                                        0)$)},
                                                      decoration
                                                      yshift label =
                                                      .2cm, label =
                                                      {\bs \nu}]{f};
                                                      \printBox{f};
                                                      \printArm{f}{-1};
                                                      \printArm{f}{1};
                                                      \printArmSep{f};
                                                      \printOutput[/parameter
                                                      ]{f}{.4}
                                                      \printOutput{f}{-.4}
                                                    \end{tikzpicture}.$
      
\item If $\alpha: l \to l'$ is a morphism of commutative
      algebras, represented by a string diagram $\begin{tikzpicture}
    \newBoxDiagram [arm height = .5cm, width = .25cm, decoration
    yshift label = .05cm, box height = .5cm, label = {\tiny{\alpha}}]{f};
    \printBox{f}; \printOutput[/secparameter]{f}{0};\printArm[/parameter]{f}{0};
  \end{tikzpicture},$ then an $l$-family $\bs \nu$ gives rise to the
  $l'$-family
  $\begin{tikzpicture}[inner sep =0mm]
    \def\sepw{.2cm}
    \def\w{.4cm}
    
    \newBoxDiagram
    [arm height =
    .5cm, width =
    \w,
    decoration
    yshift label =
    .2cm, output
    height = 1.5cm, label =
    {\bs \nu}]{f};
    \printBox{f};
    \printArm{f}{-1};
    \printArm{f}{1};
    \printArmSep{f};
    \printOutput{f}{-.4}
    \newBoxDiagram [arm height = .5cm, width = .25cm,  coord =
    {(.4cm,
      0cm)},
    decoration yshift label = .05cm, box height = .5cm, label = {\tiny{\alpha}}]{g};
    \printBox{g}; \printOutput[/secparameter]{g}{0};\printArm[/parameter]{g}{0};
  \end{tikzpicture}.
  $
 Thus there is a functor,
      \begin{displaymath}
        \Aifam A: \Algk \to \Set
      \end{displaymath}
      \begin{displaymath}
        l \mapsto \{ \text{$l$-families of \Ai-algebra structures on $A$} \}.
      \end{displaymath}
    \end{enumerate}
  \end{defn}

   \begin{lem}[Yoneda]\label{lem:Yoneda-for-deformations}
    If the functor $\Aifam A$ is representable by a $k$-algebra $\luniv$ and an isomorphism of
    functors
  $\zeta: \Hom {\Algk} {\luniv} {-} \xra{\cong} \Aifam A,$
  then $\Spec {\luniv}$ is the moduli space of \Ai-algebra structures on $A$
  and $\muniv  = \zeta(1_{\luniv})$ is the universal family of
  \Ai-algebra structures.
\end{lem}
Indeed, for any commutative $k$-algebra $l$ and any $l$-family
  $\bs \nu \in \hhca {l} {A \otimes l} {A \otimes l}$, there exists a unique morphism
  $f: \luniv \to l$ such that $\bs \nu = \Aifam A (f)(\muniv)$ (one
  can take this as the definition of moduli space and universal family). In
  particular, the set of \Ai-algebra structures on $A$ corresponds to the
  set of $k$-morphisms $\luniv \to k.$

  If $A$ is a finitely generated
  graded projective $k$-module, and is concentrated in non-negative
  degrees, or in degrees at most $-2$, then
  $\Aifam A$ is representable. Indeed, set $L^{i} = \hhc A
  A_{-i}$ and $b = [-,-]: L^{1} \otimes L^1 \to L^2.$ By the
  assumptions on $A$, $L^1$ is a finitely generated projective
  $k$-module. Thus, writing $(-)^*$ for the $k$-dual, the natural map
  $\iota: L_1^* \otimes L_1^* \to (L_1 \otimes L_1)^*$ is an
  isomorphism. If we denote by $\sqr^*: L_2^* \to \Sym^2(L_1^*)$ the
  map $L_2^* \xra{b^*} (L_1 \otimes L_1)^* \xra{\iota^{-1}} L_1^*
  \otimes L_1^* \onto \Sym^2(L_1^*),$ then the algebra $\luniv :=
  \Sym^{\bullet}(L_1^*)/(\sqr^*(L_2^*))$ represents $\Aifam A$. If $k$
  is an algebraically closed field, then the closed points of
  $\Spec{\luniv}$ correspond to the Maurer-Cartan elements of $L^1 =
  \hhc A A_{-1}$, i.e., \Ai-algebra structures on $A$.

Regardless of whether the functor $\Aifam A$ is representable, we can
view the functor
as a generalized scheme. By Yoneda's Lemma,\footnote{A more general (and standard)
  version than the one quoted in Lemma
  \ref{lem:Yoneda-for-deformations}.} an $l$-family $\bs \nu$
corresponds to
the natural transformation $\bs \nu_*: h^l \to \Aifam A$ that sends
$\beta \in h^l(l') = \Hom
{\Algk} l {l'}$ to $\Aifam A(\beta)(\bs
\nu) \in \Aifam A(l')$. Given an \Ai-structure $\nu$ on $A$, we say
the $l$-family $\bs \nu$ contains $\nu$ if there is a natural
transformation $\epsilon^{*}: h^{k} \to h^{l}$ such that the following diagram
is commutative:
\begin{displaymath}
  \begin{tikzcd}
    h^{k} \ar[dr, "\nu_{*}"] \ar[dd, dashed, "\epsilon^{*}"']&\\[-15pt]
    &\Aifam A.\\[-15pt]
    h^{l} \ar[ur,"\bs \nu_{*}"']&[-15pt]
  \end{tikzcd}
\end{displaymath}
By Yoneda again, the transformation $\epsilon^{*}$ is determined by a
$k$-algebra morphism $\epsilon: l \to k.$ Unraveling this gives an algebraic definition of $l$-family that contains a
marked $k$-point $\nu.$
\begin{defn}\label{defn:deformation}
 Let $(A, \nu)$ be a nonunital \Ai-algebra and $\epsilon: l \to k$ a
 morphism of commutative $k$-algebras. An \emph{$(l, \epsilon)$-deformation
  of $(A, \nu)$} is an $l$-family
  $\bs \nu$ such that $\Aifam A(\epsilon)(\bs \nu) = \nu.$ In
  diagrams, this means,
  \begin{displaymath}
\begin{tikzpicture}[baseline={(current bounding box.center)}]
                      \def\w{.75cm}
                      \def\hgt{.42cm}
                      \def\highPos{0}
                      \def\relw{1}
                      \def\sepw{.1}
                      \def\smallsep{.08}
                      \def\modsep{.35}
                      \pgfkeys{lowBoxStyle/.style = {box height =
                          \hgt, output height = \hgt, width = \w,
                          coord =
                          {($(0,0)$)}, decoration yshift label = .2cm,
                          decoration yshift label = .05cm, arm height
                          = 3*\hgt, label = 1 \otimes \epsilon}}
                      \newBoxDiagram [/lowBoxStyle]{lowBox};
                      \printBox{lowBox};
%
                      \lowBox.coord at top(highcoor, \highPos);
                      \pgfkeys{highBoxStyle/.style = {width =
                          \w*\relw, box height = \hgt, arm height =
                          \hgt, output height = \hgt, coord =
                          {(highcoor)}, output height = \hgt, label =
                          {\bs \nu}}}
                      \newBoxDiagram[/highBoxStyle]{highBox};
                      \printBox{highBox}; \printArmSep{highBox};
                        \printArm{highBox}{-1};
                      \printArm{highBox}{1};
                      \printOutput{highBox}{-.4};
                      \printOutput[/parameter]{highBox}{.4};
                      \printOutput{lowBox}{0};
                    \end{tikzpicture}
                    \hspace{.2cm} = \hspace{.2cm}
                    \begin{tikzpicture}[baseline={(current bounding box.center)}]
                      \def\w{.75cm}
                      \def\hgt{.42cm}
                      \def\highPos{0}
                      \def\relw{1}
                      \def\sepw{.1}
                      \def\smallsep{.08}
                      \def\modsep{.35}
                       \pgfkeys{lowBoxStyle/.style = {box height =
                          \hgt, output height = \hgt, width = \w,
                          coord =
                          {($(0,0)$)}, decoration yshift label = .2cm,
                          decoration yshift label = .05cm, arm height
                          = 3*\hgt, label = \nu}}
                      \newBoxDiagram [/lowBoxStyle]{lowBox};
                      \print{lowBox};
                      \printArmSep{lowBox};
                      \printPeriod{lowBox};
                      \end{tikzpicture}
\end{displaymath}
We denote
by $\AugAlgk$ the category with objects pairs $(l, \epsilon)$ as
above, and
morphisms the algebra morphisms commuting with the augmentations. Set $\Aidef {(A, \nu)}: \AugAlgk
\to \Set$ to be the functor that sends $(l, \epsilon)$ to the set of $(l, \epsilon)$-deformations of $(A, \nu)$.
\end{defn}

If $\Aifam A$ is represented by $\luniv$, and
    $\alpha: \luniv \to k$ is the morphism corresponding to $\nu$,
    then one checks
    the augmented $k$-algebra $(\luniv, \alpha)$ represents
    $\Aidef {(A, \nu)}$. Regardless of the representability of
    $\Aidef {(A, \nu)},$ we can view the functor as describing the
    generalized scheme $\Aifam A$ near the $k$-point corresponding to $\nu.$ We can focus
  attention on the (generalized) infinitesimal neighborhoods of the $k$-point by restricting the
  domain of $\Aidef {(A,\nu)}$ to $\finAlgk$,  the full
  subcategory of $\AugAlgk$ with objects $(l, \epsilon),$ such that $l$
  is a finitely generated projective $k$-module and $(\ker
  \epsilon)^{N} = 0$ for some $N \geq 1$ (the last condition follows
  from the first if
  $k$ is a field,
  and $l$ is local).

\begin{defn}
An \emph{infinitesimal deformation} of a nonunital \Ai-algebra $(A, \nu)$ is
  an $(l, \epsilon)$-deformation with $(l, \epsilon)$ an object of
  $\finAlgk$. The corresponding functor is denoted $\infAidef {(A,\nu)} = \Aidef {(A,\nu)}|_{\finAlgk}: \finAlgk \to \Set.$
\end{defn}

If $\Aidef {(A, \nu)}$ is represented by an augmented $k$-algebra
$(\luniv, \epsilon),$ such that $\luniv/(\ker \epsilon)^{n}$ is in
$\finAlgk$ for all $n$ (this holds when $k$ is a field and
$\luniv$ is noetherian), then
$\infAidef {(A,\nu)}$ is pro-represented by the completion
$\ds \underleftarrow{\lim}_{n \geq 0} \, \luniv/(\ker
\epsilon)^{n}$. Indeed, the canonical morphism
of functors,
\[\ds \underrightarrow{\colim}_{n} \Hom {\Algk}
{\luniv/(\ker \epsilon)^{n}} {-} \to \Hom {\Algk} \luniv -,\] is
easily checked to be an
isomorphism on $\finAlgk$, and this is the definition of
pro-representability (see e.g., \cite[\S 2]{MR1603480}).
If $\luniv$ is Noetherian, then $\operatorname{Spf}$ of the completion
is the formal completion of $\Aidef {(A, \nu)}$ along the
$k$-point of $\nu$.

Since $\infAidef {(A, \nu)}$ preserves limits, a result of Grothendieck, \cite[Corollary to
3.1]{MR1603480}, shows that it
is pro-representable, but the result does not describe the
pro-representing object. Drinfeld showed \cite{MR3285856}, in case $A$ is concentrated in non-negative
  degrees, or in degrees at most $-2$, and degreewise
finitely generated, that the degree
zero Lie algebra cohomology of the Hochschild cochains pro-represents
$\infAidef {(A, \mu)}$. He put the answer in the following more general
context, which shows where the finiteness assumptions on $A$
enter. First note that $(\finAlgk)^{\op} = \finCoalgk,$ the category
of cocomplete cocommutative coalgebras that are finitely generated projective
$k$-modules. The ind-completion of this category is equivalent to $\AugCoalgk,$ the category of
all cocommutative cocomplete coalgebras that are projective
$k$-modules\footnote{This holds since the cocomplete coalgebras are
  closed under colimits, and every element in a cocomplete coalgebra
  is contained in a sub-coalgebra that is a finitely generated
  projective $k$-module; see e.g., \cite[Chapter 1, \S 6]{MR0344261}, \cite{MR0252485}}. It follows that
$(\AugCoalgk)^{\op}$ is equivalent to the pro-completion of $\finAlgk.$
Any functor $\finAlgk \to \Set$ that preserves limits extends uniquely
to a limit preserving functor on the pro-completion, and thus such a functor is
pro-representable exactly when the corresponding functor on coalgebras
is representable. The dual of the coalgebra is then a
pro-representing object. Kontsevich and Soibelman
develop this point of view extensively in \cite{KSDefTheory}.

The functor $\infAidef {(A, \mu)}$
extends to a functor
on the category of dg-coalgebras (whose
underlying graded coalgebra is in $\AugCoalgk$). We denote this
category by $\dgAugCoalgk$. When $k$ is a field
of characteristic zero, Quillen \cite{MR0258031}, assuming certain
boundedness conditions later removed by
Hinich \cite{HinichJPAA01}, defined a model category structure on
$\dgAugCoalgk,$ and showed there is an equivalence,
\begin{displaymath}
\xymatrix{ \Ho(\dgAugCoalgk) \ar@<-1.2ex>[rr]_{\cobar{}} &&
\Ho(\dgLieAlgk), \ar@<-1.2ex>[ll]_{\Bar {}}^{\cong}}
\end{displaymath}
between the homotopy
category of this model category and the homotopy category of dg-Lie
algebras. The equivalence is given by the commutative versions of the bar and
cobar constructions. The functor $\infAidef {(A, \mu)}$ induces a functor $\Ho(\dgAugCoalgk)^{\op} \to \Set.$ Such a
functor is
representable exactly when its representable by $\Bar L,$ for some
dg-Lie algebra $L,$ by the above equivalence. Unwinding definitions, the functor represented by
$\Bar L,$ restricted to $\finAlgk,$ is
the following (see
\cite[\S 2.7]{KellerDefTheory} for details of the unwinding).

\begin{defn}\label{defn:MC-things} Let $(L, \delta)$ be a dg-Lie algebra.
  \begin{enumerate}
  \item The \emph{Maurer-Cartan elements} are
\begin{displaymath}
\MC ( L, \delta )
        := \{ v \in L_{-1} \, | \, \delta(v) + \frac{1}{2}[ v, v] = 0 \}.
\end{displaymath}
  \item The \emph{Maurer-Cartan functor} is
    \begin{displaymath}
      \begin{gathered}
        \MCFunct {(L, \delta)} : \finAlgk \to \Set\\
        (l, \epsilon) \mapsto \MC ( L \otimes \o l, \delta \otimes 1)
        := \{ v \in (L \otimes \o l)_{-1} \, | \, (\delta \otimes
        1)(v) + \frac{1}{2}[ v, v] = 0 \},
      \end{gathered}
    \end{displaymath}
    where $\o l = \ker \epsilon$ and $L \otimes \o l$ has the induced
    bracket $[v \otimes x, v' \otimes x'] = [v, v'] \otimes xx'$.
    
\item A functor $F: \finAlgk \to \Set$ is \emph{controlled by the dg-Lie
      algebra $(L, \delta)$} if there is an equivalence
    $\MCFunct {(L,\delta)} \xra{\cong} F.$
  \end{enumerate}
\end{defn}

We assumed that $k$ was a characteristic zero field in the paragraph above, but Definition \ref{defn:MC-things} makes sense over any
commutative ring (with $1/2 \in k$). The most natural context for this
story is derived algebraic
geometry, see \cite{MR2827833, 1401.1044}. Staying at a more concrete level,
Schectmann shows \cite[Theorem 2.5]{math/9802006} that if $L$ is concentrated in strictly
positive cohomological degrees, and $L^{1}$ is finite dimensional, then the
zeroth cohomology of the bar construction of $L$ pro-represents the Maurer-Cartan
functor.

One motivation behind the Maurer-Cartan approach to deformation theory is that often there is an
apparent dg-Lie algebra controlling a given functor, for instance
$\infAidef {(A, \nu)}$. We now show that the dg-Lie algebra of Hochschild
cochains controls it (this is classical, but we give details for lack
of a reference at this level of generality). Paired with  \cite[Theorem
2.5]{math/9802006}, it recovers Drinfeld's description of the
pro-representing object of $\infAidef {(A, \mu)},$ assuming certain
finiteness conditions.

\begin{prop}\label{prop:MC-equiv-to-defs}
 Let $(A, \nu)$ be a nonunital \Ai-algebra and
  $(\hhc A A, [\nu,-])$ the Hochschld cochains. The following is an equivalence:
  \begin{align*}
    \MCFunct {(\hhc A A, [\nu, -])} \to \infAidef
    {(A,\nu)}\\
    \t \nu \mapsto \theta_{l}(\t \nu + \nu \otimes 1) = \bs \nu,
  \end{align*}
  where $\theta_{l}$ is the canonical morphism of graded Lie algebras,
  \begin{displaymath}
   \begin{gathered}
    \theta_{l}: \hhc A A \otimes l \to \hhca l {A \otimes l} {A
      \otimes l}\\
    (\nu^{n}) \otimes y \mapsto ([a_1 \otimes x_1 | \ldots | a_n
    \otimes x_n] \mapsto \nu^n[a_1|\ldots|a_n] \otimes yx_1\ldots x_n
    ),
  \end{gathered}
\end{displaymath}
with the induced bracket
on the source and the Gerstenhaber bracket on the target.
\end{prop}


\begin{proof}
  Let $(l, \epsilon)$ be an object of $\finAlgk$ and set
  $\o l = \ker \epsilon.$ By definition,
$$\MCFunct {(\hhc A A, [\nu, -])}(l, \epsilon)
= \MC( \hhc A A \otimes \o l, [\nu \otimes 1, -]),$$ and one checks
the following is a bijection,
\begin{align*}
  \MC( \hhc A A \otimes \o l, [\nu \otimes 1, -]) &\xra{\cong} \MC(1
                                                    \otimes \epsilon)^{-1}(\nu) \subseteq \MC(\hhc A
                                                    A \otimes l, 0)\\
  \t \nu &\mapsto \t \nu + \nu \otimes 1.
\end{align*}
Since $l$ is a finite rank projective $k$-module, $\theta_{l}$ is an
isomorphism, and thus induces a bijection $ \MC(\hhc A A \otimes l, 0)
\xra{\cong} \MC(\hhca l {A \otimes l} {A \otimes l},0).$ This restricts to a bijection
$\MC(1 \otimes \epsilon)^{-1}(\nu) \xra{\cong} \Aifam
A(\epsilon)^{-1}(\nu) = \infAidef {(A,\nu)}(l).$
\end{proof}

One is most often interested in families and deformations modulo the
following.

\begin{defn}
  An \emph{isomorphism between $l$-families} is an $l$-linear \Ai-isomorphism.
  An \emph{equivalence of deformations} is an isomorphism of families
  that reduces to the identity on $A$.
\end{defn}

One can consider isomorphism and equivalence classes using
the following group functors. For $l$ a commutative $k$-algebra, set $H_{A}(l)= \{ \t g = (\t g^n) \in \hhca l {A \otimes l} {A \otimes
      l}_{0} \, | \, \t g^1 \text{ is an isomorphism}\}$. There is an equality $\Aut_{\Coalgk
  [l]}(\Tco{\Pi A \otimes l}) =  \Psi^{-1}(H_{A}(l))$, where $\Psi^{-1}$ is defined in \ref{lem:isom_hhc_coder}.(2); see
\cite[Proposition 2.5]{MR1989615} for a proof. Thus $H_{A}$ is a group functor
$ H_{A}: \Algk \to \Group$. Using this, set
\begin{displaymath}
  \begin{gathered}
    G_A: \AugAlgk \to \Group\\
    G_A(l,\epsilon) = \{ \t g \in H_A(l) \, | \,(1 \otimes
    \epsilon)_{*}(\t g) = 1 \}.
  \end{gathered}
\end{displaymath}
There is an action $H_A \times \Aifam A \to \Aifam A$, defined using
the isomorphisms of \ref{lem:isom_hhc_coder}, 
whose quotient functor sends $l$ to the set of isomorphism classes
  of $l$-families of \Ai-structures on $A$. If $(A, \nu)$ is an
  \Ai-structure, the action of $H_{A}$ restricts to an action $G_A \times \Aidef {(A,\nu)} \to \Aidef {(A,\nu)}$
  whose quotient functor sends $(l, \epsilon)$ to the set of equivalence classes of
        $(l,\epsilon)$-deformations of $(A,\nu)$.

\begin{cor}
Let $(A, \nu)$ be a nonunital \Ai-algebra. The following is a
bijection,
\begin{displaymath}
  \begin{gathered}
H^{1}(\hhc A A,[\nu, -]) \to {\infAidef {(A,
  \nu)}(k[t]/(t^{2}))}/{\sim},\\
\nu \mapsto \theta(\nu \otimes t + \nu \otimes 1),
\end{gathered}
\end{displaymath}
where the right side is the set of equivalence classes of $k[t]/(t^{2})$-deformations.
\end{cor}

\begin{proof}
Let $Z^{1}$ be the cohomological degree 1 cycles of the complex $(\hhc
A A, [\nu, -]).$ The assignment $\nu \mapsto \nu \otimes t$
is a bijection $Z^{1} \to \MCFunct {(A, \nu)}(k[t]/(t^{2})) = \MC(\hhc A A
\otimes kt, [\nu \otimes 1, -])$. Thus by \ref{prop:MC-equiv-to-defs},
the assignment $\nu \mapsto \theta(\nu \otimes t + \nu \otimes 1)$ is a
bijection $Z^{1} \xra{\cong} \infAidef {(A, \nu)}(k[t]/(t^{2})).$

We now claim that for $\nu, \nu' \in Z^{1}$, the deformations
$\theta(\nu \otimes t + \nu \otimes 1)$ and $\theta(\nu' \otimes t +
\nu \otimes 1)$ are equivalent if and only if $\theta((\nu' - \nu)
\otimes t) = \theta([\mu, \alpha] \otimes t),$ for some $\alpha \in
\hhc A A_{0}.$ The claim finishes the proof, since then $\theta(\nu \otimes t + \nu \otimes 1)$ and $\theta(\nu' \otimes t +
\nu \otimes 1)$ are equivalent if and only if $\nu' - \nu = [\mu,
\alpha]$ for some $\alpha,$ using that
$\theta$ is a bijection. This last condition says exactly that $[\nu]
= [\nu'] \in H^{1}(\hhc A A,[\nu, -]).$ To see the claim, note there is a
bijection $\xi: \hhc A A_{0} \xra{\cong} G_{A}(k[t]/(t^{2})), \alpha
\mapsto \theta(1 \otimes 1 + \alpha \otimes t),$ and the action of
$\xi(\alpha)$ on $\theta(\nu \otimes t + \nu
\otimes 1)$ is $\theta(\nu \otimes t + (\nu\circ \alpha - \alpha \circ
\nu) \otimes t + \nu \otimes 1).$
\end{proof}

\begin{rems}
Let $(L, \delta)$ be a dg-Lie algebra, and assume that $k$ contains
$\bQ.$ For any $(l, \epsilon) \in \finAlgk,$ the graded Lie algebra $L
\otimes \o l$ is nilpotent, and thus we can define its exponential,
which makes it a
group. This gives a functor
$\finAlgk \to \Group.$ This functor acts on the
Maurer-Cartan functor of $(L, \delta)$, and is usually the group
functor one hopes to quotient by (in case the dg-Lie algebra is the Hochschild
cochains, the group functor agrees with $G_{A}|_{\finAlgk}$). This is another advantage of
the Maurer-Cartan formalism (in characteristic zero): the group we
hope to quotient by is built into the Lie algebra.  This
point of view is due to Deligne, see \cite{MR972343, KSDefTheory}.
\end{rems}

\section{Strictly unital \Ai-algebras}
In this section, given a graded module $A$, we construct a dg-Lie
algebra whose Maurer-Cartan elements are the strictly unital
\Ai-structures on $A$. We first use this to recover Positselski's
construction of a functorial curved bar construction from a strictly unital
\Ai-algebra, and then use it to show that the reduced Hochschild
cochains control infinitesimal strictly unital deformations.

\subsection{Characterization of strictly unital structures}
\begin{defn}
  Let $A, B$ be graded modules with fixed elements
  $1 \in A_0, 1 \in B_{0}.$
  \begin{enumerate}
  \item An element $\nu = (\nu^n) \in \hhc A A_{-1}$ is \emph{strictly
      unital} (with respect to $1 \in A_{0}$) if
    \[\nu^2[1|a] = a = (-1)^{|a|} \nu^2 [a|1]\] and
    $\nu^n [a_1 | \ldots |a_{i}|1|a_{i+1}|\ldots|a_{n-1}] = 0$ for all
    $a, a_1, \ldots, a_{n-1} \in A,$ where $n \neq 2$ and
    $0 \leq i \leq n$. If $\nu$ is also an \Ai-algebra structure, we say
    $(A, \nu)$ is a strictly unital \Ai-algebra.
  
  \item An element $f = (f^{n}) \in \hhc A B_{0}$ is \emph{strictly
      unital} if $f^{1}[1] = [1]$ and
    $f^{n}[a_1|\ldots |a_i|1|a_{i+1}|\ldots|a_{n-1}] = 0$ for all
    $a_1, \ldots, a_{n-1} \in A,$  $n \geq 2$. If $f$ is also an \Ai-morphism, we say it is a
    strictly unital \Ai-morphism.
  \end{enumerate}
\end{defn}

For our main results we need to place a further assumption on the pair $(A, 1)$ (that is automatically
satisfied when $k$ is a field).

\begin{defn}
   A \emph{split element} of a graded module $A$ is an element that
  generates a rank one free module. A \emph{graded module with split
    element} is a pair $(A, 1)$ with $1$ a split element in $A$, and a fixed (unlabeled) splitting $A \to k$ of the
  inclusion $k \to A, 1 \mapsto 1.$ An \emph{\Ai-algebra with split
    unit} is a triple $(A,1, \nu)$, such that $(A, 1)$ is a
  graded module with split element, and $(A, \nu)$ is a strictly unital \Ai-algebra  (with respect to
  $1$). If $(A, 1)$ is a graded module with split element, we set $\o A =
A/(k \cdot 1).$ We  consider this as a submodule $\o A
\subseteq A$ via the fixed splitting of $1.$
\end{defn}

If $(A,1)$ and $(B,1)$ are modules with
split elements, then strictly unital elements $f \in \hhc {A} B_0$ are
assumed to preserve the fixed splittings.
In string diagrams,
$\tikz[baseline={([yshift = -.1cm]current bounding
  box.center)}]{\draw[thick] (0,0) -- (0,.5)}$ represents $\Pi \o A$
(previously it denoted $\Pi A$), $\tikz[baseline={([yshift = -.1cm]current bounding
  box.center)}]{\draw[bmod, line width = .5mm] (0,0) -- (0,.5)}$ represents $\Pi \o B,$ and
$\begin{tikzpicture}[baseline={([yshift = -.1cm]current bounding
    box.center)}]\draw[unit, line width = .35mm] (0,0) -- (0,.5);\end{tikzpicture}$
represents $\Pi k$.
  
\begin{defn}\label{def-msu} An \emph{\Ai-algebra with split unit} is a
  triple $(A,1,\nu)$ with $(A,1)$ a module with split element and
  $(A,\nu)$ a strictly unital \Ai-algebra with respect to
  $1$. The \emph{trivial \Ai-algebra with split unit},
  denoted $(A, 1, \msu)$, is defined by $\msu^{n} = 0$ for $n\neq 2$ and
  \begin{displaymath}
    \newcommand{\boxht}{.5cm}
    \newcommand{\hgt}{.6cm}
    \newcommand{\wdth}{.6cm}
    \msu^2 =
    \begin{tikzpicture}
      \newlength{\lengg} \setlength{\lengg}{0pt}\newBoxDiagram[width =
      \wdth/2, box height = \boxht, arm height = \hgt, output height =
      \hgt, label = s^{-1}, coord = {($(0,
        \lengg)$)}]{s} \printNoArms{s}
      \printLabelOnArm{s}{1}{\Pi}\setlength{\lengg}{\boxht + \hgt}
      \newBoxDiagram[width = \wdth, box height = \boxht, arm height =
      \hgt, output height = \hgt,output label = \Pi, label =
      \small{(\cong)}, coord = {($(0,
        \lengg)$)}]{isom} \printNoArms{isom} \printArm{isom}{1}
      \printArm[/unit]{isom}{-1}
    \end{tikzpicture} \hspace{.1cm} - \hspace{.1cm}
    \begin{tikzpicture}
      \setlength{\lengg}{0pt}\newBoxDiagram[width = \wdth/2, box
      height = \boxht, arm height = \hgt, output height = \hgt, label
      = s^{-1}, coord = {($(0,
        \lengg)$)}]{s} \printNoArms{s}
      \printLabelOnArm{s}{1}{\Pi}\setlength{\lengg}{\boxht + \hgt}
      \newBoxDiagram[width = \wdth, box height = \boxht, arm height =
      \hgt, output height = \hgt,output label = \Pi, label =
      \small{(\cong)}, coord = {($(0,
        \lengg)$)}]{isom} \printNoArms{isom} \printArm{isom}{-1}
      \printArm[/unit]{isom}{1}
    \end{tikzpicture} \hspace{.1cm} + \hspace{.1cm}
    \begin{tikzpicture}
      \setlength{\lengg}{0pt}\newBoxDiagram[width = \wdth/2, box
      height = \boxht, arm height = \hgt, output height = \hgt, label
      = s^{-1}, coord = {($(0,
        \lengg)$)}]{s} \printNoArms[/unit output]{s}
      \printLabelOnArm{s}{1}{\Pi}\setlength{\lengg}{\boxht + \hgt}
      \newBoxDiagram[width = \wdth, box height = \boxht, arm height =
      \hgt, output height = \hgt,output label = \Pi, label =
      \small{(\cong)}, coord = {($(0,
        \lengg)$)}]{isom} \printNoArms[/unit output]{isom}
      \printArm[/unit]{isom}{1} \printArm[/unit]{isom}{-1}
    \end{tikzpicture}
    \hspace{.1cm} = \hspace{.1cm}
    \begin{tikzpicture}
      \def\sepw{.2cm} \def\leftw{.4cm} \def\w{.5cm} \def\rightw{.6cm}
                                                  
      \newBoxDiagram [arm height = .5cm, width = \w, coord =
      {($(\sepw + \leftw + \w,
        0)$)}, decoration yshift label = .2cm, label = \msu^2]{f};
       \printNoArms{f};
                                                      \printArm[/unit]{f}{-1};
                                                      \printArm{f}{1};
                                                    \end{tikzpicture}
                                                    \hspace{.1cm} + \hspace{.1cm}
    \begin{tikzpicture}
      \def\sepw{.2cm} \def\leftw{.4cm} \def\w{.5cm} \def\rightw{.6cm}
      \newBoxDiagram [arm height = .5cm, width = \w, coord =
      {($(\sepw + \leftw + \w,
        0)$)}, decoration yshift label = .2cm, label = \msu^2]{f};
       \printNoArms{f};
                                                      \printArm{f}{-1};
                                                      \printArm[/unit]{f}{1};
                                                    \end{tikzpicture}
                                                                                                        \hspace{.1cm} + \hspace{.1cm}
    \begin{tikzpicture}
      \def\sepw{.2cm} \def\leftw{.4cm} \def\w{.5cm} \def\rightw{.6cm}
      \newBoxDiagram [arm height = .5cm, width = \w, coord =
      {($(\sepw + \leftw + \w,
        0)$)}, decoration yshift label = .2cm, label = \msu^2]{f};
       \printNoArms[/unit output]{f};
                                                      \printArm[/unit]{f}{-1};
                                                  \printArm[/unit]{f}{1};
    \end{tikzpicture}
    \in \hhn 2 A A_{-1},
  \end{displaymath}
  where $(\cong)$ denotes the following canonical isomorphisms, respectively: $\Pi k
  \otimes \Pi \o A \xra{\cong} \Pi(k \otimes \Pi \o A) = \Pi^{2} \o
  A$; $\Pi \o A \otimes \Pi k \xra{\cong} \Pi( \Pi \o A \otimes k) =
  \Pi^{2} \o A;$ $\Pi k \otimes \Pi k \xra{\cong} \Pi(k \otimes \Pi k)
  = \Pi^{2} k$ (see
  \ref{sect:notation}.(4) for signs). One checks (carefully,
  evaluating on elements) that
  $\msu \circ \msu = 0,$ and that $\msu$ is strictly unital, thus $(A,1,\msu)$ is an \Ai-algebra with split
  unit.  If $B$ is a graded module with fixed element $1 \in B_{0}$,
  the \emph{trivial strictly unital morphism $\gsu: A \to B$} is
  $\gsu^{1} = \Pi A \onto \Pi k \to \Pi B$ and $\gsu^n = 0$
  for $n \geq 2.$
\end{defn}

\begin{lem}
  \label{lem:desc-of-strictly-unital}
  Let $(A, 1)$ be a module with split element. Every strictly unital element in
  $\hhc A A_{-1}$ is of the form $\mu + \msu$ for a unique
  $\mu \in \hhc {\o A} {A}_{-1}.$ If $B$ is another graded module with
  a fixed element $1 \in B_{0}$, every strictly unital element in
  $\hhc A B_{0}$ is of the form $g + \gsu$ for a unique
  $g \in \hhc {\o A} B_{0}$.
\end{lem}

\begin{proof}
For a stricty unital element $\nu \in \hhc A A_{-1},$ set $\mu = \nu -
\msu \in \hhc A A_{-1}.$ By definition, $\mu$ is zero on any term
containing a $1,$ and thus $\mu \in \hhc {\o A} A_{-1}.$ The proof for
morphisms is similar (and easier).
\end{proof}

\begin{rems}
If we replace the category of $k$-modules by a symmetric monoidal
category, we can define $\msu$ using the diagrams above (where $k$ is
the unit of the category), and
use the lemma to define strictly unital elements of $\hhc A
A_{-1}$ and $\hhc A B_{0}$, when $A, B$ are objects in the category.
\end{rems}

We will use without remark that if $(A, 1)$ is a graded module
with split element, the splitting $A = \o A \oplus k$ induces a
splitting $\hhc {\o A} A = \hhc {\o A} {\o A} \oplus \hhc {\o A} k.$

\begin{defn}
  A strictly unital element $\mu + \msu \in \hhc {A} A_{-1},$ with
  $\mu = \o \mu + h \in \hhc {\o A}{\o A}_{-1} \oplus \hhc {\o A} k_{-1} = \hhc
  {\o A} A_{-1},$ is \emph{augmented} if $h = 0,$ i.e., if $\mu$ is in $\hhc
  {\o A} {\o A}.$ (In this case, if $\mu + \msu$ is an \Ai-algebra structure, the fixed
  splitting $A \to k$ is a strict \Ai-morphism, called the
  augmentation.)
\end{defn}

We note the term $h$ measuring the lack of augmentation is in
\[\ds
\hhc {\o A} k_{-1} = \prod_{n \geq 1} \Hom {} {(\Pi \o A)^{\otimes n}}
{\Pi k}_{-1} = \prod_{n \geq 1}
\operatorname{Hom} \left ( \left ( \left (\Pi \o
  A\right)^{\otimes n}\right )_{2},k \right ).\]

\begin{ex}\label{ex:reslns-aug}
Let $(A, 1)$ be a graded module with split element such that $A_{i}
= 0$ for $i < 0,$ $A_{0} = k,$ and $1 \in A_{0}$ is the unit in $k.$
Let $\mu = \o \mu + h \in \hhc {\o A} A$ be an element such that $(A,
1, \nu = \mu + \msu)$ is an \Ai-algebra with split unit.
Since $\o A = A_{\geq 1},$ it follows that $\left ( \left (\Pi \o
  A\right)^{\otimes n}\right )_{2} = 0$ for $n \geq 2.$ Thus $h^{n} =
0$ for all $n \geq 2;$ the map $h^{1}$ makes the following diagram commutative:
\begin{displaymath}
  \begin{tikzpicture}
      \matrix (m) [matrix of math nodes,row sep=3em,column sep=4em,minimum width=2em]{
     (\Pi A)_{2} & (\Pi k)_{1} \\
    A_{1} & A_{0}. \\
  };
 \path[-stealth, auto] (m-1-1) edge node {$\cong$}
 (m-2-1);
  \path[-stealth, auto]     (m-1-1)   edge node[swap] {$s^{-1}$} (m-2-1);
 \path[-stealth, auto] (m-1-2) edge node[swap] {$\cong$}
 (m-2-2);
   \path[-stealth, auto]     (m-1-2)   edge node {$s^{-1}$} (m-2-2);
 \path[-stealth, auto]     (m-1-1)  edge node {$h^1$} (m-1-2);
  \path[-stealth, auto]     (m-2-1) edge node[swap] {$(m^1)_1$} (m-2-2);
  \end{tikzpicture}
\end{displaymath}
Here $m^1$ is $s^{-1} \nu^1 s$, see \ref{ex:dg-alg-is-ainf-alg}. Note
that the image of $(m^1)_1$ is an ideal $I$ in $k = A_0$. The \Ai-algebra $(A, 1, \mu + \msu)$ is
augmented exactly when $h^1 = 0,$ i.e., $I = 0.$
\end{ex}

By Lemma \ref{lem:diagrams-for-sual} below, $[\o \mu, \msu] = 0$ for all
$\o \mu \in \hhc {\o A} {\o A}_{-1}$, and it follows that Maurer-Cartan elements of
$\hhc {\o A} {\o A}$ (i.e., nonunital \Ai-algebra structures on
$\o A$) correspond  to
augmented \Ai-algebra structures on $A$, via the map $\o \mu \mapsto
\o \mu + \msu$. The following generalizes this to all
strictly unital \Ai-algebras.

\begin{thm}\label{thm:main-thm-sunal}
  Let $(A, 1)$ be a module with split element, and $\msu \in \hhc A A_{-1}$ the trivial
  strictly unital \Ai-algebra structure defined in \ref{def-msu}. The submodule $\hhc {\o A} A$
  is a graded Lie subalgebra of $\hhc A A$, and the derivation
  $[\msu, -]$ of $\hhc A A$ preserves $\hhc {\o A} A$. The
  Maurer-Cartan elements of the resulting dg-Lie algebra
  $(\hhc {\o A} A, [\msu, -])$ correspond to the strictly unital
  \Ai-structures on $(A,1)$ via
  \begin{displaymath}
    \o \mu + h \mapsto \o \mu + h + \msu.
  \end{displaymath}
  (See \ref{defn:MC-things} for the definition of Maurer-Cartan
elements.)
\end{thm}

 We need the following lemma for the proof of Theorem \ref{thm:main-thm-sunal}.

\begin{lem}
  \label{lem:diagrams-for-sual} Let
  $\o \mu, \o \mu' \in \hhc {\o A} {\o A}_{-1}$ and $h, h' \in \hhc {\o A}
  {k}_{-1}$ be arbitrary elements. The following hold:
  \begin{enumerate}
  \item $[\o \mu, \msu] = 0 = [h, h'].$ \vspace{.1in}
  \item
    $[\msu, h] = \begin{tikzpicture}[baseline={(current bounding
        box.center)}] \def\w{1cm} \def\gstart{-.35} \def\relw{.6}
      \def\sepw{.1} \newBoxDiagram[width/.expand once = \w, arm height
      = 1.5cm, label = \msu^{2}]{f}; \printNoArms{f}; \f.coord at
      top(gcoor, \gstart); \newBoxDiagram[width = \w*\relw, arm height
      = .5cm, coord = {(gcoor)}, output height = .5cm, label = h]{g};
      \print[/unit output]{g}; \printArmSep{g}; \printArm{f}{.8}
      \def\smallsep{.08}
      \pgfmathsetmacro{\var}{\gstart - \relw - \sepw}
      \pgfmathsetmacro{\var}{-1+\smallsep}
      \pgfmathsetmacro{\vara}{\gstart - \relw - \sepw - \smallsep}

      \pgfmathsetmacro{\var}{\gstart + \relw + \sepw}
      %
      \pgfmathsetmacro{\var}{1-\smallsep}
      \pgfmathsetmacro{\vara}{\gstart + \relw + \sepw + \smallsep}
    \end{tikzpicture} \hspace{.1cm} + \hspace{.1cm}
    \begin{tikzpicture}[baseline={(current bounding box.center)}]
      \def\w{1cm} \def\gstart{.35} \def\relw{.6} \def\sepw{.1}
      \newBoxDiagram[width/.expand once = \w, arm height = 1.5cm,
      label = \msu^{2}]{f}; \printNoArms{f}; \f.coord at top(gcoor,
      \gstart); \newBoxDiagram[width = \w*\relw, arm height = .5cm,
      coord = {(gcoor)}, output height = .5cm, label = h]{g};
      \print[/unit output]{g}; \printArmSep{g}; \printArm{f}{-.8}
      \def\smallsep{.08}
      \pgfmathsetmacro{\var}{\gstart - \relw - \sepw}
      \pgfmathsetmacro{\var}{-1+\smallsep}
      \pgfmathsetmacro{\vara}{\gstart - \relw - \sepw - \smallsep}

      \pgfmathsetmacro{\var}{\gstart + \relw + \sepw}
      %
      \pgfmathsetmacro{\var}{1-\smallsep}
      \pgfmathsetmacro{\vara}{\gstart + \relw + \sepw + \smallsep}
    \end{tikzpicture}
    \in \hhc {\o A} {\o A}.$

  \item $\displaystyle [\o \mu, \o \mu'] = \sum_j \,
    \, \begin{tikzpicture}[baseline={(current bounding box.center)}]
      \def\w{1.3cm} \def\gstart{-.1} \def\relw{.5} \def\sepw{.1}
      \newBoxDiagram[width/.expand once = \w, arm height = 1.5cm,
      label = \o \mu]{f}; \print{f}; \f.coord at top(gcoor,
      \gstart); \newBoxDiagram[width = \w*\relw, arm height = .5cm,
      coord = {(gcoor)}, output height = .5cm, label = \o \mu']{g};
      \print{g}; \printArmSep{g};

      \def\smallsep{.08}
      \pgfmathsetmacro{\var}{\gstart - \relw - \sepw}
      \printArm{f}{\var}; \printDec[/decorate, decorate left = -1,
      decorate right = \var]{f}{j};
      \pgfmathsetmacro{\var}{-1+\smallsep}
      \pgfmathsetmacro{\vara}{\gstart - \relw - \sepw - \smallsep}
      \printArmSep[arm sep left = \var, arm sep right = \vara]{f}

      \pgfmathsetmacro{\var}{\gstart + \relw + \sepw}
      \printArm{f}{\var}
      \pgfmathsetmacro{\var}{1-\smallsep}
      \pgfmathsetmacro{\vara}{\gstart + \relw + \sepw + \smallsep}
      \printArmSep[arm sep left = \vara, arm sep right = \var]{f}
    \end{tikzpicture} \hspace{.1cm} + \hspace{.1cm} \sum_j \,
    \, \begin{tikzpicture}[baseline={(current bounding box.center)}]
      \def\w{1.3cm} \def\gstart{-.1} \def\relw{.5} \def\sepw{.1}
      \newBoxDiagram[width/.expand once = \w, arm height = 1.5cm,
      label = \o \mu']{f}; \print{f}; \f.coord at top(gcoor,
      \gstart); \newBoxDiagram[width = \w*\relw, arm height = .5cm,
      coord = {(gcoor)}, output height = .5cm, label = \o \mu]{g};
      \print{g}; \printArmSep{g};

      \def\smallsep{.08}
      \pgfmathsetmacro{\var}{\gstart - \relw - \sepw}
      \printArm{f}{\var}; \printDec[/decorate, decorate left = -1,
      decorate right = \var]{f}{j};
      \pgfmathsetmacro{\var}{-1+\smallsep}
      \pgfmathsetmacro{\vara}{\gstart - \relw - \sepw - \smallsep}
      \printArmSep[arm sep left = \var, arm sep right = \vara]{f}

      \pgfmathsetmacro{\var}{\gstart + \relw + \sepw}
      \printArm{f}{\var}
      \pgfmathsetmacro{\var}{1-\smallsep}
      \pgfmathsetmacro{\vara}{\gstart + \relw + \sepw + \smallsep}
      \printArmSep[arm sep left = \vara, arm sep right = \var]{f}
    \end{tikzpicture} \in \hhc {\o A} {\o A}.$

  \item
    $\displaystyle [\o \mu, h] = \sum_j \,
    \, \begin{tikzpicture}[baseline={(current bounding box.center)}]
      \def\w{1.3cm} \def\gstart{-.1} \def\relw{.5} \def\sepw{.1}
      \newBoxDiagram[width/.expand once = \w, arm height = 1.5cm,
      label = h]{f}; \print[/unit output]{f}; \f.coord at
      top(gcoor, \gstart); \newBoxDiagram[width = \w*\relw, arm height
      = .5cm, coord = {(gcoor)}, output height = .5cm, label = \o
      \mu]{g}; \print{g}; \printArmSep{g};

      \def\smallsep{.08}
      \pgfmathsetmacro{\var}{\gstart - \relw - \sepw}
      \printArm{f}{\var}; \printDec[/decorate, decorate left = -1,
      decorate right = \var]{f}{j};
      \pgfmathsetmacro{\var}{-1+\smallsep}
      \pgfmathsetmacro{\vara}{\gstart - \relw - \sepw - \smallsep}
      \printArmSep[arm sep left = \var, arm sep right = \vara]{f}

      \pgfmathsetmacro{\var}{\gstart + \relw + \sepw}
      \printArm{f}{\var}
      \pgfmathsetmacro{\var}{1-\smallsep}
      \pgfmathsetmacro{\vara}{\gstart + \relw + \sepw + \smallsep}
      \printArmSep[arm sep left = \vara, arm sep right = \var]{f}
    \end{tikzpicture} \in \hhc {\o A} k.$
  \end{enumerate}
\end{lem}

\begin{proof}
All of the equalities are automatic except for the first half of (1),
$[\o \mu, \msu ] = 0,$ and (2). To show (1), one can first check that for
all $j \geq 1,$ the following holds:
\begin{displaymath}
\begin{tikzpicture}[baseline={([yshift = -.35cm]current bounding
  box.center)}]
                                                      \def\sepw{.2cm}
                                                      \def\leftw{.4cm}
                                                      \def\w{.5cm}
                                                      \def\rightw{.6cm}
                                                   
                                                      \newBoxDiagram
                                                      [box height =
                                                      0cm, output
                                                      height = 0cm,
                                                      arm height =
                                                      1.5cm, width =
                                                      \leftw,
                                                      decoration
                                                      yshift label =
                                                      .2cm]{left};
                                                      \newBoxDiagram
                                                      [arm height =
                                                      .5cm, width =
                                                      \w, coord =
                                                      {($(\sepw +
                                                        \leftw + \w,
                                                        0)$)},
                                                      decoration
                                                      yshift label =
                                                      .2cm, label =
                                                      \msu^2]{f};
                                                      \newBoxDiagram
                                                      [box height =
                                                      0cm, output
                                                      height = 0cm,
                                                      arm height =
                                                      1.5cm, width =
                                                      \rightw, coord =
                                                      {(\leftw +
                                                        2*\sepw+2*\w +
                                                        \rightw, 0)},
                                                      decoration
                                                      yshift label =
                                                      .2cm]{right};
                                                      \print{left};
                                                      \printArmSep{left};
                                                      \printDec[/decorate]{left}{j};
                                                      \printNoArms{f};
                                                      \printArm[/unit]{f}{-1};
                                                      \printArm{f}{1};
                                                      \print{right};
                                                      \printArmSep{right};
                                                             \end{tikzpicture}
  \hspace{.2cm} + \hspace{.2cm} 
    \begin{tikzpicture}[baseline={([yshift = -.35cm]current bounding
  box.center)}]
                                                      \def\sepw{.2cm}
                                                      \def\leftw{.4cm}
                                                      \def\w{.5cm}
                                                      \def\rightw{.6cm}
                                                   
                                                      \newBoxDiagram
                                                      [box height =
                                                      0cm, output
                                                      height = 0cm,
                                                      arm height =
                                                      1.5cm, width =
                                                      \leftw,
                                                      decoration
                                                      yshift label =
                                                      .2cm]{left};
                                                      \newBoxDiagram
                                                      [arm height =
                                                      .5cm, width =
                                                      \w, coord =
                                                      {($(\sepw +
                                                        \leftw + \w,
                                                        0)$)},
                                                      decoration
                                                      yshift label =
                                                      .2cm, label =
                                                      \msu^2]{f};
                                                      \newBoxDiagram
                                                      [box height =
                                                      0cm, output
                                                      height = 0cm,
                                                      arm height =
                                                      1.5cm, width =
                                                      \rightw, coord =
                                                      {(\leftw +
                                                        2*\sepw+2*\w +
                                                        \rightw, 0)},
                                                      decoration
                                                      yshift label =
                                                      .2cm]{right};
                                                      \print{left};
                                                      \printArmSep{left};
                                                      \printDec[/decorate]{left}{j-1};
                                                      \printNoArms{f};
                                                      \printArm{f}{-1};
                                                      \printArm[/unit]{f}{1};
                                                      \print{right};
                                                      \printArmSep{right};
                                                             \end{tikzpicture}
                                                             \hspace{.1cm}
                                                             = 0.
                                                           \end{displaymath}
(To check this one can evaluate both diagrams on the element
$[a_{1}|\ldots|a_{j}|1|a_{j+1}|\ldots|a_{n}]$, using the sign
conventions of \ref{rem:signs_for_diags}.) The above implies that
\begin{displaymath}
\o \mu \circ \msu = \begin{tikzpicture}[baseline={(current bounding box.center)}]
      \def\w{1.3cm} \def\gstart{-.5} \def\relw{.5} \def\sepw{.1}
      \newBoxDiagram[width/.expand once = \w, arm height = 1.5cm,
      label = \o \mu]{f}; \printNoArms{f}; \printArm{f}{1}; \f.coord at top(gcoor,
      \gstart); \newBoxDiagram[width = \w*\relw, arm height = .5cm,
      coord = {(gcoor)}, output height = .5cm, label = \msu^{2}]{g};
      \printNoArms{g}; \printArm[/unit]{g}{-1}; \printArm{g}{1};

      \def\smallsep{.08}

      \pgfmathsetmacro{\var}{\gstart + \relw + \sepw}
      \printArm{f}{\var}
      \pgfmathsetmacro{\var}{1-\smallsep}
      \pgfmathsetmacro{\vara}{\gstart + \relw + \sepw + \smallsep}
      \printArmSep[arm sep left = \vara, arm sep right = \var]{f}
    \end{tikzpicture} +
    \begin{tikzpicture}[baseline={(current bounding box.center)}]
      \def\w{1.3cm} \def\gstart{.5} \def\relw{.5} \def\sepw{.1}
      \newBoxDiagram[width/.expand once = \w, arm height = 1.5cm,
      label = \o \mu]{f}; \printNoArms{f}; \printArm{f}{-1}; \f.coord at top(gcoor,
      \gstart); \newBoxDiagram[width = \w*\relw, arm height = .5cm,
      coord = {(gcoor)}, output height = .5cm, label = \msu^{2}]{g};
      \printNoArms{g}; \printArm{g}{-1}; \printArm[/unit]{g}{1}; \printArmSep{g};

      \def\smallsep{.08}
      \pgfmathsetmacro{\var}{\gstart - \relw - \sepw}
      \printArm{f}{\var};
      \pgfmathsetmacro{\var}{-1+\smallsep}
      \pgfmathsetmacro{\vara}{\gstart - \relw - \sepw - \smallsep}
      \printArmSep[arm sep left = \var, arm sep right = \vara]{f}

    \end{tikzpicture}
  \end{displaymath}
  and one checks this is $-\msu \circ \o \mu$ (by evaluating on
  elements as above). The proof of $(2)$ is similar.
\end{proof}

\begin{proof}[Proof of Theorem \ref{thm:main-thm-sunal}.] 
  For $\o \mu + h, \o \mu' + h' \in \hhc {\o A} A$, we have $[\o
  \mu + h, \o \mu' + h'] = [\o \mu, \o \mu'] + [\o \mu', h] + [\o \mu,
  h'] \in \hhc {\o A} {A}$, using the previous lemma. Thus $\hhc {\o
    A} A$ is a graded subalgebra of $\hhc A A$. Again using the lemma,
  we have $[\msu,
  \o \mu + h] = [\msu, h] \in \hhc {\o A} {\o A}$, and thus $\hhc {\o
    A} A$ is preserved by $[\msu, -]$.

  A strictly unital element $\o \mu + h + \msu$ in
  $\hhc A A_{-1}$ is an \Ai-algebra structure exactly when $[\o \mu + h + \msu, \o \mu + h + \msu] = [\o \mu + h, \o \mu + h]
    + 2[\msu, \o \mu + h]$ is zero, i.e., $\frac{1}{2}[\o \mu + h, \o
    \mu + h] + [\msu, \o \mu + h] = 0.$ And this is the definition of
    $\o \mu + h$ being a Maurer-Cartan element of $(\hhc {\o A} A,
    [\msu, -]).$
\end{proof}

\begin{rems}
  The dg-Lie algebra $(\hhc {\o A} A, [\msu, -])$ is an adaptation 
 of a construction of Schlessinger and Stasheff \cite[\S
  2]{SchlStash}, who use the cofree Lie coalgebra where we use the
  cofree coassociative 
  coalgebra $\Tco {\Pi \o A}$.  To match the definitions, one can
  check that the graded subalgebra
  $\hhc {\o A} {\o A}$ of $\hhc A A$ acts on the $k$-module
  $\hhc {\o A} k,$ via Lemma \ref{lem:diagrams-for-sual}.(4), and the
  resulting semi-direct product $\hhc {\o A} {\o A} \oplus \hhc {\o A} k$ is
  isomorphic as a graded Lie algebra to $\hhc {\o A} A$; one then checks the derivations agree.
\end{rems}

\subsection{Curved bar construction}

\begin{defn}\label{defn:ad-and-curved-coalg}
  Let $C$ be a graded coalgebra and $\xi \in \Hom {} C k$ a
  homogeneous linear map. Define $\ad{\xi} \in \Hom {} C C_{|\xi|}$ to
  be
  \begin{displaymath}
    \ad{\xi} := \begin{aligned}
      \left( C \xra{\Delta} C \otimes C \xra{\xi \otimes 1} k \otimes C
        \cong C \right)
      -\left( C \xra{\Delta} C \otimes C \xra{1 \otimes \xi} C \otimes k
        \cong C \right).
    \end{aligned}
  \end{displaymath}
  (One checks this is a coderivation of $C$.) A \emph{curved
    dg-coalgebra} is a triple $(C,d, \xi)$, with $C$ a graded
  coalgebra, $d: C \to C$ a coderivation of degree $-1$, and
  $\xi: C \to k$ a degree $-2$ linear map, such that
  $d^{2} = \ad{\xi}$ and $\xi d = 0.$ A dg-coalgebra is a curved
  dg-coalgebra with $h = 0$ (so $d^{2} = 0$).
\end{defn}

When $C = \Tco {\Pi A},$ we can calculate $\ad{\xi}$ using the
Gerstenhaber bracket and the trivial strictly unital \Ai-structure.

\begin{lem}\label{lem:ad-for-tensor-coalgebra}
  If $\xi \in \Hom {} {\Tco {\Pi A}} k$, then
  $\ad{\xi} = \Phi^{-1}([\msu, s \xi])$.
\end{lem}

\begin{proof}
  Since $\ad{\xi}$ is a coderivation, it is equal to
  $\Phi^{-1} (\pi_{1} \ad{\xi}),$ using Lemma
  \ref{lem:isom_hhc_coder}.(1) (where $\Phi^{-1}$ is defined,
  also). Thus it is enough to show that
  $\pi_{1} \ad{\xi}|_{\Pi \overbar A^{\otimes n}} = [\msu, s
  \xi]|_{\Pi \overbar A^{\otimes n}}$ for all $n \geq 1.$ If
  $\xi = (\xi^{n}),$ then
  \begin{displaymath}
    \pi_{1} \ad{\xi}|_{\Pi \overbar A^{\otimes n}} = (\Pi \o A^{\otimes n}
    \xra{\xi^{n-1} \otimes 1} k \otimes \Pi \o A \cong \Pi \o A) - (\Pi \o A^{\otimes n}
    \xra{1 \otimes \xi^{n-1}} \Pi \o A \otimes k \cong \Pi \o A).
  \end{displaymath}
  By Lemma \ref{lem:diagrams-for-sual}.(2) we have
  $[\msu, s \xi]|_{\Pi \overbar A^{\otimes n}} = \msu^{2}(s \xi^{n-1}
  \otimes 1 + 1 \otimes s \xi^{n-1})$. Using the definition of $\msu$
  in \ref{def-msu} and the isomorphisms \ref{sect:notation}.(4), one
  checks these agree.
\end{proof}

\begin{cor}\label{cor:sual-ai-alg-equiv-to-curved-bar}
  Let $(A, 1)$ be a graded module with split element. A strictly unital element
  $\o \mu + h + \msu$ in $\hhc A A_{-1}$ is an \Ai-algebra structure
  if and only if the triple
  $(\Tco {\Pi \o A}, \Phi^{-1}(\o \mu), -s^{-1} h)$ is a curved
  dg-coalgebra ($\Phi^{-1}$ is defined in
  \ref{lem:isom_hhc_coder}). In diagrams, this is equivalent
  to:
  \begin{equation}
    \begin{gathered}
        \sum_{j} \begin{tikzpicture}[baseline={(current bounding
          box.center)}] \def\w{1cm} \def\gstart{.1} \def\relw{.4}
        \def\sepw{.1} \def\sepwidth{.1cm} \def\leftw{.3cm}
        \def\rightw{.3cm} \def\hgt{.5cm} \newBoxDiagram [box height =
        \hgt, output height = \hgt, width = \w, coord =
        {($(\sepwidth + \leftw + \w, 0)$)}, decoration yshift label =
        .2cm, decoration yshift label = .05cm, arm height = 3*\hgt,
        label = \o \mu]{f}; \print{f}; \f.coord at top(gcoor,
        \gstart); \newBoxDiagram[width = \w*\relw, box height = \hgt,
        arm height = \hgt, output height = \hgt, coord = {(gcoor)},
        output height = \hgt, label = \o \mu]{g}; \print{g};
        \printArmSep{g}; \def\smallsep{.08}
        \pgfmathsetmacro{\var}{\gstart - \relw - \sepw}
        \printArm{f}{\var}; \printDec[/decorate, decorate left = -1,
        decorate right = \var]{f}{j};
        \pgfmathsetmacro{\var}{-1+\smallsep}
        \pgfmathsetmacro{\vara}{\gstart - \relw - \sepw - \smallsep}
        \printArmSep[arm sep left = \var, arm sep right = \vara]{f}
        \pgfmathsetmacro{\var}{\gstart + \relw + \sepw}
        \printArm{f}{\var} \pgfmathsetmacro{\var}{1-\smallsep}
        \pgfmathsetmacro{\vara}{\gstart + \relw + \sepw + \smallsep}
        \printArmSep[arm sep left = \vara, arm sep right = \var]{f}
      \end{tikzpicture}       \hspace{.1cm} + \hspace{.1cm}
      \begin{tikzpicture}[baseline={(current bounding box.center)}]
        \def\w{1cm} \def\gstart{-.35} \def\relw{.6} \def\sepw{.1}
        \newBoxDiagram[width/.expand once = \w, arm height = 1.5cm,
        label = \msu^{2}]{f}; \printNoArms{f}; 
        \printDec[color = white]{f}{test}

\f.coord at
        top(gcoor, \gstart); \newBoxDiagram[width = \w*\relw, arm
        height = .5cm, coord = {(gcoor)}, output height = .5cm, label
        = h]{g}; \print[/unit output]{g}; \printArmSep{g};
        \printArm{f}{.8} \def\smallsep{.08}
        \pgfmathsetmacro{\var}{\gstart - \relw - \sepw}
        \pgfmathsetmacro{\var}{-1+\smallsep}
        \pgfmathsetmacro{\vara}{\gstart - \relw - \sepw - \smallsep}

        \pgfmathsetmacro{\var}{\gstart + \relw + \sepw}
      %
        \pgfmathsetmacro{\var}{1-\smallsep}
        \pgfmathsetmacro{\vara}{\gstart + \relw + \sepw + \smallsep}
      \end{tikzpicture}
 \hspace{.1cm} + \hspace{.1cm}
      \begin{tikzpicture}[baseline={(current bounding box.center)}]
        \def\w{1cm} \def\gstart{.35} \def\relw{.6} \def\sepw{.1}
        \newBoxDiagram[width/.expand once = \w, arm height = 1.5cm,
        label = \msu^{2}]{f};
        \printDec[color = white]{f}{test}
 \printNoArms{f}; \f.coord at
        top(gcoor, \gstart); \newBoxDiagram[width = \w*\relw, arm
        height = .5cm, coord = {(gcoor)}, output height = .5cm, label
        = h]{g}; \print[/unit output]{g}; \printArmSep{g};
        \printArm{f}{-.8} \def\smallsep{.08}
        \pgfmathsetmacro{\var}{\gstart - \relw - \sepw}
        \pgfmathsetmacro{\var}{-1+\smallsep}
        \pgfmathsetmacro{\vara}{\gstart - \relw - \sepw - \smallsep}

        \pgfmathsetmacro{\var}{\gstart + \relw + \sepw}
      %
        \pgfmathsetmacro{\var}{1-\smallsep}
        \pgfmathsetmacro{\vara}{\gstart + \relw + \sepw + \smallsep}
      \end{tikzpicture}
\quad = \quad 0\\
      \sum_{j} \begin{tikzpicture}[baseline={(current bounding
          box.center)}] \def\w{1.3cm} \def\gstart{-.1} \def\relw{.5}
        \def\sepw{.1} \newBoxDiagram[width/.expand once = \w, arm
        height = 1.5cm, label = h]{f}; \print[/unit output]{f};
        \f.coord at top(gcoor, \gstart); \newBoxDiagram[width =
        \w*\relw, arm height = .5cm, coord = {(gcoor)}, output height
        = .5cm, label = \o \mu]{g}; \print{g}; \printArmSep{g};

        \def\smallsep{.08}
        \pgfmathsetmacro{\var}{\gstart - \relw - \sepw}
        \printArm{f}{\var}; \printDec[/decorate, decorate left = -1,
        decorate right = \var]{f}{j};
        \pgfmathsetmacro{\var}{-1+\smallsep}
        \pgfmathsetmacro{\vara}{\gstart - \relw - \sepw - \smallsep}
        \printArmSep[arm sep left = \var, arm sep right = \vara]{f}

        \pgfmathsetmacro{\var}{\gstart + \relw + \sepw}
        \printArm{f}{\var}
        \pgfmathsetmacro{\var}{1-\smallsep}
        \pgfmathsetmacro{\vara}{\gstart + \relw + \sepw + \smallsep}
        \printArmSep[arm sep left = \vara, arm sep right = \var]{f}
      \end{tikzpicture} \quad = \quad 0.
    \end{gathered}\label{eqn:diagrams-for-sual-ainf-alg}
  \end{equation}
\end{cor}

\begin{proof}
  Let $\o \mu + h + \msu$ be a strictly unital element with
  $\o \mu + h \in \hhc {\o A} {\o A} \oplus \hhc {\o A} k = \hhc {\o
    A} A$. By Theorem \ref{thm:main-thm-sunal}, this is an \Ai-algebra
  structure if and only if it is a Maurer-Cartan element of
  $(\hhc {\o A} {A}, [\msu, -]).$ By Lemma
  \ref{lem:diagrams-for-sual}, this is equivalent to
$$[\msu, h] + \frac{1}{2}[\o \mu, \o
\mu] = 0 \quad \text{ and } \quad [\o \mu, h] = 0,$$ and these are
equivalent to the first and second equations of
\eqref{eqn:diagrams-for-sual-ainf-alg}, respectively.

Set $\o d = \Phi^{-1}(\o \mu)$. The triple
$(\Tco {\Pi \o A}, \o d, -s^{-1} h)$ is a curved dg-coalgebra if and
only if $\o d^2 + \ad{s^{-1}h} = 0$ and $s^{-1} h \o d = 0.$ We have
$h \o d = h \Phi^{-1}(\o \mu) = h \circ \o \mu,$ so $s^{-1} h \o d$ is zero
exactly when the second equation of
\eqref{eqn:diagrams-for-sual-ainf-alg} holds. Since $\o d^{2}$
and $\ad{s^{-1} h}$ are both coderivations of $\Tco {\Pi \o A}$, $\o
d^2 = -\ad{s^{-1}h}$ holds if and only if
$\pi_{1}\o d^{2} = -\pi_{1} \ad{s^{-1} h}$ holds, by
\ref{lem:isom_hhc_coder}.(1). We have
$\pi_{1}\o d^{2} = \o \mu \circ \o \mu$, and
$\pi_{1} \ad{s^{-1} h} = [\msu, h]$ by Lemma
\ref{lem:ad-for-tensor-coalgebra}. Thus $\o d^{2} + \ad{s^{-1}h} = 0$ holds if
and only if $\o \mu \circ \o \mu + [\msu, h] = 0$ holds, and this is exactly
the first equation of
\eqref{eqn:diagrams-for-sual-ainf-alg}.
\end{proof}

\begin{defn}
  If $(A, 1, \o \mu + h + \msu)$ is an \Ai-algebra with split unit,
  the \emph{curved bar construction},
  denoted $\Bar {\o A}$, is the curved dg-coalgebra
  $(\Tco {\Pi \o A}, \Phi^{-1}(\o \mu), -s^{-1} h)$.
\end{defn}

\begin{rems}
  Note that $\Bar {\o A}$ is a dg-coalgebra if and only if
  $h = 0$ if and only if $(A, 1, \mu)$ is augmented.
\end{rems}

\begin{ex}
 Let $(A,1)$ be a graded module with split element as in Example
 \ref{ex:reslns-aug}, and let $\nu = \o \mu + h + \msu$ be a strictly
 unital \Ai-algebra structure on $(A,1)$. Set $\o d = \Phi^{-1}(\o
 \mu) \in \Coder(\Tco{\Pi \o A}, \Tco{\Pi \o A}).$ By Example
 \ref{ex:reslns-aug}, $h^n = 0$ for $n \geq 2$, thus $h \circ \o \mu =
 h^1 \o \mu^1 = 0$ and $\msu^2 \circ h$ is concentrated in tensor
 degree two. It follows from Corollary
 \ref{cor:sual-ai-alg-equiv-to-curved-bar} that $\o
 d^2[a_1|\ldots|a_n] = 0$ for $n \neq 2$ and
\begin{displaymath}
  \o d^{2}[a_{1}|a_{2}] =
    \begin{cases}
      \phantom{-}0 & |a_{1}| \neq 1 \text{ and } |a_{2}| \neq 1\\
      \phantom{-}m^{1}(a_{1})a_{2} & |a_{1}| = 1 \text{ and } |a_{2}| \neq 1\\
 -m^{1}(a_{2})a_{1} & |a_{1}| \neq 1 \text{ and } |a_{2}| = 1\\
 \phantom{-}m^{1}(a_{1})a_{2} - m^{1}(a_{2})a_{1} & |a_{1}| = 1 \text{ and } |a_{2}| = 1.
\end{cases}
\end{displaymath}
The \Ai-algebra $(A, 1, \nu)$ is augmented exactly when $h^{1} = 0,$
which is equivalent to $(m^{1})_{1} = 0.$ Thus we see directly in this
case that $\o d^{2} = 0$ if and only if $(A, 1, \nu)$ is augmented.
 \end{ex}

The smallest nontrivial case of the above is the following.

\begin{ex}\label{ex:curved-bar-constr-of-koszul-cx}
  Let $(A, 1, \mu)$ be the Koszul complex on $f \in k,$  so $\mu^{n} = 0$ for $n \geq 3$, $\mu^{2} = \msu$ and
  $\mu^{1} = (k \cdot [e] \xra{f} k \cdot [1]) \in \hhc {\o A} k_{-1}$.
  Thus $\o \mu = 0$ and $h = \mu^{1}$, so $A$ is augmented if
  and only if $f = 0.$ We also have:
\begin{displaymath} \Tco {\Pi \o A}\hspace{.2cm} = \hspace{.2cm}
  \begin{tikzpicture}[font=\fontsize{7}{22.4}\selectfont]
    \node[matrix] (my matrix) at (0,0) {
      \node {\ldots}; & \node{0};  & \node{$k[e]^{\otimes n}$};  &
      \node{0};  & \node{\ldots};  & \node{0};  & \node{$k[e|e]$};  &
      \node{0};  & \node{$k[e]$};  & \node{0}; 0  & \node{0};\\
      \node {}; & \node{$2n+1$};  & \node{$2n$};  & \node{$2n-1$};  &
      \node{};  & \node{5};  & \node{4};  & \node{3};  & \node{2};  &
      \node{1};  & \node{0};  & \node{};\\};
  \end{tikzpicture}
\end{displaymath}
and $h_{1}([e]) = -f$ and $h_{n} = 0$ for $n \geq 2$. If we set $T =
[e] \in \Tco {\Pi \o A}_{2},$ then $\Tco {\o A} = k[T],$ the divided powers coalgebra on the
1-dimensional free module generated by $T$. The $k$-dual is the symmetric
algebra $k[T^{*}]$, with curvature $-f T^{*}
\in k[T^{*}]_{-2}$.
\end{ex}

We now show the curved bar construction is functorial.

\begin{defn}
  A \emph{morphism of curved dg-coalgebras, $(C, d_{C}, h_{C}) \to
  (D, d_{D}, h_{D}),$} is a pair $(\gamma, \alpha),$ with
  $\gamma: C \to D$ a graded coalgebra morphism and
  $\alpha : {C} \to k$ a degree $-1$ linear map, such that the
  following equations hold,
  \begin{align*}
    d_D \gamma = \gamma d_{C} + \gamma \ad{\alpha} &\in \Hom {} C D\\
    h_D \gamma - \alpha^2 = \alpha d_{C} + h_{C} &\in \Hom {} C k,
  \end{align*}
  where $\ad \alpha$ is defined in \ref{defn:ad-and-curved-coalg}, and
  $\alpha^2 = (C \xra{\Delta} C \otimes C \xra{\alpha \otimes \alpha}
  k \otimes k \cong k).$
\end{defn}

\begin{cor}\label{cor-strictly-unital-morphisms-curved-morphisms}
  Let $(A, 1, \o \mu_{A}+ h_{A} + \msu)$ and $(B, 1, \o \mu_B + h_{B}
  + \msu)$ be
  \Ai-algebras with split units. A strictly unital element
  $g + \gsu \in \hhc A B_0$, with
$$g = \o g + a \in \hhc {\o A} {\o B} \oplus \hhc {\o A} k = \hhc {\o
  A} B,$$ is a morphism of \Ai-algebras if and only if
$$(\Psi^{-1}(\o g), -s^{-1}a): \Bar {\o A} \to \Bar {\o B}$$ is a
morphism of the corresponding curved dg-coalgebras, where
$\Psi^{-1}$ is defined in \ref{lem:isom_hhc_coder}.(2). In diagrams this
is equivalent to:
\begin{equation}
  \label{eq:cor-for-strictly-unital-morphism}
\def\wida{.33cm} 
    \def\widb{.30cm}
 \def\widk{.33cm}
 \def\dotsw{.32cm}
 \def\sepw{.1cm}
\def\hgt{.52cm}
\def\wdth{.55cm}
\begin{gathered}
\begin{tikzpicture}
        \newBoxDiagram[width = \wida, box height = \hgt, arm height
        = \hgt, coord = {(0,0)}, output height = \hgt, label =
        {\o g}]{g1};
        \print[/bmod]{g1};
        \printArmSep{g1};
        \printDec[color = white]{g1}{test}
        \node at
        ($(\wida + \sepw + \dotsw, \hgt +
        \hgt/2)$) {\ldots};
 \newBoxDiagram[width = \widk, box height
        = \hgt, arm height = \hgt, output height = \hgt, coord =
        {($(\wida+\sepw+2*\dotsw+\widk,
          0)$)}, label = {\o g}]{gk};
 \print[/bmod]{gk}; 
\printArmSep{gk};
        \setlength{\var}{\wida/2+\sepw/2+\dotsw+\widk/2}
        \setlength{\vara}{-\hgt - \hgt}
 \newBoxDiagram[width = \var,
        box height = \hgt, arm height = \hgt, output height = \hgt,
        coord = {(\var, \vara)}, label = {\o \mu_{B}}]{mut}
        \printNoArms[/bmod]{mut}
      \end{tikzpicture}
\hspace{.2cm} - \hspace{.2cm}
 \sum_{j}\hspace{.2cm}
\begin{tikzpicture}
      \def\w{1cm} \def\gstart{-.1} \def\relw{.4} \def\sepw{.1}
      \newBoxDiagram[width/.expand once = \w, arm height = 3*\hgt, box
      height = \hgt, coord = {(0,0)},
      label = {\o g}]{f}; \print[/bmod]{f};

 \f.coord at top(gcoor, \gstart);
      \newBoxDiagram[width = \w*\relw, arm height = \hgt, coord =
      {(gcoor)}, output height = \hgt, label = {\o \mu_{A}}]{g};
      \print{g}; \printArmSep{g};
      \def\smallsep{.08}
      \pgfmathsetmacro{\var}{\gstart - \relw - \sepw}
      \printArm{f}{\var}; \printDec[/decorate, decorate left = -1,
      decorate right = \var]{f}{j};
      \pgfmathsetmacro{\var}{-1+\smallsep}
      \pgfmathsetmacro{\vara}{\gstart - \relw - \sepw - \smallsep}
      \printArmSep[arm sep left = \var, arm sep right = \vara]{f}
      \pgfmathsetmacro{\var}{\gstart + \relw + \sepw}
      \printArm{f}{\var} \pgfmathsetmacro{\var}{1-\smallsep}
      \pgfmathsetmacro{\vara}{\gstart + \relw + \sepw + \smallsep}
      \printArmSep[arm sep left = \vara, arm sep right = \var]{f}
    \end{tikzpicture}
\hspace{.2cm} +  \hspace{.2cm}
\begin{tikzpicture}[baseline={([yshift = .2cm]current bounding
  box.center)}]

\setlength{\lengg}{0pt}

      \newBoxDiagram[width = \wdth, box height = \hgt, arm height =
      \hgt, output height = \hgt, label =
      \msu^2, coord = {(0,
       0)}]{isom} \printNoArms[/bmod]{isom}
      \printArm[/unit]{isom}{1};

\newlength{\whighbox}
\setlength{\whighbox}{\wdth - \sepw - \sepw}
\isom.coord at top(hghcoor, -1)
\newBoxDiagram[width/.expand once = \whighbox, box height = \hgt, arm height =
      \hgt, output height = \hgt, label = {\o g}, coord = {(hghcoor)}]{gmap}
       \printNoArms[/bmod]{gmap}
\printArmSep{gmap};
\printArm{gmap}{-1};
\printArm{gmap}{1};

\setlength{\whighbox}{\wdth - \sepw - \sepw}
\isom.coord at top(hghcoor, 1)
\newBoxDiagram[width/.expand once = \whighbox, box height = \hgt, arm height =
      \hgt, output height = \hgt, label = {a}, coord = {(hghcoor)}]{amap}
       \printBox{amap}
\printArmSep{amap};
\printArm{amap}{-1};
\printArm{amap}{1};
    \end{tikzpicture} 
\hspace{.2cm} + \hspace{.2cm}
\begin{tikzpicture}[baseline={([yshift = .2cm]current bounding
  box.center)}]
\setlength{\lengg}{\hgt + \hgt}
      \newBoxDiagram[width = \wdth, box height = \hgt, arm height =
      \hgt, output height = \hgt, label =
      \msu^2, coord = {(0,
        0)}]{isom} \printNoArms[/bmod]{isom}
      \printArm[/unit]{isom}{-1};

\setlength{\whighbox}{\wdth - \sepw - \sepw}
\isom.coord at top(hghcoor, -1)
\newBoxDiagram[width/.expand once = \whighbox, box height = \hgt, arm height =
      \hgt, output height = \hgt, label = a, coord = {(hghcoor)}]{amap}
       \printBox{amap}
\printArmSep{amap};
\printArm{amap}{-1};
\printArm{amap}{1};

\setlength{\whighbox}{\wdth - \sepw - \sepw}
\isom.coord at top(hghcoor, 1)
\newBoxDiagram[width/.expand once = \whighbox, box height = \hgt, arm height =
      \hgt, output height = \hgt, label = {\o g}, coord = {(hghcoor)}]{gmap}
       \printNoArms[/bmod]{gmap};
\printArmSep{gmap};
\printArm{gmap}{-1};
\printArm{gmap}{1};
    \end{tikzpicture}  \hspace{.2cm} = \hspace{.2cm} 0\\
    %
    %
    %
    %
   \begin{tikzpicture}
        \newBoxDiagram[width = \wida, box height = \hgt, arm height
        = \hgt, coord = {(0,0)}, output height = \hgt, label =
        {\o g}]{g1};
        \print[/bmod]{g1};
        \printArmSep{g1};
        \printDec[color = white]{g1}{test}

        \node at
        ($(\wida + \sepw + \dotsw, \hgt +
        \hgt/2)$) {\ldots};
 \newBoxDiagram[width = \widk, box height
        = \hgt, arm height = \hgt, output height = \hgt, coord =
        {($(\wida+\sepw+2*\dotsw+\widk,
          0)$)}, label = {\o g}]{gk};
 \print[/bmod]{gk}; 
\printArmSep{gk};
        \setlength{\var}{\wida/2+\sepw/2+\dotsw+\widk/2}
        \setlength{\vara}{-\hgt - \hgt}
 \newBoxDiagram[width = \var,
        box height = \hgt, arm height = \hgt, output height = \hgt,
        coord = {(\var, \vara)}, label = {h_{B}}]{mut}
        \printNoArms[/unit output]{mut}
      \end{tikzpicture}
\hspace{.2cm} - \hspace{.2cm}
 \sum_{j}\hspace{.2cm}
\begin{tikzpicture}
      \def\w{1cm} \def\gstart{-.1} \def\relw{.4} \def\sepw{.1}
      \newBoxDiagram[width/.expand once = \w, arm height = 3*\hgt, box
      height = \hgt, coord = {(0,0)},
      label = {a}]{f}; \print[/unit output]{f};

 \f.coord at top(gcoor, \gstart);
      \newBoxDiagram[width = \w*\relw, arm height = \hgt, coord =
      {(gcoor)}, output height = \hgt, label = {\o \mu_{A}}]{g};
      \print{g}; \printArmSep{g};
      \def\smallsep{.08}
      \pgfmathsetmacro{\var}{\gstart - \relw - \sepw}
      \printArm{f}{\var}; \printDec[/decorate, decorate left = -1,
      decorate right = \var]{f}{j};
      \pgfmathsetmacro{\var}{-1+\smallsep}
      \pgfmathsetmacro{\vara}{\gstart - \relw - \sepw - \smallsep}
      \printArmSep[arm sep left = \var, arm sep right = \vara]{f}
      \pgfmathsetmacro{\var}{\gstart + \relw + \sepw}
      \printArm{f}{\var} \pgfmathsetmacro{\var}{1-\smallsep}
      \pgfmathsetmacro{\vara}{\gstart + \relw + \sepw + \smallsep}
      \printArmSep[arm sep left = \vara, arm sep right = \var]{f}
    \end{tikzpicture}
\hspace{.2cm} -  \hspace{.2cm}
\begin{tikzpicture}[baseline={([yshift = .2cm]current bounding
  box.center)}]
\setlength{\lengg}{\hgt + \hgt}
      \newBoxDiagram[width = \wdth, box height = \hgt, arm height =
      2*\hgt, output height = 2*\hgt, label =
      h_{A}, coord = {($(0,
        \lengg)$)}]{hmap} 
      \print[/unit output]{hmap}
      \printArmSep{hmap}
\end{tikzpicture}
\hspace{.2cm} + \hspace{.2cm}
\begin{tikzpicture}[baseline={([yshift = .2cm]current bounding
  box.center)}]

\setlength{\lengg}{\hgt + \hgt}
      \newBoxDiagram[width = \wdth, box height = \hgt, arm height =
      \hgt, output height = \hgt, label =
      \msu^2, coord = {($(0,
        \lengg)$)}]{isom} \printNoArms[/unit output]{isom} \printArm[/unit]{isom}{1}
      \printArm[/unit]{isom}{-1};

\setlength{\whighbox}{\wdth - \sepw - \sepw}
\isom.coord at top(hghcoor, -1)
\newBoxDiagram[width/.expand once = \whighbox, box height = \hgt, arm height =
      \hgt, output height = \hgt, label = a, coord = {(hghcoor)}]{amap}
       \printBox{amap}
\printArmSep{amap};
\printArm{amap}{-1};
\printArm{amap}{1};

\setlength{\whighbox}{\wdth - \sepw - \sepw}
\isom.coord at top(hghcoor, 1)
\newBoxDiagram[width/.expand once = \whighbox, box height = \hgt, arm height =
      \hgt, output height = \hgt, label = a, coord = {(hghcoor)}]{gmap}
       \printBox{gmap}
\printArmSep{gmap};
\printArm{gmap}{-1};
\printArm{gmap}{1};
    \end{tikzpicture} \hspace{.2cm} = \hspace{.2cm} 0.
    \end{gathered}
\end{equation}
\end{cor}

We need the following lemma for the proof. For later use, we assume that $B$
has a strict, but not necessarily split, unit; e.g., $B = k/I$ for
some ideal $I$.
\begin{lem}\label{lem:strictly-unital-morphisms}
  Let $(A, 1, \o \mu_{A} + h_{A} + \msu)$ be an \Ai-algebra with split unit
  and $(B, \nu_B)$ an \Ai-algebra with strict unit $1 \in B_{0}$. A
  strictly unital element $g + \gsu$, with $g \in \hhc {\o A} B_{0}$
  is an \Ai-morphism if and only if $\nu_B * g$ and $g \circ \o \mu_A
  + \gsu \circ h_{A}$ are equal,
where $*$ is defined in \ref{rem:defining-start-and-gerst-produc} and
$\circ$ is the Gerstenhaber product.
\end{lem}

\begin{proof}
  By definition,
  $g + \gsu$ is an \Ai-morphism exactly when
  \begin{equation}\label{eqn-ainf-morphism}
    \nu_B * (g + \gsu) = (g + \gsu) \circ (\o \mu_{A} + h_{A} + \msu).
  \end{equation}
  We claim this equation always holds for elements of $\Tco{\Pi A}
  \setminus \Tco{\Pi \o A}.$ Assuming the claim, we now note that the
  above equation holds on elements of $\Tco {\Pi \o A}$ if and only if
  $\nu_{B}* g = (g + \gsu) \circ (\o \mu_{A} + h_{A}),$ since $\gsu$ and $\msu$ are
  zero on $\Tco {\Pi \o A}.$ Also, clearly $g \circ h_{A} = 0$ and
  $\gsu \circ \o \mu_{A} = 0.$ Thus \eqref{eqn-ainf-morphism} holds on
  $\Tco {\Pi \o A}$ if and only if $\nu_B * g = g \circ \o \mu_A
  + \gsu \circ h_{A}.$

We are left to prove the claim, i.e., that \eqref{eqn-ainf-morphism}
holds on any element of the form
$\bs a = [a_1|\ldots|a_{l-1}|1|a_{l+1}|\ldots|a_{n}].$  We first compute the
left side. Since $g + \gsu$
is strictly unital, we have
\begin{displaymath}
\nu_B * (g + \gsu)(\bs a) =
      \nu_{B}\left (\Psi^{-1}(g+ \gsu)[a_1|\ldots|a_{l-1}] \otimes [1]
      \otimes \Psi^{-1}(g+\gsu)[a_{l+1}|\ldots|a_{n}]\right ).
  \end{displaymath}
Using that $\nu_{B}$ is strictly unital, we have: if $l = 1 = n$, the result is $0$; if $l = 1 < 2 = n$ and
 $a_{2} = 1,$ or if $l = 2 = n$ and $a_{1} = 1,$ the
 result is $[1];$ if $l = 1 < n,$ the result is
 $g^{n-1}[a_2|\ldots|a_{n}];$ if $1 < l = n,$ the result is
  $(-1)^{|a_{1}| + \ldots + |a_{n-1}| + n -
    1}g^{n-1}[a_1|\ldots|a_{n-1}];$ all other cases are zero. One now
  checks the same six cases on the right side.
\end{proof}

\begin{proof}[Proof of \ref{cor-strictly-unital-morphisms-curved-morphisms}]
  By the previous lemma, $g + \gsu$ is an \Ai-morphism if and only if
  $(\o \mu_B +  h_{B} + \msu) * g = g \circ \o \mu_A + \gsu \circ h_{A}.$
  Substituting $g = \o g + a$, and using the equalities $\o \mu_{B}*g =
  \o \mu_{B} * \o g$ and $h_{B}*g = h_{B} * \o g,$ this is
  equivalent to:
  \begin{equation}\label{eq:in-proof-for-ainf-morphism}
    \o \mu_B * \o g + h_B * \o g + \msu*(\o g + a) - \o g \circ \o \mu_{A} - a
    \circ \o \mu_{A} - \gsu \circ h_{A} = 0.
  \end{equation}
  We can match each term of \eqref{eq:in-proof-for-ainf-morphism}
  with a diagram in \eqref{eq:cor-for-strictly-unital-morphism}: $\o
  \mu_{B} * \o g$ is the first diagram and $- \o g
  \circ \o \mu_{A}$ is the second diagram, both in the first line; $h_B * \o g, -a \circ \o \mu_{A}$ and $-\gsu \circ h_{A}$
  are the first, second, and third diagrams of the
  second line; $\msu*(\o g +
  a)$ is the sum of the third and fourth diagrams of the first line,
  and the fourth diagram on the second line. It follows that $g +
  \gsu$ is an \Ai-morphism if and only if the equations
  \eqref{eq:cor-for-strictly-unital-morphism} hold.

  We now claim the equations
  \eqref{eq:cor-for-strictly-unital-morphism} hold if and only if
  $(\Psi^{-1}(\o g), -s^{-1}a)$ is a morphism of curved dg-coalgebras,
  i.e.,
  \begin{align*}
    \Phi^{-1}(\o \mu_{B}) \Psi^{-1}(\o g) = \Psi^{-1}(\o g) \Phi^{-1}(\o
    \mu_{A}) - \psi^{-1}(\o g) \ad {s^{-1}a}\\
    -s^{-1} h_B \Psi^{-1}(\o g) - (-s^{-1}a)^2 = -s^{-1} a \Phi^{-1}(\o
    \mu_{B}) - s^{-1}h_{A}.
  \end{align*}
  Using \ref{cor:morphisms-commuting-with-coderivations} to reduce the
  first equation, and applying $-s$ to the second, we have
  \begin{align*}
    \o \mu_{B} \Psi^{-1}(\o g) = \o g \Phi^{-1}(\o
    \mu_{A}) - \o g \ad {s^{-1}a}\\
    h_B \Psi^{-1}(\o g) + s(s^{-1}a)^2 = a \Phi^{-1}(\o
    \mu_{B}) + h_{A}.
  \end{align*}
  Using \ref{lem:ad-for-tensor-coalgebra}, one calculates that
   $\o g \ad {s^{-1}a}$ is the third and fourth terms of the first equation,
   and one checks $s(s^{-1}a)^2$ is the last diagram in the second equation.
 The other terms are easily matched to their counterparts in the
  equations \eqref{eq:cor-for-strictly-unital-morphism}, 
  which completes the proof.
\end{proof}

\subsection{Strictly unital deformation theory}
We will use without comment that if $(A,1)$ is a graded $k$-module
with split element and $l$ is a $k$-algebra, then
$(A \otimes l, 1 \otimes 1)$ is a graded $l$-module with split
element, and $\o{A \otimes l} = \o A \otimes l$.

\begin{defn}
  A \emph{strictly unital
    $(l, \epsilon)$-deformation of an \Ai-algebra with split unit $(A,
    1, \mu)$}, where $(l, \epsilon)$ is an augmented algebra, is an
  $l$-linear \Ai-algebra with split unit of the form
  $(A\otimes l, 1 \otimes 1, \bs \mu)$, such that
  $(A \otimes l, \bs \mu)$ is a nonunital $(l, \epsilon)$-deformation
  of $(A, \mu)$. We denote the resulting functor
  \begin{displaymath}
    \begin{gathered}
      \suinfAidef {(A,1,\mu)}: \finAlgk \to \Set.
    \end{gathered}
  \end{displaymath}
  If $(A, 1, \mu)$ is an augmented \Ai-algebra, an
  \emph{augmented $(l, \epsilon)$ deformation} is a strictly unital
  deformation $(A \otimes l, 1 \otimes 1, \bs \mu)$ that is augmented;
  the corresponding functor is denoted $\auginfAidef {(A,1, \mu)}$.
\end{defn}

We denote by $\msul l \in \hhca l {A \otimes l} {A \otimes l}$ the
$l$-linear trivial strictly unital algebra structure (see
\ref{def-msu} for the definition). It follows from
\ref{lem:desc-of-strictly-unital} that strictly unital elements of
$\hhca l {A \otimes l} {A \otimes l}$ are of the form
$\bs {\mu} + \msul l$, where
$$ \bs {\mu} \in \hhca l {\o{A \otimes l}} {A \otimes l} =
\hhca l {\o A \otimes l} {A \otimes l} \cong \hhc {\o A} {A \otimes
  l}.$$ Using the decomposition
$\hhc {\o A} {A \otimes l} = \hhc {\o A} {A \otimes \o l} \oplus \hhc
{\o A} A,$ the strictly unital element $\bs \mu$ is a deformation of
an \Ai-algebra structure with split unit $(A, 1, \mu + \msu)$ if and
only if $\bs \mu = \t{\bs \mu} + \theta_{l}(\mu \otimes 1)$ for some
$\t{\bs \mu} \in \hhc {\o A} {A \otimes \o l}.$

Using a different decomposition, we can write
$$\bs \mu =
\o{\bs {\mu}} + \bs h \in \hhc {\o A} {\o A \otimes l} \oplus \hhc {\o
  A} l = \hhc {\o A} {A \otimes l}.$$ The element $\bs \mu + \msul l$
is augmented if and only if $\bs h = 0.$ If $(A, 1, \o \mu + \msu)$ is
an augmented \Ai-algebra (so $(\o A, \o \mu)$ is a nonunital
\Ai-algebra) it now follows easily that there is a natural equivalence
of functors,
\begin{align*}
  \infAidef {(\overbar A, \overbar \mu)} &\to \auginfAidef {(A, 1, \overbar \mu + \msu)}\\
  \o{\bs \mu} &\mapsto \o{\bs \mu} + \msul l.
\end{align*}
Thus by \ref{prop:MC-equiv-to-defs}, the dg-Lie algebra
$(\hhc {\o A} {\o A}, [\o \mu, -])$ controls the infinitesimal
augmented deformations of the augmented \Ai-algebra
$(A, 1, \overbar \mu + \msu)$.

\begin{defn}
  The \emph{reduced Hochschild cochains} of an \Ai-algebra
  with split unit $(A, 1, \mu + \msu)$ is the dg-Lie algebra
  $(\hhc {\o A} A, [\mu + \msu, -])$ (it follows from
  \ref{lem:diagrams-for-sual} that $[\mu + \msu, -]$ preserves the
  subalgebra $\hhc {\o A} A$ of $\hhc A A$).
\end{defn}

\begin{cor}\label{cor:mc-repn-of-sual-and-aug-deforms}
  Let $(A,1, \mu + \msu)$ be \Ai-algebra with split unit. The reduced Hochschild cochains control the infinitesimal strictly
  unital deformation functor via the natural transformation
  \begin{displaymath}
    \begin{array}{ccc}
      \MCFunct {(\hhc {\overbar A} {A}, [\mu + \msu, -])} &\to& \suinfAidef{(A,1,\mu+\msu)}\\
      \t \mu & \mapsto & \theta_{l}(\t \mu +\mu \otimes 1 + \msu \otimes 1)
                         = \bs \mu + \msul l.
    \end{array}
  \end{displaymath}
\end{cor}

  \begin{proof}
    Let $(l, \epsilon)$ be an object of $\finAlgk$ and set
    $\o l = \ker \epsilon.$ By definition,
$$\MCFunct {(\hhc {\overbar A} A, [\mu + \msu, -])}(l, \epsilon)
= \MC( \hhc {\overbar A} A \otimes \o l, [\mu \otimes 1 + \msu \otimes
1, -]),$$ and the following is seen to be a bijection,\\
\scalebox{.85}{\parbox{.5\linewidth}{
  \begin{align*}
  \MC( \hhc {\overbar A} A \otimes \o l, [\mu \otimes 1 + \msu \otimes
  1, -]) &\xra{\cong} \MC(1
           \otimes \epsilon)^{-1}(\mu) \subseteq \MC(\hhc {\overbar A}
           A \otimes l, [\msu \otimes 1, -])\\
  \t \mu &\mapsto \t \mu + \mu \otimes 1.
\end{align*}}}\\
One checks that $\theta_{l}(\msu \otimes 1) = \msul l,$ and thus
$\theta_{l}$ is a morphism of dg-Lie algebras $(\hhc {\overbar A}
  A \otimes l, [\msu \otimes 1, -]) \to (\hhca l {\overbar{A \otimes
      l}} {A \otimes l},[\msul l, -])$. Since $l$ is a finitely
generated projective $k$-module, $\theta_{l}$ is an isomorphism,
and thus induces a bijection between MC elements.
The target is the set of \Ai-algebra structures on
$A \otimes l$ such that $1 \otimes 1$ is a split unit by Theorem
\ref{thm:main-thm-sunal}. Finally, one checks the bijection restricts
to a bijection
$\MC(1 \otimes \epsilon)^{-1}(\mu) \xra{\cong} \Aifam
A(\epsilon)^{-1}(\mu) = \infAidef {(A,\mu)}(l).$
\end{proof}

\begin{rems}
The reduced Hochschild cochains are quasi-isomorphic to the standard
Hochschild complex, see \cite[Theorem 4.4]{MR1989615}, but not as
dg-Lie algebras. Indeed, the functors they control, infinitesimal
strictly unital deformations and infinitesimal nonunital deformations,
are different.
\end{rems}

\section{Representations of \Ai-algebras}
In this section we treat strictly unital
\Ai-modules. In particular, we give a proof of
Positselski's result that strictly unital modules over a strictly
unital \Ai-algebra correspond to cofree curved dg-comodules over the
curved bar construction.

\subsection{Representations of nonunital \Ai-algebras}
If $(M, \delta_{M})$ is a complex of modules, $\Hom {} M M$ is
  a dg algebra with multiplication equal to composition and
  differential $\dhom = [\delta_{M}, -].$ We denote by $(\End M,
  \muend)$ the corresponding
  \Ai-algebra, see \ref{ex:dg-alg-is-ainf-alg}. 
\begin{defn}
  A \emph{representation of a nonunital \Ai-algebra $(A, \mu)$ on
 a complex $(M, \delta_{M})$} is an \Ai-morphism
  $p = (p^{n}) \in \hhc {A} {\End M}_{0}$ from $(A, \mu)$ to $(\End M,
  \muend).$
\end{defn}

\begin{defn}
  Let $M, N$ be graded modules. The \emph{adjoint} of an element
  $p^{n}$ in $\hhcn {A} {\Hom {} M N}_{k}$ is $\lambda^{n+1} =
  \ev(s^{-1}p^{n} \otimes 1): (\Pi A)^{\otimes n} \otimes M \to N$, where $\ev(f \otimes m ) = f(m)$. Thus
  $\lambda^{n+1}$ is the image of $p^{n}$ under the following isomorphisms:
\begin{align*}
\hhcn A {\Hom {} M N}_{k} & \cong \Pi \Hom {} {(\Pi A)^{\otimes n}} {\Hom
                         {} M N}_{k}\\
  &\cong \Pi \Hom {} {(\Pi A)^{\otimes n} \otimes M} N_{k} = \Hom {} {(\Pi A)^{\otimes n} \otimes M} N_{k-1}, 
\end{align*}
where the first isomorphism is from \ref{sect:notation}.(4).
In string diagrams, $\tikz[baseline={([yshift = -.1cm]current bounding
  box.center)}]{\draw[thick] (0,0) -- (0,.5)}$ denotes $\Pi A$, $\tikz[baseline={([yshift = -.1cm]current bounding
    box.center)}]{\draw[module] (0,0) -- (0,.5)}$ denotes $M$, $\tikz[baseline={([yshift = -.1cm]current bounding
  box.center)}]{\draw[secmodule] (0,0) -- (0,.5)}$ denotes $N$, and $\tikz[baseline={([yshift = -.1cm]current bounding
  box.center)}]{\draw[hom output] (0,0) -- (0,.5)}$ represents
$\Pi \Hom {} M N.$ We then have:
\begin{displaymath}
\def\wida{.33cm} 
    \def\widb{.30cm}
 \def\widk{.33cm}
 \def\dotsw{.32cm}
 \def\sepw{.1cm}
\def\sep{.1}
\def\hgt{.52cm}
 \def\wdth{1cm}
      \def\modsep{.35}
\begin{tikzpicture}
\setlength{\lengg}{\hgt + \hgt}
      \newBoxDiagram[width = \wdth, box height = \hgt, arm height =
      3*\hgt, output height = 3*\hgt, label =
      \lambda^{n+1}, coord = {($(0,
        \lengg)$)}]{g} 
      \print[/modulemap]{g}

      \pgfmathsetmacro{\rightsolidarm}{1 - \modsep}
      \printArm{hmap}{\rightsolidarm} 
\pgfmathsetmacro{\vara}{-1 +
        2*\sep} 
\pgfmathsetmacro{\varb}{\rightsolidarm - 2*\sep}
      \printArmSep[arm sep left = \vara, arm sep right =
      \varb]{g}
      \printArm{g}{\rightsolidarm}
      \printDec[/decorate, decorate right = \rightsolidarm]{g}{n}
\end{tikzpicture}
\hspace{.3cm} = \hspace{.3cm}
\begin{tikzpicture}

\setlength{\lengg}{0pt}
\newBoxDiagram[width =
      \wdth, box height = \hgt, arm height = 5*\hgt, output height =
      \hgt, label = \ev, coord = {($(0,
        \lengg)$)}]{eval}
 \printNoArms[/secmodule]{eval}
\printRightArm[/module]{eval}{1}

\eval.coord at top(coor, -.6)
      \newBoxDiagram[width = \wida, box height = \hgt, arm height =
      \hgt, output height = \hgt,output label = \Pi^{-1}, label =
      {s^{-1}}, coord = {(coor)}]{s}
 \printNoArms[/hom output]{s} 

\setlength{\whighbox}{\wdth - \sepw - \sepw}
\s.coord at top(hghcoor, 0)
\newBoxDiagram[width/.expand once = .82*\wdth, box height = \hgt, arm height =
      \hgt, output height = \hgt, label = {p^{n}}, coord = {(hghcoor)}]{p}
       \print[/hom output]{p};
       \printArmSep{p}
       \printPeriod{eval}
       \printDec[/decorate]{p}{n}
    \end{tikzpicture}
\end{displaymath}
\end{defn}

\begin{lemma}\label{lem:represen-iff-diagrams}
An element $p = (p^{n}) \in \hhc {A} {\End M}_{0}$ is a representation
of $(A, \mu = (\mu^{n}))$ on $(M,
\delta_{M})$ if and
only if the adjoint family $(\lambda^{n+1}),$ 
with $\lambda^{1} = \delta_{M}$, satisfies:
\begin{displaymath} \sum_{i = 2}^{n+1} \sum_{j=0}^{i-2} \, \,
\begin{tikzpicture}[baseline={(current bounding
            box.center)}]
          \def\w{1.4cm};
          \def\gstart{-.2}; \def\relw{.45}
          \def\sepw{.1} \def\sepwidth{.1cm} \def\leftw{.3cm}
          \def\rightw{.3cm} \def\hgt{.5cm}
          \newBoxDiagram [box height =
          \hgt, output height = \hgt, width = \w, coord =
          {($(\sepwidth + \leftw + \w,
            0)$)}, decoration yshift label = .2cm, decoration yshift label
          = .05cm, arm height = 3*\hgt, label = \lambda^{i}]{f};
          \print[/module]{f};
      
          \f.coord at top(gcoor, \gstart); \newBoxDiagram[width =
          \w*\relw, box height = \hgt, arm height = \hgt, output
          height = \hgt, coord = {(gcoor)}, output height = \hgt,
          label = \mu^{n-i+2}]{g}; \print{g}; \printArmSep{g};
          \def\smallsep{.08}\def\modsep{.35}
          \pgfmathsetmacro{\var}{\gstart - \relw - \sepw}
          \printArm{f}{\var}; \printDec[/decorate, decorate left = -1,
          decorate right = \var]{f}{j};
          \pgfmathsetmacro{\var}{-1+\smallsep}
          \pgfmathsetmacro{\vara}{\gstart - \relw - \sepw - \smallsep}
          \printArmSep[arm sep left = \var, arm sep right = \vara]{f}
          \pgfmathsetmacro{\var}{\gstart + \relw + \sepw}
          \printArm{f}{\var}
          \pgfmathsetmacro{\vara}{1-\modsep+\smallsep}
          \printDec[/decorate, decorate left = \var, decorate right = \vara]{f}{i
          - j - 2}
\pgfmathsetmacro{\var}{1-\modsep}
          \pgfmathsetmacro{\vara}{\gstart + \relw + \sepw + \smallsep}
          \pgfmathsetmacro{\varb}{\var+\smallsep} \printArmSep[arm sep
          left = \vara, arm sep right = \var]{f} \printArm{f}{\varb}
        \end{tikzpicture}
\hspace{.2cm} + \hspace{.2cm}
\sum_{i = 1}^{n+1}\, \, \begin{tikzpicture}[baseline={(current bounding
            box.center)}] \def\w{1.3cm}; \def\gstart{.5};
          \def\relw{.52} \def\sepw{.1} \def\sepwidth{.1cm}
          \def\leftw{.3cm} \def\rightw{.3cm} \def\hgt{.5cm}
          \newBoxDiagram [box height = \hgt, output height = \hgt,
          width = \w, coord = {($(\sepwidth + \leftw + \w,
            0)$)}, decoration yshift label = .2cm, decoration yshift
          label = .05cm, arm height = 3*\hgt, label = {\lambda^{i}}]{f};
          \printNoArms[/module]{f}; \printArm{f}{-1};
      
          \def\smallsep{.08}\def\modsep{.35} \f.coord at top(gcoor,
          \gstart); 
\newBoxDiagram[width = \w*\relw, box height =
          \hgt, arm height = \hgt, output height = \hgt, coord =
          {(gcoor)}, output height = \hgt, decoration yshift
          label = .05cm, label = {\lambda^{n-i+2}}]{g};
          \print[/module]{g}; \pgfmathsetmacro{\var}{1-\modsep}
          \printArm{g}{\var} 
\printDec[/decorate, decorate right = \var]{g}{n-i+1}
\printArmSep[arm sep right = .2]{g};
          \pgfmathsetmacro{\var}{\gstart - \relw - \sepw}
          \printArm{f}{\var};
          \printDec[/decorate, decorate right = \var]{f}{i-1}

\pgfmathsetmacro{\var}{-1+2*\smallsep}
          \pgfmathsetmacro{\vara}{\gstart - \relw - \sepw -
            2*\smallsep} \printArmSep[arm sep left = \var, arm sep
          right = \vara]{f}
        \end{tikzpicture} = 0.
\end{displaymath}
\end{lemma}

\begin{proof}
By the definition of \Ai-morphism, $p$ is a
representation if and only if $p \circ \mu_{A} - \muend * p = 0.$ This
equation holds if and only if it holds in every tensor degree. Applying the isomorphism
$\hhcn A {\End M}_{0} \cong \Hom {} {(\Pi A)^{\otimes n} \otimes M} M_{-1}$,
one checks the equation in tensor degree $n$ is equivalent to the diagrams above, with $p \circ
\mu_{A}$ corresponding to the left diagram and $-\muend * p$ to the
right diagram.
\end{proof}

To define a morphism of representations, we need to add a counit
to the tensor coalgebra (else we would have to fix a morphism of
complexes, and talk about morphisms of
representations over that fixed morphism of complexes). Set
\begin{align*}
\Tcou {\Pi A} &= k \times \Tco {\Pi A} = \bigoplus_{n \geq 0} (\Pi
  A)^{\otimes n}\\
  \hhcu A B &= \Hom {} {\Tcou {\Pi A}} {\Pi B} \cong \hhc A B \oplus\Pi
B.
\end{align*}
Using this isomorphism, given a representation $p$ on a complex $(M, \delta_{M})$, we set
$p_{M} = p + \delta_{M} \in \hhcu A {\End M}_{0}$. Conversely, we can view
 representations as elements $p_{M} \in \hhcu A {\End M}_{0}$
such that $p^{0}_{M} \in \End M$ is a differential and $p^{\geq 1}_{M}$ is an
\Ai-morphism from $A$ to the endomorphism \Ai-algebra of the complex
$(M, p^{0}_{M}).$ 

\begin{defn}
Let $M, N, P$ be graded modules. We consider the action
\begin{gather*}
  \star:  \hhcu A {\Hom {} N P}_{k} \otimes \hhcu A {\Hom {} M N}_{l} \to \hhcu A {\Hom
  {} M P}_{k + l -1 }\\
  \alpha \otimes \beta = (\alpha^n) \otimes (\beta^n) \mapsto \left (\gamma  \sum_{j =
  0}^{n} \alpha^j\otimes \beta^{n-j}\right) = \alpha \star \beta,
\end{gather*}
where $\gamma = s c (s^{-1} \otimes s^{-1}),$ with $c$ the
composition map. If $a^{n+1}$ and $b^{n+1}$ are the adjoints of
$\alpha^{n}$ and $\beta^{n}$, and $\tikz[baseline={([yshift = -.1cm]current bounding
  box.center)}]{\draw[thirdmodule] (0,0) -- (0,.5)}$ represents
$P$, then the adjoint of $(\alpha \star \beta)^n$ is
\begin{displaymath}
          \def\w{1.4cm}
          \def\gstart{.4} \def\relw{.5}
          \def\sepw{.1} \def\sepwidth{.1cm} \def\leftw{.3cm}
          \def\rightw{.3cm} \def\hgt{.55cm}
\def\smallsep{.08} \def\modsep{.35}
 (-1)^{|a|-1}\sum_{i = 1}^{n+1} \, \,
\begin{tikzpicture}[baseline={(current bounding
            box.center)}]
          \newBoxDiagram [box height = \hgt, output height = \hgt,
          width = \w, coord = {($(\sepwidth + \leftw + \w,
            0)$)}, decoration yshift label = .2cm, decoration yshift
          label = .05cm, arm height = 3*\hgt, label = {a^{i}}]{g};
          \printNoArms[/secmodulemap]{g}; \printArm{g}{-1};
      
         \g.coord at top(gcoor,
          \gstart); 
\newBoxDiagram[width = \w*\relw, box height =
          \hgt, arm height = \hgt, output height = \hgt, coord =
          {(gcoor)}, output height = \hgt, decoration yshift
          label = .05cm, label = {b^{n-i+2}}]{m};
          \print[/modulemap]{m}; \pgfmathsetmacro{\var}{1-\modsep}
          \printArm{m}{\var} 
\printDec[/decorate, decorate right = \var]{m}{n-i+1}
\printArmSep[arm sep right = .2]{m};
          \pgfmathsetmacro{\var}{\gstart - \relw - \sepw}
          \printArm{g}{\var};
          \printDec[/decorate, decorate right = \var]{g}{i-1}

\pgfmathsetmacro{\var}{-1+2*\smallsep}
          \pgfmathsetmacro{\vara}{\gstart - \relw - \sepw -
            2*\smallsep} \printArmSep[arm sep left = \var, arm sep
          right = \vara]{g}
          \printPeriod{g}
        \end{tikzpicture}
\end{displaymath}
\end{defn}

\begin{defn}
\emph{A morphism of representations $(M, p_{M}) \to (N, p_{N})$ of a nonunital
\Ai-algebra $(A, \mu)$} is
  an element $f \in \hhcu {A} {\Hom {} M N}_{1}$ such that
  $p_{N} \star f + (f^{\geq 1})\circ \mu +  f \star p_{M} = 0.$ The
  composition with a second morphism, $\t f \in \hhcu {A} {\Hom {} N
    P}_{1}$ is $\t f \star f \in \hhcu A {\Hom {} M N}_{1}.$
\end{defn}

\begin{lemma}\label{lem:morphism-represen-iff-diagrams}
Let $(M, p_{M})$ and $(N, p_{N})$ be representations of a nonunital
\Ai-algebra $(A, \mu)$, with adjoints $\lambda_{M}$ and $\lambda_{N},$ respectively. An element $f = (f^{n}) \in \hhcu {A} {\Hom {}
  M N}_{1}$ is a morphism of representations
 if and
only if the adjoint family $(g^{n+1})$ satisfies the equations
\begin{displaymath}
          \def\w{1.2cm}
          \def\gstart{.5} \def\relw{.5}
          \def\sepw{.1} \def\sepwidth{.1cm} \def\leftw{.3cm}
          \def\rightw{.3cm} \def\hgt{.55cm}
\def\smallsep{.08} \def\modsep{.35}
 \sum_{i = 1}^{n+1} \, \,
\begin{tikzpicture}[baseline={(current bounding
            box.center)}]

          \newBoxDiagram [box height =
          \hgt, output height = \hgt, width = \w, coord =
          {($(\sepwidth + \leftw + \w,
            0)$)}, decoration yshift label = .2cm, decoration yshift label
          = .05cm, arm height = 3*\hgt, label = {\lambda_{N}^{i}}]{f};
          \printNoArms[/secmodule]{f};
          \printArm{f}{-1};
     
           \f.coord at top(gcoor,
          \gstart);
 \newBoxDiagram[width = \w*\relw, box height =
          \hgt, arm height = \hgt, output height = \hgt, coord =
          {(gcoor)}, output height = \hgt, decoration yshift
          label = .05cm, label = {g^{n-i+2}}]{g};

          \print[/modulemap]{g}; 
\pgfmathsetmacro{\var}{1-\modsep}
          \printArm{g}{\var} 
\printDec[/decorate, decorate right = \var]{g}{n-i+1}
\printArmSep[arm sep right = .2]{g};
          \pgfmathsetmacro{\var}{\gstart - \relw - \sepw}
          \printArm{f}{\var}; 
\printDec[/decorate, decorate right = \var]{f}{i-1}
\pgfmathsetmacro{\var}{-1+2*\smallsep}
          \pgfmathsetmacro{\vara}{\gstart - \relw - \sepw -
            2*\smallsep} 
\printArmSep[arm sep left = \var, arm sep
          right = \vara]{f}
        \end{tikzpicture} 
\hspace{.2cm} - \hspace{.2cm}
\sum_{i = 2}^{n+1} \sum_{j=0}^{i-2} \, \,
\begin{tikzpicture}[baseline={(current bounding
            box.center)}]
          \def\w{1.45cm};
          \def\gstart{-.2}; \def\relw{.45}
          \def\sepw{.1} \def\sepwidth{.1cm} \def\leftw{.3cm}
          \def\rightw{.3cm} \def\hgt{.55cm}
          \newBoxDiagram [box height =
          \hgt, output height = \hgt, width = \w, coord =
          {($(\sepwidth + \leftw + \w,
            0)$)}, decoration yshift label = .2cm, decoration yshift label
          = .05cm, arm height = 3*\hgt, label = g^{i}]{f};
          \print[/modulemap]{f};
      
          \f.coord at top(gcoor, \gstart); \newBoxDiagram[width =
          \w*\relw, box height = \hgt, arm height = \hgt, output
          height = \hgt, coord = {(gcoor)}, output height = \hgt,
          label = \mu^{n-i+2}]{g}; \print{g}; \printArmSep{g};
          \def\smallsep{.08}\def\modsep{.35}
          \pgfmathsetmacro{\var}{\gstart - \relw - \sepw}
          \printArm{f}{\var}; \printDec[/decorate, decorate left = -1,
          decorate right = \var]{f}{j};
          \pgfmathsetmacro{\var}{-1+\smallsep}
          \pgfmathsetmacro{\vara}{\gstart - \relw - \sepw - \smallsep}
          \printArmSep[arm sep left = \var, arm sep right = \vara]{f}
          \pgfmathsetmacro{\var}{\gstart + \relw + \sepw}
          \printArm{f}{\var}
          \pgfmathsetmacro{\vara}{1-\modsep+\smallsep}
          \printDec[/decorate, decorate left = \var, decorate right = \vara]{f}{i
          - j - 2}
\pgfmathsetmacro{\var}{1-\modsep}
          \pgfmathsetmacro{\vara}{\gstart + \relw + \sepw + \smallsep}
          \pgfmathsetmacro{\varb}{\var+\smallsep} \printArmSep[arm sep
          left = \vara, arm sep right = \var]{f} \printArm{f}{\varb}
        \end{tikzpicture}
\hspace{.2cm} - \hspace{.2cm}
\sum_{i = 1}^{n+1}\, \, \begin{tikzpicture}[baseline={(current bounding
            box.center)}]\def\gstart{.3}\def\relw{.6}
          \newBoxDiagram [box height = \hgt, output height = \hgt,
          width = \w, coord = {($(\sepwidth + \leftw + \w,
            0)$)}, decoration yshift label = .2cm, decoration yshift
          label = .05cm, arm height = 3*\hgt, label = {g^{i}}]{g};
          \printNoArms[/modulemap]{g}; \printArm{g}{-1};
      
         \g.coord at top(gcoor,
          \gstart); 
\newBoxDiagram[width = \w*\relw, box height =
          \hgt, arm height = \hgt, output height = \hgt, coord =
          {(gcoor)}, output height = \hgt, decoration yshift
          label = .05cm, label = {\lambda^{n-i+2}_{M}}]{m};
          \print[/module]{m}; \pgfmathsetmacro{\var}{1-\modsep}
          \printArm{m}{\var} 
\printDec[/decorate, decorate right = \var]{m}{n-i+1}
\printArmSep[arm sep right = .2]{m};
          \pgfmathsetmacro{\var}{\gstart - \relw - \sepw}
          \printArm{g}{\var};
          \printDec[/decorate, decorate right = \var]{g}{i-1}

\pgfmathsetmacro{\var}{-1+2*\smallsep}
          \pgfmathsetmacro{\vara}{\gstart - \relw - \sepw -
            2*\smallsep} \printArmSep[arm sep left = \var, arm sep
          right = \vara]{g}
        \end{tikzpicture} = 0.
\end{displaymath}
\end{lemma}

Indeed, each of the three terms above is the adjoint of (-1) times the corresponding
term in the definition of morphism.

\subsection{Representations of strictly unital \Ai-algebras}


\begin{defn}
Let $(A,1, \mu)$ be a strictly unital \Ai-algebra. A \emph{strictly
  unital representation} on a complex $(M, \delta_{M})$ is a strictly unital
\Ai-morphism $p \in \hhc A {\End M}_{0}$. A \emph{morphism of strictly unital
representations $(M, p_{M})\to (N, p_{N})$} is a morphism of
representations $f$ such that $f \in \hhcu {\o A}
{\Hom {}M N}_{1},$ i.e.,
$$f^{n}([a_{1}|\ldots|a_{i-1}|1|a_{i+1}|\ldots|a_{n}]) = 0 \quad \text{for all $n
\geq 1$}.$$
\end{defn}

In string diagrams, $\tikz[baseline={([yshift = -.1cm]current bounding
  box.center)}]{\draw[thick] (0,0) -- (0,.5)}$ will now denote $\Pi \o A$
(previously it denoted $\Pi A$), while  $\tikz[baseline={([yshift = -.1cm]current bounding
    box.center)}]{\draw[module] (0,0) -- (0,.5)}$ continues to represent $M$, $\tikz[baseline={([yshift = -.1cm]current bounding
  box.center)}]{\draw[secmodule] (0,0) -- (0,.5)}$ represents $N$, and $\tikz[baseline={([yshift = -.1cm]current bounding
    box.center)}]{\draw[unit] (0,0) -- (0,.5)}$ represents $\Pi k.$

\begin{lem}\label{lem:diagrams-for-s-ual-reps}
Let $(A, 1, \o \mu + h + \msu)$ be an \Ai-algebra with split unit.
\begin{enumerate}
\item A strictly unital element $p = \o p +
\gsu \in \hhc A {\End M}_{0}$, with $\o p \in \hhcu {\o A} {\End M}_{0}$, is a representation if and only if the
adjoint family $\o \lambda = (\o \lambda^{n+1})$ of $\o p$, where $\o \lambda^{1} = \delta_{M}$, satisfies
\begin{displaymath}
            \def\w{1.4cm}
          \def\gstart{-.2} \def\relw{.45}
          \def\sepw{.1} \def\sepwidth{.1cm} \def\leftw{.3cm}
          \def\rightw{.3cm} \def\hgt{.5cm}
 \sum_{i = 2}^{n+1} \sum_{j=0}^{i-2} \, \,
\begin{tikzpicture}[baseline={(current bounding
            box.center)}]
          \newBoxDiagram [box height =
          \hgt, output height = \hgt, width = \w, coord =
          {($(\sepwidth + \leftw + \w,
            0)$)}, decoration yshift label = .2cm, decoration yshift label
          = .05cm, arm height = 3*\hgt, label = {\o \lambda}^{i}]{f};
          \print[/module]{f};
      
          \f.coord at top(gcoor, \gstart); 
\newBoxDiagram[width = 
          \w*\relw, box height = \hgt, arm height = \hgt, output
          height = \hgt, coord = {(gcoor)}, output height = \hgt,
          label = {\o \mu}^{n-i+2}]{g}; \print{g}; \printArmSep{g};
          \def\smallsep{.08}\def\modsep{.35}
          \pgfmathsetmacro{\var}{\gstart - \relw - \sepw}
          \printArm{f}{\var}; \printDec[/decorate, decorate left = -1,
          decorate right = \var]{f}{j};
          \pgfmathsetmacro{\var}{-1+\smallsep}
          \pgfmathsetmacro{\vara}{\gstart - \relw - \sepw - \smallsep}
          \printArmSep[arm sep left = \var, arm sep right = \vara]{f}
          \pgfmathsetmacro{\var}{\gstart + \relw + \sepw}
          \printArm{f}{\var}
          \pgfmathsetmacro{\vara}{1-\modsep+\smallsep}
          \printDec[/decorate, decorate left = \var, decorate right = \vara]{f}{i
          - j - 2}
\pgfmathsetmacro{\var}{1-\modsep}
          \pgfmathsetmacro{\vara}{\gstart + \relw + \sepw + \smallsep}
          \pgfmathsetmacro{\varb}{\var+\smallsep} \printArmSep[arm sep
          left = \vara, arm sep right = \var]{f} \printArm{f}{\varb}
        \end{tikzpicture}
\hspace{.2cm} + \hspace{.2cm}
\sum_{i = 1}^{n+1}\, \, \begin{tikzpicture}[baseline={(current bounding
            box.center)}]
  \def\w{1.1cm}; \def\gstart{.35};
           \def\relw{.6}
          \newBoxDiagram [box height = \hgt, output height = \hgt,
          width = \w, coord = {($(\sepwidth + \leftw + \w,
            0)$)}, decoration yshift label = .2cm, decoration yshift
          label = .05cm, arm height = 3*\hgt, label = {\o \lambda^{i}}]{f};
          \printNoArms[/module]{f}; \printArm{f}{-1};
      
          \def\smallsep{.08}\def\modsep{.35} \f.coord at top(gcoor,
          \gstart); 
\newBoxDiagram[width = \w*\relw, box height =
          \hgt, arm height = \hgt, output height = \hgt, coord =
          {(gcoor)}, output height = \hgt,decoration yshift
          label = .05cm, label = {\o \lambda^{n-i+2}}]{g};
          \print[/module]{g}; \pgfmathsetmacro{\var}{1-\modsep}
          \printArm{g}{\var} 
\printDec[/decorate, decorate right = \var]{g}{n-i+1}
\printArmSep[arm sep right = .2]{g};
          \pgfmathsetmacro{\var}{\gstart - \relw - \sepw}
          \printArm{f}{\var};
          \printDec[/decorate, decorate right = \var]{f}{i-1}

\pgfmathsetmacro{\var}{-1+2*\smallsep}
          \pgfmathsetmacro{\vara}{\gstart - \relw - \sepw -
            2*\smallsep} \printArmSep[arm sep left = \var, arm sep
          right = \vara]{f}
        \end{tikzpicture}
\hspace{.1cm} + \hspace{.1cm}
\begin{tikzpicture}

  \def\gstart{-.3}
\def\wida{.33cm} 
    \def\widb{.30cm}
 \def\widk{.33cm}
 \def\dotsw{.32cm}
 \def\sepw{.1cm}
 \def\wdth{.75cm}
\setlength{\lengg}{0pt}

      \newBoxDiagram[width = \wdth, box height = \hgt, arm height =
      3*\hgt, output height = \hgt, label =
      \small{(\cong)}, coord = {($(0,
        0)$)}]{isom}
      \printNoArms[/module]{isom}
      \printRightArm[/module]{isom}{1}

\setlength{\whighbox}{\wdth - \sepw}
\isom.coord at top(hghcoor,\gstart)
\newBoxDiagram[width/.expand once = \whighbox, box height = \hgt, arm height =
      \hgt, output height = \hgt, label = {s^{-1}h^{n}}, output label
      = \Pi^{-1}, coord = {(hghcoor)}]{gmap}
       \print[/unit output]{gmap}
\printArmSep{gmap};
\printArm{gmap}{-1};
\printArm{gmap}{1};
\printDec[color = white]{gmap}{test}
    \end{tikzpicture} 
\hspace{.1cm} = \hspace{.1cm} 0.
\end{displaymath}
\item An
  element $f \in \hhcu {\o A} {\Hom {} M N}_{1}$ is a morphism of
  strictly unital
  representations $(M,\o p_{M}) \to (N, \o p_{N})$ if and only if the
  following holds, where $g, \o \lambda_{M}, \o \lambda_{N}$ are the adjoint
  families of $f, \o p_{M}, \o p_{N}:$
\begin{displaymath}
          \def\w{1.2cm}
          \def\gstart{.5} \def\relw{.5}
          \def\sepw{.1} \def\sepwidth{.1cm} \def\leftw{.3cm}
          \def\rightw{.3cm} \def\hgt{.55cm}
\def\smallsep{.08} \def\modsep{.35}
 \sum_{i = 1}^{n+1} \, \,
\begin{tikzpicture}[baseline={(current bounding
            box.center)}]

          \newBoxDiagram [box height =
          \hgt, output height = \hgt, width = \w, coord =
          {($(\sepwidth + \leftw + \w,
            0)$)}, decoration yshift label = .2cm, decoration yshift label
          = .05cm, arm height = 3*\hgt, label = {\o \lambda_{N}^{i}}]{f};
          \printNoArms[/secmodule]{f};
          \printArm{f}{-1};
     
           \f.coord at top(gcoor,
          \gstart);
 \newBoxDiagram[width = \w*\relw, box height =
          \hgt, arm height = \hgt, output height = \hgt, coord =
          {(gcoor)}, output height = \hgt, decoration yshift
          label = .05cm, label = {g^{n-i+2}}]{g};

          \print[/modulemap]{g}; 
\pgfmathsetmacro{\var}{1-\modsep}
          \printArm{g}{\var} 
\printDec[/decorate, decorate right = \var]{g}{n-i+1}
\printArmSep[arm sep right = .2]{g};
          \pgfmathsetmacro{\var}{\gstart - \relw - \sepw}
          \printArm{f}{\var}; 
\printDec[/decorate, decorate right = \var]{f}{i-1}
\pgfmathsetmacro{\var}{-1+2*\smallsep}
          \pgfmathsetmacro{\vara}{\gstart - \relw - \sepw -
            2*\smallsep} 
\printArmSep[arm sep left = \var, arm sep
          right = \vara]{f}
        \end{tikzpicture} 
\hspace{.2cm} - \hspace{.2cm}
\sum_{i = 2}^{n+1} \sum_{j=0}^{i-2} \, \,
\begin{tikzpicture}[baseline={(current bounding
            box.center)}]
          \def\w{1.45cm};
          \def\gstart{-.2}; \def\relw{.45}
          \def\sepw{.1} \def\sepwidth{.1cm} \def\leftw{.3cm}
          \def\rightw{.3cm} \def\hgt{.55cm}
          \newBoxDiagram [box height =
          \hgt, output height = \hgt, width = \w, coord =
          {($(\sepwidth + \leftw + \w,
            0)$)}, decoration yshift label = .2cm, decoration yshift label
          = .05cm, arm height = 3*\hgt, label = g^{i}]{f};
          \print[/modulemap]{f};
      
          \f.coord at top(gcoor, \gstart); \newBoxDiagram[width =
          \w*\relw, box height = \hgt, arm height = \hgt, output
          height = \hgt, coord = {(gcoor)}, output height = \hgt,
          label = \o \mu^{n-i+2}]{g}; \print{g}; \printArmSep{g};
          \def\smallsep{.08}\def\modsep{.35}
          \pgfmathsetmacro{\var}{\gstart - \relw - \sepw}
          \printArm{f}{\var}; \printDec[/decorate, decorate left = -1,
          decorate right = \var]{f}{j};
          \pgfmathsetmacro{\var}{-1+\smallsep}
          \pgfmathsetmacro{\vara}{\gstart - \relw - \sepw - \smallsep}
          \printArmSep[arm sep left = \var, arm sep right = \vara]{f}
          \pgfmathsetmacro{\var}{\gstart + \relw + \sepw}
          \printArm{f}{\var}
          \pgfmathsetmacro{\vara}{1-\modsep+\smallsep}
          \printDec[/decorate, decorate left = \var, decorate right = \vara]{f}{i
          - j - 2}
\pgfmathsetmacro{\var}{1-\modsep}
          \pgfmathsetmacro{\vara}{\gstart + \relw + \sepw + \smallsep}
          \pgfmathsetmacro{\varb}{\var+\smallsep} \printArmSep[arm sep
          left = \vara, arm sep right = \var]{f} \printArm{f}{\varb}
        \end{tikzpicture}
\hspace{.2cm} - \hspace{.2cm}
\sum_{i = 1}^{n+1}\, \, \begin{tikzpicture}[baseline={(current bounding
            box.center)}]\def\gstart{.3}\def\relw{.6}
          \newBoxDiagram [box height = \hgt, output height = \hgt,
          width = \w, coord = {($(\sepwidth + \leftw + \w,
            0)$)}, decoration yshift label = .2cm, decoration yshift
          label = .05cm, arm height = 3*\hgt, label = {g^{i}}]{g};
          \printNoArms[/modulemap]{g}; \printArm{g}{-1};
      
         \g.coord at top(gcoor,
          \gstart); 
\newBoxDiagram[width = \w*\relw, box height =
          \hgt, arm height = \hgt, output height = \hgt, coord =
          {(gcoor)}, output height = \hgt, decoration yshift
          label = .05cm, label = {\o \lambda^{n-i+2}_{M}}]{m};
          \print[/module]{m}; \pgfmathsetmacro{\var}{1-\modsep}
          \printArm{m}{\var} 
\printDec[/decorate, decorate right = \var]{m}{n-i+1}
\printArmSep[arm sep right = .2]{m};
          \pgfmathsetmacro{\var}{\gstart - \relw - \sepw}
          \printArm{g}{\var};
          \printDec[/decorate, decorate right = \var]{g}{i-1}

\pgfmathsetmacro{\var}{-1+2*\smallsep}
          \pgfmathsetmacro{\vara}{\gstart - \relw - \sepw -
            2*\smallsep} \printArmSep[arm sep left = \var, arm sep
          right = \vara]{g}
        \end{tikzpicture} = 0.
\end{displaymath}
\end{enumerate}
\end{lem}

\begin{proof}
By Lemma
\ref{lem:strictly-unital-morphisms}, $\o p + \gsu$ is an \Ai-morphism if and
only if $\o p \circ \o \mu - \mu_{\End {}} * \o p  + \gsu \circ
h_{A}= 0$. One checks that the adjoints of the three terms of this
equation agree with the three families of displayed diagrams, and this
proves part 1.

For part 2, the element $f$ is a morphism of representation if and only if
$p_{N} \star f + f^{\geq 1} \circ \mu + f \star p_{M} = 0.$ Using the
decompositions $\mu = \o \mu + h + \msu$, $p_{M} = \o p_{M} + \gsu,$ and $p_{N} = \o p_{N} +
\gsu,$ together with the equality $\gsu \star f - f^{\geq 1}
\circ \msu + f \star \gsu = 0$, we see $f$ is a morphism if and
only if $\o p_{N} \star f + f^{\geq 1}
\circ \o \mu + f \star \o p_{M} = 0.$
Each of the three terms of this equation is the adjoint of $(-1)$ times the
corresponding term in the displayed equation.
\end{proof}

\begin{ex}\label{ex:strictly-unital-modules-over-Koszul-complex}
 Let $(A, 1, \mu)$ be the Koszul complex on $f \in k,$ see
 \ref{ex:curved-bar-constr-of-koszul-cx}, and $M$ a graded
 module. Let $\o p_{M} \in \hhcu {\o A} {\End M}_{0}$ be an arbitrary element
 with adjoint family $(\o \lambda^{n})$ and set $\sigma^{n}: M \xra{\cong}
 k[e]^{\otimes n} \otimes M \xra{\overbar \lambda^{n}} M,$ a degree $2n -1$
 endomorphism of $M$. Since $\o \mu = 0,$ we see that $\o p_{M}$ is a
 representation if and only if $\sigma^{1} \sigma^0 + \sigma^0
 \sigma^{1} = -f \cdot 1_{M}$ and $\sum_{i = 0}^n \sigma^{n-i} \sigma^i
 = 0$ for $n \geq 2.$ Such a system of maps was first considered by
 Shamash \cite{MR0241411}, who assumed that $M$ was the $k$-free
 resolution of a $k/(f)$-module, and has since been important in
 the construction of free resolutions in commutative algebra; see e.g., \cite[\S 3.1]{MR1648664}
 and the references contained there.
\end{ex}

\subsection{Comodules}
If $(A, \mu)$ is a nonunital \Ai-algebra, set $\Baru A$ to be the
coaugmented bar construction $(\Tcou
{\Pi A}, d)$ (here $d|_{\Tco {\Pi A}} = \Phi^{-1}(\mu)$ and $d(1) = 0$).

\begin{defn}
  Let $C$ be a graded coalgebra. The \emph{cofree $C$-comodule on a
    graded module $M$} has underlying graded module $C \otimes M$ and
  comultiplication $\Delta_C \otimes 1$. If $d$ is a graded
  coderivation of $C$ and $P$ is a graded
  $C$-comodule, a \emph{coderivation} of $P$ (with respect to $d$) is
  a homogeneous map $d_P: P \to P,$ with $|d_P| = |d|,$ that satisfies
  $(d \otimes 1 + 1 \otimes d_{P})\Delta_{P} = \Delta_P d_{P}.$ We
  denote by $\Coder^{d}(P, P)$ the set of coderivations of $P$. If
  $(C, d)$ is a dg-coalgebra, a dg-comodule is a pair $(P, d_{P})$ with
  $P$ a graded comodule and $d_{P}$ an element of $\Coder^{d}(P, P)$ such that
  $d_{P}^{2} = 0.$ A morphism of dg-comodules is a morphism of
  comodules that commutes with the given coderivations. A
  dg-comodule is 
  \emph{cofree} if the underlying comodule is cofree.
\end{defn}

Cofree comodules satisfy the linear analogue of
  \ref{lem:isom_hhc_coder}.
\begin{lemma}
  \label{rem:univ-props-cofree-comods}
 Let $(C, \epsilon)$ be a graded coalgebra with counit
  $\epsilon: C \to k$, and $M$ a graded
  module. The following hold.
  \begin{enumerate}
  \item  For any degree $n$ coderivation $d$ of $C$, the following is an isomorphism,
    \[ \phi: \operatorname{Coder}^{d}(C \otimes M, C \otimes M)
      \xra{\cong} \Homd {} {n} {C \otimes M} {M}\]
    \begin{displaymath}
      d_{C \otimes M} \mapsto (\epsilon \otimes 1) d_M,
    \end{displaymath}
    with
    $\phi^{-1}(m) = d \otimes 1 + (1 \otimes m)(\Delta_C \otimes 1).$ A
    coderivation $\phi^{-1}(m)$ is a differential, i.e., squares to
    zero, if and only if $m \phi^{-1}(m) = 0.$
  \item For any graded $C$-comodule $P$, the following is an isomorphism,
    \[ \psi:\Hom C P {C \otimes M} \xra{\cong} \Hom {} P M,\]
    \begin{displaymath}
      \beta \mapsto (\epsilon \otimes 1) \beta,
    \end{displaymath}
 where
    $\Hom C {-} {-}$ denotes morphisms of graded $C$-comodules. The
    inverse is given by $\psi^{-1}(\alpha) =  (1 \otimes \alpha) \Delta_P$. A morphism
 $\psi^{-1}(g)$ commutes with coderivations
    $d_{P}$ of $P$ and $\phi^{-1}(m)$ of $C \otimes
    M$ if and only if $g d_{P} = m_{N} \psi^{-1}(g).$
  \end{enumerate}
\end{lemma}

Note the above properties emphasize the need to adjoin a counit to $\Bar
A$.

\begin{prop}
  Let $(A, \mu)$ be a nonunital \Ai-algebra with counital bar construction
  $\Baru A$, and $M, N$ graded modules.

  \begin{enumerate}
  \item An element $p_{M} \in \hhcu A {\End M}_{0}$ is a
    representation of $(A, \mu)$ if and only if, for $\lambda_{M}$ the
    adjoint family, the pair $(\Baru A \otimes M,
    \phi^{-1}(\lambda_{M}))$ is a dg-$\Baru A$ comodule.

    \item An element $f \in \hhcu A {\Hom {} M N}_{1}$ is a
      morphism of representations $(M, p_{M}) \to
      (N,p_{N})$ if and only if $\psi^{-1}(g): (\Baru
      A \otimes M,\phi^{-1}(\lambda_{M})) \to (\Baru
      A \otimes N,\phi^{-1}(\lambda_{N}))$ is a morphism of dg-$\Baru
      A$ comodules, where $g, \lambda_{M}, \lambda_{N}$ are the adjoint families
      of $f, p_{M}, p_{N}.$
  \end{enumerate}
\end{prop}

\begin{proof}
The pair $(\Baru A \otimes M,
    \phi^{-1}(\lambda_{M}))$ is a dg $\Baru A$-comodule if and only
    if $(\phi^{-1}(\lambda_{M}))^{2} = 0.$ By
    \ref{rem:univ-props-cofree-comods}.(1), this is equivalent to the equation
    $\lambda_{M} \phi^{-1}(\lambda_{M}) = 0,$ and by the definition of
    $\phi^{-1},$ we see this is equivalent
    to the equations of \ref{lem:represen-iff-diagrams}.

Analogously, $\psi^{-1}(g)$ is a morphism of dg-comodules if and only
if it commutes with the coderivations $\phi^{-1}(\lambda_{M})$ and
$\phi^{-1}(\lambda_{N})$. By \ref{rem:univ-props-cofree-comods}.(2) this
is equivalent to $g \phi^{-1}(\lambda_{M}) = \lambda_{N} \psi^{-1}(g)$, and from
the definitions of $\phi^{-1}$ and $\psi^{-1},$ this in
 equivalent to the equations of \ref{lem:morphism-represen-iff-diagrams}.
\end{proof}

\begin{cor}
Let $(A, \mu)$ be a nonunital \Ai-algebra. There is a functor from the
category of
representations of $A$ to the category of dg $\Baru A$ comodules, that sends $(M,
p_{M})$ to $(\Baru A \otimes M, \phi^{-1}(\lambda_{M}))$. This is fully faithful with
image the full subcategory of cofree dg comodules.
\end{cor}

We now assume that $A$ has a split unit, and construct the analogue of
the above for strictly unital representations of $A$. 

\begin{defn}
A \emph{curved dg-comodule} over a curved dg-coalgebra $(C, d, \xi)$
is a pair $(P, d_{P}),$ with $P$ a graded $C$ comodule and $d_{P} \in \Coder^{d}(P, P)_{-1},$ that satisfies
\[d_P^{2} = \left ( P \xra{\Delta} C \otimes P \xra{\xi \otimes 1} k
    \otimes P \cong P \right ) =: L_{\xi}.\]
A \emph{morphism of curved dg-comodules $(P, d_{P}) \to (N, d_{N})$}
is a degree zero morphism of graded $C$-comodules $f: P \to N$ that satisfies $f d_{P} = d_{N} f.$
\end{defn}

If $(A, 1, \o \mu + h + \msu)$ is an \Ai-algebra with split unit, we
denote by $\Baru {\o A}$ the counital curved bar construction $(\Tcou {\Pi \o A}, \Phi^{-1}(\o \mu),
-s^{-1}h),$ where $\Phi^{-1}(\o \mu)$ and $-s^{-1}h$ are extended by zero from $\Tco {\Pi \o A}$
to $\Tcou {\Pi  \o A}.$

\begin{thm}
Let $(A, 1, \o \mu + h + \msu)$ be an \Ai-algebra with split unit and
counital curved bar construction $\Baru {\o A}.$
Let $M, N$ be graded modules.
\begin{enumerate}
\item A strictly unital element $p = \o p + \gsu$, with $\o p \in
  \hhcu {\o A} {\End M}_{0}$, is a representation if and only if, for
  $\o \lambda$ the adjoint family of $\o p$, the pair
  $(\Baru {\o A} \otimes M, \phi^{-1}(\o \lambda))$ is a curved dg-$\Baru
  {\o A}$ comodule.
 
    \item  An element $f \in \hhcu {\o A} {\Hom {} M N}_{1}$ is a
      morphism of strictly unital representations $(M, \o p_{M}) \to
      (N,\o p_{N})$ if and only if
      \begin{displaymath}
\psi^{-1}(g): (\Baru
      {\o A} \otimes M,\phi^{-1}(\o \lambda_{M})) \to (\Baru
      {\o A} \otimes N,\phi^{-1}(\o \lambda_{N}))
    \end{displaymath}
    is a morphism of curved dg-$\Baru
      {\o A}$ comodules, where $g, \o \lambda_{M},\o \lambda_{N}$ are the adjoint families of
      $f, \o p_{M}, \o p_{N}.$
\end{enumerate}
\end{thm}

\begin{proof}
The pair $(\Baru {\o A} \otimes M, \phi^{-1}(\o \lambda))$ is a curved dg $\Baru
  {\o A}$-comodule if and only if $\phi^{-1}(\o \lambda)^{2} = L_{-s^{-1}
    h}.$ Since $\phi^{-1}(\o \lambda)$ is a coderivation with respect to $\o
  d = \Phi^{-1}(\o \mu)$, $\phi^{-1}(\o \lambda)^{2}$ is a coderivation with
  respect to $\o d^{2} = \ad {-s^{-1}h},$ and one checks $L_{-s^{-1}h}$
  is also a coderivation with respect to $\ad {-s^{-1}h}$. Thus by
  \ref{rem:univ-props-cofree-comods}.(1), $\phi^{-1}(\o \lambda)^{2} + L_{s^{-1}
    h} = 0$ exactly when $\o \lambda\phi^{-1}(\o \lambda) + (\epsilon \otimes
  1)L_{s^{-1}h} = 0.$ The adjoint of $\o \lambda\phi^{-1}(\o \lambda) = \o \lambda(\o d
  \otimes 1 + (1 \otimes \o \lambda)(\Delta \otimes 1))$ is equal to the
  first two terms of the equation of
  \ref{lem:diagrams-for-s-ual-reps}.(1), while the adjoint of $(\epsilon \otimes
  1)L_{s^{-1}h} = (\Baru {\o A} \otimes M \xra{s^{-1}h \otimes 1} k
  \otimes M \cong M)$ is the third term  of
  \ref{lem:diagrams-for-s-ual-reps}.(1). Now by  \ref{rem:univ-props-cofree-comods}.(2), $\psi^{-1}(g)$ is a
  morphism of curved dg comodules if and only if $\o \lambda_{N}
  \psi^{-1}(g) - g\phi^{-1}(\o \lambda_{M}) = 0$. The first term is
  the adjoint of the first term of
  \ref{lem:diagrams-for-s-ual-reps}.(2), and the second term is the
  adjoint of the second and third terms of \ref{lem:diagrams-for-s-ual-reps}.(2).
\end{proof}

\begin{ex}
  Let $(A, 1, \mu)$ be the Koszul complex on $f \in k$ with curved bar
  construction $\Bar \o A = (k[T], f T^{*})$, see
  \ref{ex:curved-bar-constr-of-koszul-cx}, and let $(M, \o \lambda)$ be a
  strictly unital representation, described in
  \ref{ex:strictly-unital-modules-over-Koszul-complex}. Set $d_{M} =
  \phi^{-1}(\o \lambda): k[T] \otimes M \to k[T] \otimes M.$ For $x \in M,$
  $d_{M}(T^j \otimes x) = \sum_{k = 0}^{j} T^k \otimes
  \sigma^{j-k}(x),$ where $\sigma^{j-k}$ is the composition $M \xra{\cong}
  k[e]^{\otimes j - k} \otimes M \xra{\overbar \lambda^{j-k}} M.$ It follows
  from \ref{ex:strictly-unital-modules-over-Koszul-complex} that
  $d_M^{2}(T^j\otimes x) = -fT^{j-1} \otimes x$.
  
Dualizing gives a graded module over the polynomial ring $k[T^{*}]$
and a degree -1 map on $k[T^{*}] \otimes M^{*}$ whose square is
multiplication by $-fT^{*}.$ Sheafifying this, we get two $k$-modules
and maps,
$M^{\ev} \to M^{\odd} \to \Pi^{-1} M^{\ev},$ whose composition is
multiplication by $f \in k.$ This is exactly a matrix factorization in
the sense of Eisenbud \cite{Ei80}.
\end{ex}

\begin{cor}
Let $(A, 1, \o \mu + h + \msu)$ be an \Ai-algebra with split unit. There is a functor from the
category of strictly unital
representations of $A$ to the category of curved dg $\Baru {\o A}$ comodules, that sends $(M,
\o p_{M})$ to $(\Baru {\o A} \otimes M, \phi^{-1}(\o \lambda_{M}))$. This is fully faithful with
image the full subcategory of cofree curved dg comodules.
\end{cor}






\def\cprime{$'$}

\end{document}

%% file: definitions.tex
\usepackage{amscd}
\usepackage{amssymb, amsmath}

\usepackage{tikz}
\usetikzlibrary{shapes.multipart, positioning,
decorations.pathreplacing, decorations.markings,
decorations.pathmorphing, calc, matrix, cd}

\usepackage[all, knot]{xy}
\usepackage{hyperref}
 \usepackage[calc]{picture}
\usepackage{graphicx}
\usepackage{scalerel}

\theoremstyle{plain} 
\newtheorem{thm}{Theorem}[section]

\newtheorem*{introthm*}{Theorem}

\newtheorem{cor}[thm]{Corollary}
\newtheorem{lem}[thm]{Lemma}
\newtheorem{lemma}[thm]{Lemma}
\newtheorem{prop}[thm]{Proposition}

\theoremstyle{definition}
\newtheorem{defn}[thm]{Definition}

\newtheorem{ex}[thm]{Example}

\newlength{\thmtopspace}                
\newlength{\thmbotspace}                
\newlength{\thmheadspace}               
\newlength{\thmindent}                  
\theoremstyle{remark}
\newtheorem{rem}[thm]{Remark}

\newtheorem*{rems}{Remark}

\makeatletter
\newsavebox\myboxA
\newsavebox\myboxB
\newlength\mylenA
\newcommand*\xoverline[2][0.75]{%
    \sbox{\myboxA}{$\m@th#2$}%
    \setbox\myboxB\null
    \ht\myboxB=\ht\myboxA%
    \dp\myboxB=\dp\myboxA%
    \wd\myboxB=#1\wd\myboxA
    \sbox\myboxB{$\m@th\overline{\copy\myboxB}$}
    \setlength\mylenA{\the\wd\myboxA}
    \addtolength\mylenA{-\the\wd\myboxB}%
    \ifdim\wd\myboxB<\wd\myboxA%
       \rlap{\hskip 0.5\mylenA\usebox\myboxB}{\usebox\myboxA}%
    \else
        \hskip -0.5\mylenA\rlap{\usebox\myboxA}{\hskip 0.5\mylenA\usebox\myboxB}%
    \fi}

\newcommand*\xunderline[2][0.75]{%
    \sbox{\myboxA}{$\m@th#2$}%
    \setbox\myboxB\null
    \ht\myboxB=\ht\myboxA%
    \dp\myboxB=\dp\myboxA%
    \wd\myboxB=#1\wd\myboxA
    \sbox\myboxB{$\m@th\underline{\copy\myboxB}$}
    \setlength\mylenA{\the\wd\myboxA}
    \addtolength\mylenA{-\the\wd\myboxB}%
    \ifdim\wd\myboxB<\wd\myboxA%
       \rlap{\hskip 0.5\mylenA\usebox\myboxB}{\usebox\myboxA}%
    \else
        \hskip -0.5\mylenA\rlap{\usebox\myboxA}{\hskip 0.5\mylenA\usebox\myboxB}%
    \fi}
\makeatother

\newcommand{\overbar}[1]{\mkern 1.5mu\overline{\mkern-1.5mu#1\mkern-1.5mu}\mkern 1.5mu}
\renewcommand{\o}[1]{\xoverline{#1}}
\renewcommand{\u}[1]{\xunderline{#1}}

\makeatletter
\newcommand*\bigcdot{\mathpalette\bigcdot@{.5}}
\newcommand*\bigcdot@[2]{\mathbin{\vcenter{\hbox{\scalebox{#2}{$\m@th#1\bullet$}}}}}
\makeatother

\def\op{\text{op}}

\def\Spec{\operatorname{Spec}}

\newcommand{\odd}{\operatorname{odd}}
\newcommand{\ev}{\operatorname{ev}}

\newcommand{\bC}{\mathbb{C}}

\newcommand{\bQ}{\mathbb{Q}}

\newcommand{\bZ}{\mathbb{Z}}

\newcommand\bs[1]{\boldsymbol #1}

\newcommand{\ds}{\displaystyle}

\renewcommand{\t}[1]{\widetilde{#1}}

\newcommand{\xra}[1]{\xrightarrow{#1}}

\def\onto{\twoheadrightarrow}

\newcommand{\Set}{\text{Set}}

\newcommand{\Group}{\text{Group}}

\newcommand{\Sym}{\operatorname{Sym}}

\newcommand{\colim}{\operatornamewithlimits{colim}}
\newcommand{\ad}{\operatorname{ad}}

\newcommand{\Hom}[3]{\operatorname{Hom}_{#1}(#2,#3)}
\newcommand{\Homd}[4]{\operatorname{Hom}_{#1}(#3,#4)_{#2}}

\newcommand{\cobar}[1]{\operatorname{Cob}#1}

\renewcommand{\bar}[1]{\operatorname{Bar} #1}
\renewcommand{\Bar}[1]{\operatorname{Bar} #1}
\newcommand{\Baru}[1]{\operatorname{Bar}_{\operatorname{u}} #1}

\newcommand{\dhom}{\delta_{\operatorname{Hom}}}

\newcommand{\Ai}{$\text{A}_{\infty}$}

\newcommand{\hhc}[2]{CC^{\bullet} (#1, #2)}
\newcommand{\hhcu}[2]{CC_{\mathsf{u}}^{\bullet} (#1, #2)}
\newcommand{\hhn}[3]{CC^{#1} (#2, #3)}
\newcommand{\hhcn}[3][n]{CC^{#1} (#2, #3)}

\newcommand{\hhca}[3]{CC^{\bullet}_{#1} (#2, #3)}

\newcommand{\Tco}[1]{T_{\mathsf{co}}(#1)}

\newcommand{\Tcou}[1]{T_{\mathsf{co,u}}(#1)}

\newcommand{\Coder}[1][]{\operatorname{Coder}_{#1}}

\newcommand{\Aut}{\operatorname{Aut}}
\newcommand{\End}{\operatorname{End}}

\newcommand{\Aidef}[1]{\operatorname{Def}_{#1}}

\newcommand{\Aifam}[1]{\operatorname{Fam}_{#1}}
\newcommand{\infAidef}[1]{\operatorname{infDef}_{#1}}

\newcommand{\suinfAidef}[1]{\operatorname{infDef}_{#1}^{\,\operatorname{su}}}
\newcommand{\auginfAidef}[1]{\operatorname{infDef}_{#1}^{\,\operatorname{aug}}}

\newcommand{\MCFunct}[1]{\operatorname{MC}_{#1}}
\newcommand{\MC}{\operatorname{MC}}

\newcommand{\Algk}[1][k]{\operatorname{ComAlg}_{#1}}

\newcommand{\Coalgk}[1][k]{\operatorname{Coalg}_{#1}}
\newcommand{\AugAlgk}[1][k]{\operatorname{ComAlg}_{#1}^{\operatorname{aug}}}
\newcommand{\AugCoalgk}[1][k]{\operatorname{CocomCoalg}_{#1}^{\operatorname{aug}}}
\newcommand{\dgAugCoalgk}[1][k]{\operatorname{dgCocomCoalg}_{#1}^{\operatorname{aug}}}
\newcommand{\dgLieAlgk}[1][k]{\operatorname{dgLie}_{#1}}

\newcommand{\augdgAlgk}[1][k]{\operatorname{dgAlg}_{#1}^{\operatorname{aug}}}
\newcommand{\augaiAlgk}[1][k]{\text{A}_{\infty}\operatorname{Alg}_{#1}^{\operatorname{aug}}}
\newcommand{\suaiAlgk}[1][k]{\text{A}_{\infty}\operatorname{Alg}_{#1}^{\operatorname{su}}}
\newcommand{\coaugdgCoalgk}[1][k]{\operatorname{dgCoalg}_{#1}^{\operatorname{aug}}}
\newcommand{\curvdgCoalgk}[1][k]{\operatorname{curv-dgCoalg}_{#1}^{\operatorname{aug}}}
\newcommand{\augdgCoalgk}[1][k]{\operatorname{dgCoalg}_{#1}^{\operatorname{aug}}}
\newcommand{\finAlgk}[1][k]{\operatorname{finComAlg}_{#1}^{\operatorname{aug}}}
\newcommand{\finCoalgk}[1][k]{\operatorname{finCocomCoalg}_{#1}^{\operatorname{aug}}}

\renewcommand{\ad}[1]{\operatorname{ad}(#1)}
\newcommand{\msu}{\mu_{\operatorname{su}}}
\newcommand{\msul}[1]{\mu_{\operatorname{su}}^{#1}}
\newcommand{\gsu}{g_{\operatorname{su}}}
\newcommand{\muniv}{\bs \nu_{\operatorname{univ}}}
\newcommand{\luniv}{l_{\operatorname{u}}}
\newcommand{\muend}{\mu_{\operatorname{End}}}

\newcommand{\sqr}{\operatorname{sq}}

\usepackage{twoopt}
\newcommandtwoopt{\hsp}[2][\alpha][\beta]{\,{}_{#1}\!\!*_{#2}}

\newcommand{\Ho}{\operatorname{Ho}}


%% file: boxdiagrams.tex
\usepackage{tikz}
\usetikzlibrary{shapes.misc,shadings, decorations.pathmorphing}
\usepgfmodule{oo}

\makeatletter

\newcommand\newBoxDiagram [2][]{
  \pgfkeys{/box diagrams/.cd, make box = #2, #2/.cd, #1, /.cd}
  \edef \n@w{\noexpand \pgfoonew \csname#2\endcsname=new box
    diagram(#2)}
  \n@w
}

\newcommand\print [2][]{
  {\pgfkeys{/box diagrams/#2/.cd, #1}
   \def \n@w{\csname#2\endcsname} 
   \n@w.print()
  }
}

\newcommand\printNoArms [2][]{
  {\pgfkeys{/box diagrams/#2/.cd, #1}
   \def \n@w{\csname#2\endcsname} 
   \n@w.print box()
   \n@w.print output(0)
  }
}

\newcommand\printBox [2][]{
  {\pgfkeys{/box diagrams/#2/.cd, #1}
   \def \n@w{\csname#2\endcsname} 
   \n@w.print box()
  }
}

\newcommand\printArm[3][]{
  {\pgfkeys{/box diagrams/#2/.cd, #1}
   \def \n@w{\csname#2\endcsname} 
   \n@w.print arm(#3)
  }
}

\newcommand\printRightArm[3][]{
  {\pgfkeys{/box diagrams/#2/.cd, #1}
   \def \n@w{\csname#2\endcsname} 
   \n@w.print right arm(#3)
  }
}

\newcommand\printOutput[3][]{
  {\pgfkeys{/box diagrams/#2/.cd, #1}
   \def \n@w{\csname#2\endcsname} 
   \n@w.print output(#3)
  }
}

\newcommand\printLabelOnArm[4][]{
  {\pgfkeys{/box diagrams/#2/.cd, #1}
   \def \n@w{\csname#2\endcsname} 
   \n@w.print label above arm(#3, #4)
  }
}

\newcommand\printArmSep[2][]{
  {\pgfkeys{/box diagrams/#2/.cd, #1}
   \def \n@w{\csname#2\endcsname} 
   \n@w.print arm sep()
  }
}

\newcommand\printDec[3][]{
  {\pgfkeys{/box diagrams/#2/.cd, #1}
   \def \n@w{\csname#2\endcsname} 
   \n@w.decorate(#3)
  }
}

\newcommand\printComma[1]{
  \def \n@w{\csname#1\endcsname} 
   \n@w.print comma()}

\newcommand\printPeriod[1]{
  \def \n@w{\csname#1\endcsname} 
   \n@w.print period()}

\pgfkeys{
  /box diagrams/.cd,
  make box/.style = {
    #1/width/.initial = .5cm,
    #1/arm height/.initial = .75cm,
    #1/arm decoration/.initial = {},
    #1/arm style/.initial ={},
    #1/box height/.initial = .5cm,
    #1/output height/.initial = .5cm,
    #1/output color/.initial = black,
    #1/output style/.initial = thick,
    #1/output decoration/.initial = {},
    #1/output label/.initial = {},
    #1/output width/.initial = .3mm,
    #1/label/.initial = #1,
    #1/line style/.initial = thick,
    #1/color/.initial = black,
    #1/right arm color/.initial = black,
    #1/right arm style/.initial = thick,
    #1/right arm decoration/.initial = {},
    #1/decorate left/.initial = -1,
    #1/decorate right/.initial = 1,
    #1/decoration amplitude/.initial = 3pt,
    #1/decoration yshift label/.initial = .1cm,
    #1/decoration yshift brace/.initial = .05cm,
    #1/coord/.initial = {(0,0)},
    #1/arm sep left/.initial = -.6,
    #1/arm sep right/.initial = .6,
    #1/arm sep style/.initial = thin,
    #1/arm sep color/.initial = black,
    #1/arm dec amp/.initial = 2,
    #1/arm dec seg length/.initial = .2cm,
},}

\pgfooclass{box diagram}{
  \attribute width;
  \attribute arm height;
  \attribute arm decoration;
  \attribute arm style;
  \attribute box height;
  \attribute output height;
  \attribute output color;
  \attribute output style;
  \attribute output decoration;
  \attribute output label;
  \attribute output width;
  \attribute label;
  \attribute line style;
  \attribute color;
  \attribute right arm color;
  \attribute right arm style;
\attribute right arm decoration;
  \attribute decoration amplitude;
  \attribute decorate left;
   \attribute decorate right;
  \attribute decoration yshift label;
  \attribute decoration yshift brace;
  \attribute coord;
  \attribute arm sep left;
  \attribute arm sep right;
  \attribute arm sep style;
  \attribute arm sep color;
 \attribute arm dec amp;
 \attribute arm dec seg length;

  \method init(#1){ 
    \pgfoothis.set attributes to key values(#1)
  }

\method set attributes to key values(#1){
  \pgfooset{width}{\pgfkeysvalueof{/box
        diagrams/#1/width}}
    \pgfooset{arm height}{\pgfkeysvalueof{/box
        diagrams/#1/arm height}}
    \pgfooset{arm decoration}{\pgfkeysvalueof{/box
        diagrams/#1/arm decoration}}
    \pgfooset{arm style}{\pgfkeysvalueof{/box
        diagrams/#1/arm style}}
    \pgfooset{box height}{\pgfkeysvalueof{/box
        diagrams/#1/box height}}
    \pgfooset{output height}{\pgfkeysvalueof{/box
        diagrams/#1/output height}}
     \pgfooset{output color}{\pgfkeysvalueof{/box
        diagrams/#1/output color}}
     \pgfooset{output style}{\pgfkeysvalueof{/box
        diagrams/#1/output style}}
    \pgfooset{output decoration}{\pgfkeysvalueof{/box
        diagrams/#1/output decoration}}
    \pgfooset{output label}{\pgfkeysvalueof{/box
        diagrams/#1/output label}}
    \pgfooset{output width}{\pgfkeysvalueof{/box diagrams/#1/output width}}
    \pgfooset{label}{\pgfkeysvalueof{/box
        diagrams/#1/label}}
    \pgfooset{line style}{\pgfkeysvalueof{/box
        diagrams/#1/line style}}
        \pgfooset{color}{\pgfkeysvalueof{/box
        diagrams/#1/color}}
    \pgfooset{right arm color}{\pgfkeysvalueof{/box
        diagrams/#1/right arm color}}
    \pgfooset{right arm style}{\pgfkeysvalueof{/box
        diagrams/#1/right arm style}}
 \pgfooset{right arm decoration}{\pgfkeysvalueof{/box
        diagrams/#1/right arm decoration}}
    \pgfooset{decoration amplitude}{\pgfkeysvalueof{/box
        diagrams/#1/decoration amplitude}}
    \pgfooset{decorate left}{\pgfkeysvalueof{/box
        diagrams/#1/decorate left}}
    \pgfooset{decorate right}{\pgfkeysvalueof{/box
        diagrams/#1/decorate right}}
    \pgfooset{decoration yshift label}{\pgfkeysvalueof{/box
        diagrams/#1/decoration yshift label}}
    \pgfooset{decoration yshift brace}{\pgfkeysvalueof{/box
        diagrams/#1/decoration yshift brace}}
    \pgfooset{coord}{\pgfkeysvalueof{/box diagrams/#1/coord}}
    \pgfooset{arm sep left}{\pgfkeysvalueof{/box diagrams/#1/arm sep left}}
    \pgfooset{arm sep right}{\pgfkeysvalueof{/box diagrams/#1/arm sep right}}
    \pgfooset{arm sep style}{\pgfkeysvalueof{/box diagrams/#1/arm sep style}}
    \pgfooset{arm sep color}{\pgfkeysvalueof{/box diagrams/#1/arm sep color}}
    \pgfooset{arm dec amp}{\pgfkeysvalueof{/box diagrams/#1/arm dec amp}} 
\pgfooset{arm dec seg length}{\pgfkeysvalueof{/box diagrams/#1/arm dec seg length}}
}
  
  \method get(#1,#2) {\pgfooget{#1}{#2}}
  \method set(#1,#2) {\pgfooset{#1}{#2}}
  \method let(#1,#2) {\pgfoolet{#1}#2}

  \method coord at top(#1,#2){
     \node[shape=coordinate] (st@rt) at \pgfoovalueof{coord} {};
    \node[shape=coordinate] (#1) at
    ($(st@rt)+(#2*\pgfoovalueof{width},
    \pgfoovalueof{box height} + \pgfoovalueof{output height})$){};}
  
  \method coord at bottom(#1,#2){
     \node[shape=coordinate] (st@rt) at \pgfoovalueof{coord} {};
    \node[shape=coordinate] (#1) at
    ($(st@rt)+(#2*\pgfoovalueof{width},0)$){};}

    \method print() { 
    \pgfoothis.print box();
    \pgfoothis.print arm(-1);
      \pgfoothis.print right arm(1);
\pgfoothis.print output(0);
  }

\method print box() {
\node[shape=coordinate] (st@rt) at \pgfoovalueof{coord} {};
    \begin{scope}[\pgfoovalueof{line style}, \pgfoovalueof{color}]     
    \ifdim\pgfoovalueof{box height}=0cm {} 
    \else 
    \node[minimum width = 2*\pgfoovalueof{width}, minimum height =
    \pgfoovalueof{box height}, draw,fill=black!20] at ($(st@rt) + (0,
    \pgfoovalueof{output height} + \pgfoovalueof{box height}/2)$)
    {\small{$\pgfoovalueof{label}$}};
    \fi
  \end{scope}
}

  \method print arm(#1) {
    \begin{scope}[\pgfoovalueof{line style},color=\pgfoovalueof{color}]
      \pgfoothis.coord at top(cor, #1);
    \draw[decorate, decoration = {\pgfoovalueof{arm decoration}, amplitude = \pgfoovalueof{arm dec amp}}, segment length = \pgfoovalueof{arm dec seg length},\pgfoovalueof{arm style}] (cor) -- ($(cor) + (0,\pgfoovalueof{arm height})$);
    \end{scope}
  }

  \method print right arm(#1) {
    \begin{scope}[decorate, decoration = \pgfoovalueof{right arm decoration}, color =\pgfoovalueof{right arm color},\pgfoovalueof{right arm style}]
    \node[shape=coordinate] (st@rt) at \pgfoovalueof{coord} {};
    \draw[decorate, decoration = {\pgfoovalueof{arm decoration},amplitude = \pgfoovalueof{arm dec amp}}, \pgfoovalueof{arm style}]
      ($(st@rt) + (#1*\pgfoovalueof{width}, \pgfoovalueof{box height} + \pgfoovalueof{output height})$) --
      ($(st@rt) + (#1*\pgfoovalueof{width}, \pgfoovalueof{box height} + \pgfoovalueof{output height} +
      \pgfoovalueof{arm height})$);
    \end{scope}
   }

   \method print output(#1){
\pgfoothis.coord at bottom(st@rt, #1);
    \begin{scope}[\pgfoovalueof{line style}, \pgfoovalueof{color}, decoration = {\pgfoovalueof{arm decoration}, amplitude = \pgfoovalueof{arm dec amp}}, segment length = \pgfoovalueof{arm dec seg length},\pgfoovalueof{arm style}, line width = {\pgfoovalueof{output width}},]
      \draw[decorate, decoration = \pgfoovalueof{output decoration}, color = \pgfoovalueof{output color},\pgfoovalueof{output style}] (st@rt) -- ++(0, \pgfoovalueof{output height}) node[midway, right]{\tiny{$\pgfoovalueof{output label}$}};
      \end{scope}
}

  \method print label above arm(#1,#2) {
    \node[shape=coordinate] (st@rt) at \pgfoovalueof{coord} {};
    \node[above] at ($(st@rt) + (#1*\pgfoovalueof{width}, \pgfoovalueof{box height} + \pgfoovalueof{output height} +
      \pgfoovalueof{arm height})$) {\tiny{$#2$}};
   }

  \method print arm sep(){
    \node[shape=coordinate] (st@rt) at \pgfoovalueof{coord} {};
    \draw[\pgfoovalueof{arm sep style}, color = \pgfoovalueof{arm sep color}]
    ($(st@rt)+(\pgfoovalueof{arm sep left}*\pgfoovalueof{width}, \pgfoovalueof{box height} + \pgfoovalueof{output
      height} + 1/2*\pgfoovalueof{arm height})$) --
    ($(st@rt)+(\pgfoovalueof{arm sep right}*\pgfoovalueof{width}, \pgfoovalueof{box height} + \pgfoovalueof{output
      height}+ 1/2*\pgfoovalueof{arm height})$);
  }

  \method decorate(#1){
    \begin{scope}[color = \pgfoovalueof{color}]
      \node[shape=coordinate] (st@rt) at \pgfoovalueof{coord} {};
      \node[shape=coordinate] (upperl@ft) at
      ($(st@rt)+(0,\pgfoovalueof{decoration yshift
        brace})+(\pgfoovalueof{decorate left}*\pgfoovalueof{width},
      \pgfoovalueof{arm height}+\pgfoovalueof{box height} +
      \pgfoovalueof{output height})$) {}; \node[shape=coordinate]
      (upperr@ght) at
      ($(st@rt)+(0,\pgfoovalueof{decoration yshift
        brace})+(\pgfoovalueof{decorate right}*\pgfoovalueof{width},
      \pgfoovalueof{arm height}+\pgfoovalueof{box height} +
      \pgfoovalueof{output height})$) {};
      \draw[decorate,decoration={brace,amplitude=\pgfoovalueof{decoration
          amplitude}}] (upperl@ft) -- (upperr@ght) node[midway, above,
      yshift = \pgfoovalueof{decoration yshift label}] {\tiny{$#1$}};
    \end{scope}
}

\method print period(){
  \def\smallxshift{.1cm}
  \def\smallyshift{.05cm}
      \node[shape=coordinate] (st@rt) at \pgfoovalueof{coord} {};
      \node at ($(st@rt) +
      (\pgfoovalueof{width}+\smallxshift,\smallyshift)$) {.};}

    \method print comma(){
      \def\smallxshift{.1cm}
      \def\smallyshift{.05cm}
      \node[shape=coordinate] (st@rt) at \pgfoovalueof{coord} {};
      \node at ($(st@rt) + (\pgfoovalueof{width}+\smallxshift,\smallyshift)$) {,};}
}

\makeatother


%% file: curvature_and_augmentation13__arxiv_submission_.bbl
\begin{thebibliography}{{T}oe14}

\bibitem[Ada78]{MR505692}
John~Frank Adams.
\newblock {\em Infinite loop spaces}, volume~90 of {\em Annals of Mathematics
  Studies}.
\newblock Princeton University Press, Princeton, N.J.; University of Tokyo
  Press, Tokyo, 1978.

\bibitem[Avr98]{MR1648664}
Luchezar~L. Avramov.
\newblock Infinite free resolutions.
\newblock In {\em Six lectures on commutative algebra ({B}ellaterra, 1996)},
  volume 166 of {\em Progr. Math.}, pages 1--118. Birkh\"auser, Basel, 1998.

\bibitem[Bur15]{1508.03782}
Jesse Burke.
\newblock Higher homotopies and {G}olod rings, 2015, arXiv:1508.03782.

\bibitem[CLM16]{MR3597150}
Joseph Chuang, Andrey Lazarev, and Wajid Mannan.
\newblock Koszul-{M}orita duality.
\newblock {\em J. Noncommut. Geom.}, 10(4):1541--1557, 2016.

\bibitem[Dem72]{MR0344261}
Michel Demazure.
\newblock {\em Lectures on {$p$}-divisible groups}.
\newblock Lecture Notes in Mathematics, Vol. 302. Springer-Verlag, Berlin-New
  York, 1972.

\bibitem[Dri14]{MR3285856}
V.~Drinfeld.
\newblock A letter from {K}harkov to {M}oscow.
\newblock {\em EMS Surv. Math. Sci.}, 1(2):241--248, 2014.
\newblock Translated from the Russian by Keith Conrad.

\bibitem[Eis80]{Ei80}
David Eisenbud.
\newblock Homological algebra on a complete intersection, with an application
  to group representations.
\newblock {\em Trans. Amer. Math. Soc.}, 260(1):35--64, 1980.

\bibitem[FK16]{DerKosDuality}
Hu~Po Foster, \mbox{Tyler} and Igor Kriz.
\newblock $d$-structures and derived {K}oszul duality for unital operad
  algebras.
\newblock {\em J. Pure Appl. Algebra}, 220(3):1133--1156, 2016.

\bibitem[FP02]{MR1935035}
Alice Fialowski and Michael Penkava.
\newblock Deformation theory of infinity algebras.
\newblock {\em J. Algebra}, 255(1):59--88, 2002.

\bibitem[Ger63]{MR0161898}
Murray Gerstenhaber.
\newblock The cohomology structure of an associative ring.
\newblock {\em Ann. of Math. (2)}, 78:267--288, 1963.

\bibitem[Get93]{MR1261901}
Ezra Getzler.
\newblock Cartan homotopy formulas and the {G}auss-{M}anin connection in cyclic
  homology.
\newblock In {\em Quantum deformations of algebras and their representations
  ({R}amat-{G}an, 1991/1992; {R}ehovot, 1991/1992)}, volume~7 of {\em Israel
  Math. Conf. Proc.}, pages 65--78. 1993.

\bibitem[GM88]{MR972343}
William~M. Goldman and John~J. Millson.
\newblock The deformation theory of representations of fundamental groups of
  compact {K}\"ahler manifolds.
\newblock {\em Inst. Hautes \'Etudes Sci. Publ. Math.}, (67):43--96, 1988.

\bibitem[Gri16]{1612.02254}
Brice~Le Grignou.
\newblock Homotopy theory of unital algebras, 2016, arXiv:1612.02254.

\bibitem[Gro95]{MR1603480}
Alexander Grothendieck.
\newblock Technique de descente et th\'eor\`emes d'existence en g\'eom\'etrie
  alg\'ebrique. {II}. {L}e th\'eor\`eme d'existence en th\'eorie formelle des
  modules.
\newblock In {\em S\'eminaire {B}ourbaki, {V}ol.\ 5}, pages Exp.\ No.\ 195,
  369--390. Soc. Math. France, Paris, 1995.

\bibitem[GS88]{MR981619}
Murray Gerstenhaber and Samuel~D. Schack.
\newblock Algebraic cohomology and deformation theory.
\newblock In {\em Deformation theory of algebras and structures and
  applications ({I}l {C}iocco, 1986)}, volume 247 of {\em NATO Adv. Sci. Inst.
  Ser. C Math. Phys. Sci.}, pages 11--264. Kluwer Acad. Publ., Dordrecht, 1988.

\bibitem[Hin01]{HinichJPAA01}
Vladimir Hinich.
\newblock D{G} coalgebras as formal stacks.
\newblock {\em J. Pure Appl. Algebra}, 162(2-3):209--250, 2001.

\bibitem[Hin03]{MR1978336}
Vladimir Hinich.
\newblock Tamarkin's proof of {K}ontsevich formality theorem.
\newblock {\em Forum Math.}, 15(4):591--614, 2003.

\bibitem[HM12]{MR2993002}
Joseph Hirsh and Joan Mill\`es.
\newblock Curved {K}oszul duality theory.
\newblock {\em Math. Ann.}, 354(4):1465--1520, 2012.

\bibitem[Kel]{KellerDefTheory}
Bernhard Keller.
\newblock Notes for an {I}ntroduction to {K}ontsevich's quantization theorem.
\newblock Preprint, available at
  \url{https://webusers.imj-prg.fr/~bernhard.keller/publ/emalca.pdf}.

\bibitem[KS]{KSDefTheory}
Maxim Kontsevich and Yan Soibelman.
\newblock Deformation theory.
\newblock Unpublished manuscript, available at
  \url{https://www.math.ksu.edu/~soibel/Book-vol1.ps}.

\bibitem[KS09]{MR2596638}
M.~Kontsevich and Y.~Soibelman.
\newblock Notes on {$A_\infty$}-algebras, {$A_\infty$}-categories and
  non-commutative geometry.
\newblock In {\em Homological mirror symmetry}, volume 757 of {\em Lecture
  Notes in Phys.}, pages 153--219. Springer, Berlin, 2009.

\bibitem[Laz03]{MR1989615}
A.~Lazarev.
\newblock Hoschschild cohomology and moduli spaces of strongly homotopy
  associative algebras.
\newblock {\em Homology Homotopy Appl.}, 5(1):73--100, 2003.

\bibitem[Lur10]{MR2827833}
Jacob Lurie.
\newblock Moduli problems for ring spectra.
\newblock In {\em Proceedings of the {I}nternational {C}ongress of
  {M}athematicians. {V}olume {II}}, pages 1099--1125. Hindustan Book Agency,
  New Delhi, 2010.

\bibitem[LV12]{MR2954392}
Jean-Louis Loday and Bruno Vallette.
\newblock {\em Algebraic operads}, volume 346 of {\em Grundlehren der
  Mathematischen Wissenschaften}.
\newblock Springer, Heidelberg, 2012.

\bibitem[Lyu11]{MR2769322}
Volodymyr Lyubashenko.
\newblock Homotopy unital {$A_\infty$}-algebras.
\newblock {\em J. Algebra}, 329:190--212, 2011.

\bibitem[Lyu14]{1403.3644}
Volodymyr Lyubashenko.
\newblock Curved cooperads and homotopy unital {A}-infty-algebras, 2014,
  arXiv:1403.3644.

\bibitem[Mon93]{MR1243637}
Susan Montgomery.
\newblock {\em Hopf algebras and their actions on rings}, volume~82 of {\em
  CBMS Regional Conference Series in Mathematics}.
\newblock Published for the Conference Board of the Mathematical Sciences,
  Washington, DC; by the American Mathematical Society, Providence, RI, 1993.

\bibitem[MT14]{MR3176644}
Fernando Muro and Andrew Tonks.
\newblock Unital associahedra.
\newblock {\em Forum Math.}, 26(2):593--620, 2014.

\bibitem[NR66]{MR0195995}
Albert Nijenhuis and R.~W. Richardson, Jr.
\newblock Cohomology and deformations in graded {L}ie algebras.
\newblock {\em Bull. Amer. Math. Soc.}, 72:1--29, 1966.

\bibitem[Pos93]{MR1250981}
L.~E. Positsel{\cprime}ski{\u\i}.
\newblock Nonhomogeneous quadratic duality and curvature.
\newblock {\em Funktsional. Anal. i Prilozhen.}, 27(3):57--66, 96, 1993.

\bibitem[Pos11]{MR2830562}
Leonid Positselski.
\newblock Two kinds of derived categories, {K}oszul duality, and
  comodule-contramodule correspondence.
\newblock {\em Mem. Amer. Math. Soc.}, 212(996):vi+133, 2011.

\bibitem[Pro11]{MR2844537}
Alain Prout{\'e}.
\newblock {$A_\infty$}-structures. {M}od\`eles minimaux de {B}aues-{L}emaire et
  {K}adeishvili et homologie des fibrations.
\newblock {\em Repr. Theory Appl. Categ.}, (21):1--99, 2011.
\newblock Reprint of the 1986 original.

\bibitem[PS95]{MR1364455}
Michael Penkava and Albert Schwarz.
\newblock {$A_\infty$} algebras and the cohomology of moduli spaces.
\newblock In {\em Lie groups and {L}ie algebras: {E}. {B}. {D}ynkin's
  {S}eminar}, volume 169 of {\em Amer. Math. Soc. Transl. Ser. 2}, pages
  91--107. Amer. Math. Soc., Providence, RI, 1995.

\bibitem[Qui69]{MR0258031}
Daniel Quillen.
\newblock Rational homotopy theory.
\newblock {\em Ann. of Math. (2)}, 90:205--295, 1969.

\bibitem[Sch98]{math/9802006}
Vadim Schechtman.
\newblock Remarks on formal deformations and {B}atalin-{V}ilkovisky algebras,
  1998, arXiv:math/9802006.

\bibitem[Sha69]{MR0241411}
Jack Shamash.
\newblock The {P}oincar\'e series of a local ring.
\newblock {\em J. Algebra}, 12:453--470, 1969.

\bibitem[SS85]{SchlStash}
Michael Schlessinger and James Stasheff.
\newblock The {L}ie algebra structure of tangent cohomology and deformation
  theory.
\newblock {\em J. Pure Appl. Algebra}, 38:313--322, 1985.

\bibitem[Sta63a]{MR0158400}
James Stasheff.
\newblock Homotopy associativity of {$H$}-spaces. {I}.
\newblock {\em Trans. Amer. Math. Soc.}, 108:275--292, 1963.

\bibitem[Sta63b]{MR0158400part2}
James Stasheff.
\newblock Homotopy associativity of {$H$}-spaces. {II}.
\newblock {\em Trans. Amer. Math. Soc.}, 108:293--312, 1963.

\bibitem[Sta93]{MR1239562}
Jim Stasheff.
\newblock The intrinsic bracket on the deformation complex of an associative
  algebra.
\newblock {\em J. Pure Appl. Algebra}, 89(1-2):231--235, 1993.

\bibitem[Swe69]{MR0252485}
Moss~E. Sweedler.
\newblock {\em Hopf algebras}.
\newblock Mathematics Lecture Note Series. W. A. Benjamin, Inc., New York,
  1969.

\bibitem[{T}oe14]{1401.1044}
Bertrand {T}oen.
\newblock Derived algebraic geometry, 2014.
\newblock arXiv:1401.1044.

\end{thebibliography}
